\documentclass[11pt,a4paper]{article}
\usepackage{empheq}
\usepackage{amssymb,amsbsy,amsmath,amsfonts,amssymb,amscd,amsthm, mathrsfs, bbm}
\usepackage{amssymb}
\usepackage{color}
\newcommand{\isEquivTo}[1]{\underset{#1}{\sim}}

\usepackage{hyperref}

\newcommand{\1}{\mathbbm{1}}
\newcommand{\CC}{\mathbb{C}}
\newcommand{\NN}{\mathbb{N}}
\newcommand{\RR}{\mathbb{R}}

\newcommand{\ZZ}{\mathbb{Z}}
\newtheorem{theo}{Theorem}
\newtheorem{prop}[theo]{Proposition}
\newtheorem{lem}[theo]{Lemma}
\newtheorem{cor}[theo]{Corollary}
\newtheorem{rem}[theo]{Remark}

\newcommand{\beqn}{\begin{equation}}
\newcommand{\eeqn}{\end{equation}}
\newcommand{\bear}{\begin{eqnarray}}
\newcommand{\eear}{\end{eqnarray}}
\newcommand{\bean}{\begin{eqnarray*}}
\newcommand{\eean}{\end{eqnarray*}}

\date{}
\begin{document} 
\centerline{\huge Classical approximation of a linearized}
\vskip 0.3cm
\centerline{ \huge three waves  kinetic equation.}
\vskip 0.5cm
\begin{center}
{\large Miguel  Escobedo}\\
\null
{\small Departamento de Matem\'aticas,} \\
{\small Universidad del
Pa{\'\i}s Vasco,} \\
{\small Apartado 644, E--48080 Bilbao, Spain.}\\
{\small {\tt miguel.escobedo@ehu.es}}
\end{center}

\noindent
{\bf Abstract}: The fundamental solution of the classical approximation of a  three waves kinetic equation that happens in the kinetic theory of a  condensed  gas of bosons near the critical temperature is obtained. It is also  proved to be unique in a suitable space of distributions, and several of its  properties  are  described. The fundamental solution is used  to solve the initial value problem for a general class of initial data.



\section{Introduction}
\setcounter{equation}{0}
\setcounter{theo}{0}

Our purpose in this article is to study the classical approximation of the linearized version of a three wave kinetic equation, around one of its equilibria,
\begin{align}
&\frac {\partial u} {\partial \tau }(\tau , x)=\int _0^\infty (u(\tau , y)-u(\tau , x)) K(x, y) dy,\,\,\tau >0,\,x>0 \label{S3EMK2}\\
&K(x, y)=\left(\frac {1} {|x^2-y^2|}-\frac {1} {x^2+y^2} \right)\frac {y} {x},\,\,\forall x>0,\,\,\forall y>0,\,\,x\not =y. \label{S3E2359B}
\end{align}

In a condensed Bose gas,   correlations arise between the superfluid component and the normal fluid part corresponding to the excitations. This causes number-changing processes and, as a consequence, in the hydrodynamic regime, a collision integral $C _{ 1, 2 }$  describing  the  splitting of an excitation into two others in the presence of the condensate is needed. 

A  kinetic equation  which includes these processes in a uniform Bose gas  was first deduced in a series of papers by  Kirkpatrick and Dorfman \cite{Kirkpatrick}.
More recently, Zaremba \& al. extended the treatment to a trapped Bose gas by including Hartree--Fock corrections to the energy of the excitations, and derived coupled kinetic equations for the distribution functions of the normal and superfluid components, sometimes called ZNG system (\cite{Za}).

In order  to simplify the notations it may be assumed without any loss of generality, that the mass of the particles  and the so called  interaction coupling constant are equal to one. Then, under the conditions of spatial homogeneity and isotropy, in the limit of  temperature below but  close to the critical temperature, the following system of equations  was  first deduced in \cite{Eckern} and  \cite{Kirkpatrick},
\begin{empheq}[left=\empheqlbrace]{align}
&\frac {\partial n} {\partial t}(t, p)=C _{ 1, 2 }(n_c(t), n(t))(p)\qquad t>0,\; p\in \RR^3, \label{PA}\\
&n_c'(t)=-\int _{ \RR^3 }C _{ 1, 2 }(n_c(t), n(t))(p)dp \qquad t>0,\label{PB}
\end{empheq}
where $C _{ 1, 2 }(n_c, n)$ is  the three waves collision integral,
\begin{align}
C _{ 1, 2 }(n_c(t), n(t))=&\,\,n_c(t) I_3(n(t))(p)\label{E1C12}\\
I_3(n(t))(p)=&\!\!\iint _{(\RR^3)^2}\!\!\big[R(p, p_1, p_2)\!-\!R(p_1, p, p_2)\!-\!R(p_2, p_1, p) \big]dp_1dp_2, \label{E1BCD} 
\end{align}
\begin{align}
R(p, p_1, p_2)\,=&\left[\delta (|p|^2-|p_1|^2-|p_2|^2)  \delta (p-p_1-p_2)\right]\times \nonumber \\
&\qquad \times \left[ n_1n_2(1+n)-(1+n_1)(1+n_2)n \right].  \label{S1EA4BC}
\end{align}
In these notations $n_\ell=n(t, p _{ \ell })$,  $n(t, p)$ denotes the density of particles in the normal gas that at time $t>0$ have momentum $p$ and $n_c(t)$ the density of the condensate at time $t$.

It is well known that the equation (\ref{PA})  has  a  family of non trivial equilibria $n_0$,
\begin{align}
&n_0(p)=\nu _{0}(|p|^2) \label{S1EEq1}\\
&\nu_0(\omega )=\left(e^{\beta \omega}-1\right)^{-1},\,\, \forall \omega  >0.\label{S1EEq21}
\end{align}
The parameter $\beta $ may be any positive constant and is related to  the temperature $T>0$  of the gas at  equilibrium $n_0$ through the formula,
$ \beta =1/(k_B T)$ where $k_B$ is the Boltzmann's constant. It is easily checked that  
 $R(p, p_k, p_\ell)\equiv 0$ in  (\ref{S1EA4BC})  for $n=n_0$.
 
It was proved in \cite{CE}  that for all constants $\rho >0$ and all non negative measures $n _{ in }$ with a finite first  moment,  the system (\ref{PA})--(\ref{S1EA4BC}) has a weak solution  $(n(t), n_c(t))$ with initial data  $(n _{ in }, \rho )$. For all $t>0$,  $n(t)$ is  a non negative measure with finite first moment that does not charge the origin,  and $n_c(t)>0$. System (\ref{PA})-(\ref{S1EA4BC}) was also  treated in \cite{AN, ST}.

One basic aspect of the non equilibrium behavior of the system condensate--normal fluid  is the growth of the condensate after its formation (cf. \cite{Bi}).  Although the relation of $n_c$  with the condensate amplitude is not straightforward, it seems nevertheless very closely related to the  total number of particles of the system having an energy less than some arbitrarily small, but fixed, value (cf. \cite{J}).
It turns out that the evolution of $n_c(t)$ crucially depends  on the behavior of $n(t, p)$ as $|p|\to 0$ (cf.  for example Proposition 2 in \cite{S} and  Theorem 1.7 in  \cite{CE}); when  the measure $n(t)$  is written as $n(t, p)=|p|^{-1}g(t, |p|^2)$, if $g(t)$ has no atomic part and has an algebraic behavior as $|p|\to 0$ then,
\begin{equation}
\label{S1EASN}
n(t, p) \isEquivTo{|p|\to 0} a(t)|p|^{-2}.
\end{equation} 
for some function $a(t)$, (cf. \cite{CE}).  However, these properties  have not been proved to hold  for general  solutions $n$  of the system (\ref{PA})--(\ref{S1EA4BC}).

\subsection{Small perturbation of a Planck distribution.}
\label{S1smallPert}
In order to prove the existence of solutions to (\ref{PA})-(\ref{S1EA4BC})  satisfying (\ref{S1EASN}), we adopt a very classical strategy. First linearize the equation (\ref{PA}) around an equilibrium $n_0$ and derive precise estimates on the solutions of the resulting equation.  Then, use the properties of the  linearized equation to solve (\ref{PA})-(\ref{S1EA4BC}) for suitable initial data. 

The linearized equation was  essentially obtained  in \cite{EPV} where it was  seen that some care  in the linearization  procedure is necessary.  A new dependent variable  $\Omega $ is defined as,
\begin{align}
n(t, p)&=n_0(p)+n_0(p)(1+n_0(p))\Omega  (t, |p|)=n_0(p)+\frac {\Omega (t, |p|)} {4 \, \sinh^2 \left(\frac{\beta |p|^2}{2}\right)}.\label{S2Elinzn2}
\end{align}
Under the change of variables
\begin{align}
\label{S2CcV}
&x=\frac {\sqrt \beta } {2}|p|,\,\,\tau =\int _0^t \frac {mc_0(s) \pi } {3}\left(\frac {2} {\beta } \right)^{\frac {3} {2}} ds, \,\,\,u(\tau , x)=\frac {\Omega (t, |p|)} {|p|^2},
\end{align}
the linearized equation for $u$ reads (cf. \cite{EPV} and the Appendix)
\begin{align}
&\frac {\partial u} {\partial \tau }=\int _0^\infty (u(\tau , y)-u(\tau , x)) M(x, y) dy\label{S3E23590}\\
&M(x, y)= \left(\frac {1} {\sinh|x^2-y^2|}-\frac {1} {\sinh (x^2+y^2) } \right) \frac {y^3\sinh x^2} {x^3\sinh y^2}, \label{S3E2359M}
\end{align}
To solve the Cauchy problem for equation (\ref{S3E23590}), where the term $(u(y)-u(x))/\sinh|x^2-y^2|$ introduces a differential operator, is not  straightforward  yet,  and we have  adopted what seems to be  the  shortest way to reach that purpose. It is not based  in a detailed study of the spectral properties of the operator at the right hand side of (\ref{S3E23590}), as it is done  in \cite{CIP} for the linearized homogeneous Boltzmann equation for classical particles.
We rather consider in a first step,  and that is the content of this work,  a simplified version of  the equation  (\ref{S3E23590}),  that preserves some its mean features. It is obtained by only keeping in  $M$ the leading terms of the hyperbolic sine functions for small values of their arguments. This reminds the classical limit, where the particle's energy $\hbar p$ is sent to zero and gives the equation (\ref{S3EMK2}), (\ref{S3E2359B}), that we recall here,
\begin{align*}
&\frac {\partial u} {\partial \tau }=\int _0^\infty (u(\tau , y)-u(\tau , x)) K(x, y) dy=: L(u(t)) \\
&K(x, y)=\left(\frac {1} {|x^2-y^2|}-\frac {1} {x^2+y^2} \right)\frac {y} {x}.
\end{align*}
The Cauchy problem for  equation (\ref{S3E2359M}) is solved in a forthcoming article, using  the properties of equation (\ref{S3EMK2}), (\ref{S3E2359B}) obtained here, seeing (\ref{S3E2359M}) as a perturbation of  (\ref{S3E2359B}).

For $u$ a regular function, equation  (\ref{S3EMK2}) may be written as (cf. (\ref{S4EHL007}) in the Appendix),
\begin{align}
&\frac {\partial u } {\partial \tau }=\int _0^\infty H\left( \frac {x} {y}\right)\frac {\partial u } {\partial y}(\tau , y) \frac {dy} {y} \label{S4E1}\\
H (r)&=\1 _{ 0<r<1 }\frac {1} {r}\log\left(\frac {1+r^2} {1-r^2} \right)+\1 _{ r>1 }\frac {1} {r}\log\left(1-\frac {1} {r^4} \right).\label{S4EHL}
\end{align}
Equation  (\ref{S4E1}) may be solved using the Mellin transform and classical arguments.

 Similar questions were considered in  \cite{EV}, \cite{EV2}  for  ``post gelation'' solutions of a coagulation equation. Some of the technical results in the last Section of  \cite{EV} will be of some use in this work. The equation (\ref{PA}) may actually be written as a  coagulation-fragmentation equation, with nonlinear fragmentation, in terms of the energy $\omega =|p|^2$ as independent variable  for a new dependent variable $g$ defined as $|p|n(t, p)=g(t,\omega )$ (cf  \cite{EV3}).

{\bf Remark.}{\it
The same  linear equation (\ref{S3EMK2})  follows if, first only the quadratic terms are kept in (\ref{E1BCD}), (\ref{S1EA4BC}), and then linearization is performed around the equilibrium $\omega^{-1} (p)=|p|^{-2}$. The first step  yields a wave turbulence type equation,  considered by several authors \cite{D, EV3, AV},
and  (\ref{S4E1}) is  the linearization of that equation  around the equilibrium $\omega^{-1} (p)$.}

\subsection{Main results.}
The  fundamental solution  of (\ref{S3EMK2}) is  obtained as a weak solution of (\ref{S4E1}) in the sense of distributions on $(0, \infty)$, and  is proved afterwards to satisfy (\ref{S3EMK2}). 

 The use of the Mellin transform makes the spaces $E' _{ p,\,q }$ for $p<q$, presented for example in Chapter 11 of \cite{ML}, very suitable. They are  defined as the dual of the spaces $E _{ p,\,q }$ of all the functions $\phi \in \mathscr{C}^\infty (0, \infty)$ such that:
\begin{eqnarray*}
&&N _{ p, q, k }(\phi )=\sup _{ x>0 }\left(k _{ p, q }(x)x^{k+1}\left|\phi ^{k}(x) \right|\right)<\infty\\
&&k _{ p, q }(x)=
\left\{
\begin{split}
&x^{-p},\,\,\hbox{if}\,\,0<x\le 1\\
&x^{-q},\,\,\hbox{if}\,\,x> 1
\end{split}
\right.
\end{eqnarray*}
with the topology defined by the numerable set of seminorms $\left\{N _{ p, q, k } \right\} _{ k\in \NN }$. It follows that $E' _{ p,\,q }$ is a subspace of $\mathscr{D}'([0, \infty))$.  As indicated  in \cite{ML}, these are the spaces of Mellin transformable distributions. Let us denote, for $p\in \RR, q\in \RR$, $p<q$,
\begin{equation}
\mathcal S _{ p, q }=\{s\in \CC; \mathscr Re s\in (p, q)\}.
\end{equation}

\begin{theo}
\label{TheoMain1}
There exists a unique function $\Lambda \in C(0,\infty); L^1(0, \infty)$ weak solution of (\ref{S4E1})) in $\mathscr D'((0, \infty)\times (0, \infty))$,  such that for all $T>0$, $\Lambda(t)\in  E' _{ 0, 2 }$ and $\mathscr M(\Lambda(t))$ is bounded on $\mathcal S _{ 0, 2 }$ for all $t\in (0, T)$.  That function is such that
\begin{align}
&(\log x)   \Lambda\in C((0 , \infty)\times  [0, \infty))\label{S5CInvME0300}\\
&(\log x)  \frac {\partial ^m \Lambda} {\partial  t^m }\in C((0 , \infty)\times  (0, \infty))\,\,\,\forall m\in \NN\setminus \{0\}, \label{S5CInvME03}\\
&(\log x)^2\frac  {\partial^{1+m} \Lambda} {\partial t^m\, \partial x}\in C((0, \infty)\times (0, \infty)), \, \forall m\in \NN,\label{S5CInvM2E1}\\
&\forall k\in \NN,\,\, \Lambda \in C^m\left(\left(\frac {k+1} {2}, \infty\right); C^k( 0, \infty)\right), \forall m\in \NN. \label{S5CInvME3B}\\
& \forall r\in (0, 1),\,\forall \alpha \in [0, r);\,\, \frac {\log x} {(x-1)^{\alpha} }\,\Lambda \in C \left(\left( \frac {r} {2}, \frac {1} {2}\right) ; H _{ {\rm loc} }^{r-\alpha } (0, \infty)\right), \label{S5CAlph1}\\
&\hbox{and,}\,\,\,\,\forall t\in (r/2, 1/2),\,\forall r' \in (r, 2), \exists C _{ r'} >0; \nonumber \\
&\qquad \quad  \,\left|\frac {(\log x)\Lambda(t, x)} {(x-1)^{\alpha} }-\frac {(\log y )\Lambda(t, y)  } {(y-1)^{\alpha} } \right|\le \frac {C _{ r' }|x-y|^{r-\alpha }} {x^{r'}}\label{S5CAlph10}
\end{align}

and satisfies   (\ref{S3EMK2}) for all $t>0, x>0$, $x\not =1$. The function $\Lambda$ also satisfies,
\begin{align}
&\lim _{ t\to 0 }\Lambda (t)=\delta _1,\,\,\hbox{weakly in}\,\,\mathscr D'((0, \infty), \label{S5DataL}\\
&\lim _{ t\to 0 }\, t^{-1}\left|e^{-1/t} Y\right|^{1-2t}\Lambda \left(t, 1+e^{-1/t}Y\right)=1 \label{S3Cor17E1-7}
\end{align}
uniformly  for $Y$ on bounded subsets of $\RR$.
\end{theo}

As shown by (\ref{S5CInvME0300}), the Dirac measure at $x=1$ is instantly regularized to a function $\Lambda(t)$, whose smoothness  is given by (\ref{S5CInvM2E1}), (\ref{S5CInvME3B}). Property, (\ref{S3Cor17E1-7}) shows that, for small values of $t>0$, $\Lambda (t)$  still has a singularity  at $x=1$, of order $|x-1|^{2t-1}$. The regularity properties of $\Lambda (t)$ that are  proved in Theorem \ref{TheoMain1} improve as the value of $t$ increases,  as seen  in  (\ref{S5CInvME3B}).  By  (\ref{S5CAlph1}), (\ref{S5CAlph10}) it may be said that for all $t\in (0, 1/2)$ the function $((\log x) \Lambda (t))/(|x-1|^\alpha )$ is H\"older continuous of order $r-\alpha $ for any $r<2t$ and $\alpha \in (0, r)$. For $t>1$ it follows from (\ref{S5CInvME3B}) that $\Lambda (t)\in C^1(0, \infty)$.  Probably  $\Lambda (t)$ is  H\"older of order $2t-1$ for $t\in (1/2, 1)$ although we did not pursue in that direction. The regularity properties of $\Lambda(t)$ are important in order to prove that it satisfies (\ref{S3EMK2}) for all $t>0, x>0$, $x\not =1$, once it is  has been shown to be a weak solution of (\ref{S4E1})). Detailed asymptotics of $\Lambda (t, x)$ as $x\to 0$ and $x\to \infty$ are given in the  Sectsion below.

The fundamental solution is  used to solve the homogeneous initial value  problem.
\begin{theo}
\label{TheoMain2}
Suppose that $f_0\in L^1(0, \infty)$ and define,
\begin{align}
\label{SISolu}
u(t, x)=\int _0^\infty f_0(y)\Lambda  \left(\frac {t} {y}, \frac {x} {y}\right)\frac {dy} {y},\,\,\,\,\forall t>0, \,\,\forall x>0.
\end{align}
Then, $u\in L^\infty ((0, \infty); L^1 (0, \infty))\cap C((0, \infty); L^1(0, \infty))$ is a weak solution of (\ref{S4E1}). There exists $C>0$ such that
\begin{align}
&||u(t)||_1\le C||f_0||_1,\,\,\forall t>0 \label{TheoMain2-1}
\end{align}
and
\begin{equation}
\label{TheoMain2-3}
u(t)\underset{t\to 0}{\rightharpoonup} f_0,\,\, \hbox{in}\,\, \mathscr D'(0, \infty).
\end{equation}
If $f_0\in L^1(0, \infty)\cap L^\infty(0, \infty)$ then $u(t)\in L^\infty(0, \infty)$ for all $t>0$, there exists a constant $C _{ \infty}>0$ such that,
\begin{equation}
\label{TheoMain2-4}
||u(t)||_\infty \le C _{ \infty }||f_0|| _{ \infty },\,\,\forall t>0. 
\end{equation}
If $f_0\in L^1(0, \infty)\cap L^\infty_{loc}(0, \infty)$, 
\begin{equation}
L(u)\in L^\infty ((0, \infty)\times (0, \infty)),\label{TheoMain2-2}
\end{equation}
there exists a constant $C>0$ and, for $\varepsilon >0$ arbitrarily small, there exists a constant $C_\varepsilon >0$ such that, for all $t>0$ and $x>0$,
\begin{align}
\label{S5FE21-2}
\left|L(u(t))(x)\right|\le C\left(\frac {||f_0|| _{ L^1(0, t) }} {\max\{t^2, x^4\}}+
 ||f_0||_1 xt^{-3}\1 _{ 2x<t }+||f_0|| _{ L^\infty(2x/3, 2x) }\1 _{ t<2x }\right)+\nonumber \\
+C_\varepsilon ||f_0||_1 t^{1-\varepsilon }x^{-3+\varepsilon }\1 _{ t<2x/3 }
\end{align}
and $u$ satisfies  (\ref{S3EMK2})  for all $t>0, x>0$; furthermore $u(t)\in C(0, \infty)$ if $f_0\in C(0, \infty)$.
\end{theo}

The solution $u$ also satisfies the following property,
\begin{prop}
\label{S4EM1PN}
Suppose $f_0\in L^1(0, \infty)$ and $u$ is given by (\ref{SISolu}). Then, for all $t>0$,
\begin{align}
\lim _{ x \to 0 }u(t, x)&=\ell(f_0; t) \in (-\infty, \infty) \label{S4L25E1}\\
\ell(f_0; t)&=A_1t^{-3}\int _0^tf_0(y)y^2dy+A_2t^{-4}\int _0^t f_0(y)y^3dy +\int _0^t f_0(y)\,b_1\!\!\left(\frac {t} {y} \right)\frac {dy} {y}\label{S4L25E2}
\end{align}
where, $A_1$ and $A_2$ are constants given in (\ref{S4Ecata})   and $b_1$ is the function given in (\ref{S3EBB1}) that satisfies $b_1(r)=\mathcal O(r^{-8})$ as $r\to \infty$.
\end{prop}

The existence and uniqueness  of the fundamental solution $\Lambda$ and some of its regularity  properties are proved in Section 2. In Section 3, further properties of $\Lambda$ are obtained like point wise estimates in different regions of the $(t, x)$ plane, and integrability. The  Cauchy problem is solved in Section 4 where Proposition \ref{S4EM1PN} is also proved. Several  quite  technical proofs of some of the auxiliary results are given in the Appendix. 
\section{The fundamental solution $\Lambda$. First properties.}
\setcounter{equation}{0}
\setcounter{theo}{0}

Following the arguments of  \cite{BZ} (cf. \cite{EV, EV2} for two other examples),  the  fundamental solution of (\ref{S4E1}) is obtained  as the inverse Mellin and inverse Laplace  transforms of a solution $V$ of the equation,
\begin{align}
&zV(z, s)=W(s-1)V(z, s-1)+\frac {1} {\sqrt{2\pi} },\,\,z\in \CC,\,\,  {\mathscr Re}(z)>0,\,\,s\in \mathcal S _{ 0, 2 } \label{S4EV1}\\
&W(s)=-2\gamma_e -2\psi \left(\frac {s} {2} \right)-\pi \cot \left(\frac {\pi s} {4} \right),\,\,\,s\in \mathcal S _{ -2, 4 }  \label{S4E7}
\end{align}
where $\gamma _e$ is the Euler constant and $\psi (z)=\Gamma '(z)/\Gamma (z)$ is the Digamma function.  The function $W$ in (\ref{S4E7}) is related with the Mellin transform of the function $H$ in (\ref{S4EHL}) as,
\begin{equation}
\label{S2EHW}
W(s)=-s \int _0^\infty r^{s}H(r)dr,\,\,\forall s\in \mathcal S _{ -2, 4 } 
\end{equation}

\begin{prop}
\label{S3PW}
The function $W$ is meromorphic on $\CC$, analytic on the domain ${\mathscr Re}(s)\in (-2, 4)$ and is such that $W(0)=W(2)=0$. It has actually  a sequence  of zeros and a sequence of poles distributed as follows.  

1.- Poles. The poles of the function $W$ are located at $\{s_n=4n,\, n=1, 2, 3, 4, \cdots\}$ (the residue at these points is $4$) and at $\{s^*_n=-2(2n+1),\,\,n=0, 1, 2, 3\cdots\}$ (the residues at these points is $-4$).

2.- Zeros. The zeros of the function $W$, different from $0$ and $2$, are located at two series of points that we denote $\{\sigma _n,\,\, n=1, 2, 3, \cdots \}$ and $\{\sigma _n^*, \,\,n=0, 1, 2, 3, \cdots\}$. These points are such that $\sigma _n\in (s _{ n +1}-1, s _{ n+1 })$ and $\sigma _n^*\in(s_n^*, s _{ n }^*+1)$.
\end{prop}

\begin{prop}
\label{S3PW17}
The winding number of $W(s)$ is zero for ${\mathscr Re}(s)\in (0, 2)$ and 
\begin{align}
&W(s)= - 2\log |s/2|-\gamma _e+O\left(\frac {1} {|s|^2} \right),\,\,\,s=u+iv,\,|v|\to \infty.\label{S3PW17E1}\\
&W'(s)= \frac {2i} {s}+\mathcal O\left(\frac {1} {v^2} \right),\,\,|v|\to \infty \label{S3PW17E3}
\end{align}
\end{prop}
\begin{proof}
If for all $z\in \CC$,  ${\rm arg} (z)$ denotes the principal argument of $z$ (i.e. $-\pi <{\rm arg} (z)\le \pi $),
\begin{align*}
2\psi \left( \frac {s} {2}\right)=&2\log \left( \left|\frac {s} {2}\right|\right)+2i{\rm arg}\left( \frac {s} {2}\right)+\mathcal O (|v|^{-1}),\,\,v\to \infty\\
=&2\log \left( \left|\frac {s} {2}\right|\right)+i\pi +\mathcal O (|v|^{-1}),\,\,v\to \infty\\
&\pi \cot \left(\frac {\pi s} {4} \right)=-i\pi +\mathcal O\left( e^{-2v}\right),\,\,v\to \infty.
\end{align*}
It follows that
\begin{align*}
W(s)=-2\gamma _e-2\log \left( \left|\frac {s} {2}\right|\right)-i\pi +i\pi +\mathcal O\left( e^{-2v}\right),\,\,v\to \infty
\end{align*}
When $v\to -\infty$,
\begin{align*}
2\psi \left( \frac {s} {2}\right)=&2\log \left( \left|\frac {s} {2}\right|\right)+2i {\rm arg} \left( \frac {s} {2}\right)+\mathcal O (|v|^{-1}),\,\,v\to -\infty\\
=&2\log \left( \left|\frac {s} {2}\right|\right)-i\pi +\mathcal O (|v|^{-1}),\,\,v\to -\infty\\
\pi \cot \left(\frac {\pi s} {4} \right)=& i\pi +\mathcal O\left( e^{2v}\right),\,\,v\to -\infty.
\end{align*}
and (\ref{S3PW17E1}) follows. Similar arguments give (\ref{S3PW17E3}) using
\begin{equation*}
W'(s)= \frac {\pi ^2} {4}\left(\csc\left( \frac {\pi \rho } {4}\right)\right)^2-{ \rm PolyGamma} \left(1, \frac {s} {2}\right)
\end{equation*}
\end{proof}
 
As a first step to solve  (\ref{S4EV1}), (\ref{S4E7}) we consider the ``stationary and homogeneous'' case.
\begin{prop}
\label{S3PB}
For any $\beta \in (0, 2)$ fixed, the problem
\begin{align}
\label{S3PBE1}
B(s)=-W(s-1)B(s-1),\,\,\,\forall s\in \CC; {\mathscr Re}(s)\in (\beta , \beta +1)
\end{align}
admits the following solution,
\begin{align}
\label{S3PBE0}
B(s)=\exp\left(\int  _{ {\mathscr Re} (\rho) =\beta  }  \log (-W(\rho ))\left( \frac {1} {1-e^{2i\pi (s-\rho) } }-\frac {1} {1+e^{-2i\pi \rho }}\right)d\rho\right).
\end{align}
\end{prop}
\begin{proof} In order to solve  (\ref{S3PBE1}) we notice that, if logarithms may be taken to both sides of the equation the following identity would follow:
\begin{align}
\label{S5EBL}
\log (B(s+1))=\log (B(s))+\log(-W(s)).
\end{align}
Then, for any $\beta \in (0, 1)$ fixed, we define the new variables,
\begin{align}
&\forall s\in \CC;\, {\mathscr Re}(s)\in (\beta , \beta +1),\,\,\,\zeta  =e^{2i\pi (s-\beta )} \label{S5EPsi1} \\
& Q(\zeta )  =\log (-W(s)) \label{S5EPsi2}
\end{align}
In order for the change of variable (\ref{S5EPsi2}) to be uniquely defined it is necessary to fix the argument of the function $\log (-W(s))$.
Since $W(s)$ is analytic on the strip ${\mathscr Re}(s)\in (0, 3)$, the function $Q$ is analytic on $\CC$. By (\ref{S3PW17E1}),
\begin{align*}
-W(s)&=2\log\left(\frac {|v|} {2} \right)+\gamma _e+\mathcal O\left(\frac {1} {|v|} \right),\,\,\,s=u+iv,\,\,\,|v|\to \infty\\
\log (-W(s))&=\log\left(2\log\left(\frac {|v|} {2} \right)+\gamma _e+\mathcal O\left(\frac {1} {|v|} \right) \right)=\log (\log |v|)+\mathcal O(1),\,\,\,|v|\to \infty.
\end{align*}
Since by definition $
|\zeta |=e^{-2\pi v}$, $|v|=\frac {\left|\log |\zeta | \right|} {2\pi }$ and
\begin{align}
Q(\zeta )&\underset{|v|\to \infty}{=}\log (\log |v|)+\mathcal O(1)\underset{|\log |\zeta ||\to \infty}{=}\log (\log \left|\log |\zeta | \right|)+\mathcal O(1).\label{S5EQX1}
\end{align}
The function $Q$ is then very slowly divergent as $|\zeta |\to \infty$ or $|\zeta |\to 0$. 

On the other hand, let us write $s=u+iv$ with $u\in \RR$ and $v\in \RR$ and consider the limits of the variable $\zeta =\zeta (s)$ defined in (\ref{S5EPsi1}) when $u\to \beta^+$ and $u\to (\beta +1)^-$ for $v\in \RR$ fixed,
\begin{align*}
\forall v\in \RR:\qquad \quad \lim _{ u\to \beta^+ }\zeta =e^{-2\pi v}\lim _{ \theta\to 0 }e^{i\theta},\quad \lim _{ u\to (\beta+1)^- }\zeta =e^{-2\pi v}\lim _{ \theta\to 2\pi }e^{i\theta}  
\end{align*}
By (\ref{S5EQX1}), the following Cauchy's integral:
\begin{equation}
\psi (\zeta )=\frac {1} {2i\pi }\int _0^\infty Q(r) \left( \frac {1} {r-\zeta }-\frac {1} {r+1}\right)dr,\,\,\, \forall \zeta \in \CC\setminus [0, \infty) \label{S5PsiDef}
\end{equation}
is  absolutely convergent  for all $\zeta \in \CC\setminus [0, \infty)$. If we  denote,
\begin{align}
\forall r\in \RR: \qquad \quad \psi(r+i0)=\lim _{ \theta\to 0 }\psi (re^{i\theta}) ,\quad \psi(r-i0)=\lim _{ \theta\to 2\pi  }\psi (re^{i\theta}),  \label{S5Psi3}
\end{align}
then, 
\begin{equation}
\label{S5EBL2}
\psi (r-i0)=\psi (r+i0 )+Q(r),\,\,\,\forall r>0.
\end{equation}
  
The function  $b(s)=\psi (\zeta )$,  defined, for $s\in \CC;\, {\mathscr Re}(s)\in (\beta , \beta +1)$ as,
\begin{align*}
b(s)=&\int _0^\infty Q(r) \left( \frac {1} {r-\zeta }-\frac {1} {r+1}\right)dr,\,\,\, r=e^{2i\pi (\rho -\beta )},\,\,\, dr=2i\pi rd\rho \\
=&\int  _{ {\mathscr Re} (\rho) =\beta  } \log (-W(\rho ))\left( \frac {1} {1-e^{2i\pi (s-\rho )} }-\frac {1} {1+e^{-2i\pi (\rho -\beta )}}\right)d\rho
\end{align*}
satisfies,
\begin{align*}
b(s+1)=b(s)+\log(-W(s)), \forall s\in \CC; {\mathscr Re}(s)\in (\beta , \beta +1)
\end{align*}
and  the function  $B(s)=\exp(b(s))$,
\begin{equation}
\label{S5EtheB}
B(s)=\exp\left(\int _{ {\mathscr Re} (\rho) =\beta  } \log (-W(\rho ))\left( \frac {1} {1-e^{2i\pi (s-\rho) } }-\frac {1} {1+e^{-2i\pi \rho }}\right)d\rho\right).
\end{equation}
satisfies (\ref{S3PBE1}).\end{proof}
By classical arguments of complex variables it is  straightforward to check that the function $B$ obtained in Proposition \ref{S3PB} satisfies the following,
\begin{prop}
\label{S3PBP}
The function $B$ is  analytic on the domain $ s\in \mathcal S _{\, 0, 2 }$ where it is given by the integral in (\ref{S5EtheB}) for some $\beta \in (0, 1)$ such that $\beta <\mathscr Re s<\beta +1$. It is extended to a meromorphic on the complex plane by the following relation, 
\begin{equation}
B(s)=-W(s-1)B(s-1),\forall s\in \CC. \label{S5EBX5}
\end{equation}
It has a sequence of poles and a sequence of zeros, determined by the zeros and poles of the function $W$ as follows.

1.-Poles. The poles of the function $B$  are located at $s=0$, $s=-1$,  at $\{4n+1, n=1, 2, 3, \cdots\}$ and at  $\{\sigma _n^*,\,\,n=1, 2, 3, \cdots \}$.

2.-Zeros. The zeros of the function $B$ are at $s=3$, $s=4$ at $\{-n, \,\,\,n=6, 7, 8, \cdots\}$  and at $\{\sigma _n+1,\,\,\,n=1, 2, \cdots\}$.
\end{prop}

\begin{prop}
\label{S3PBP24}
Let $B$ the function defined by  (\ref{S5EtheB}). Then, for  all $R>0$ there exist two positive constants $C_1$ and $C_2$ such that
$$
C_1\le |B(s)|\le C_2.
$$
for all $\mathscr Re(s)\in (0, 2)$  and $|\mathscr Im (s)|>R$.
\end{prop}
\begin{proof}
The function $\log (-W(s))$ is,
$$
\log(-W(s))=\log(|W(s)|)+i Arg (-W(s)).
$$
Since, by Proposition (\ref{S3PW17}), $ {\rm arg} (-W(s))\to 0$ as $\mathscr Im(s)\to \pm \infty$, we may take the principal branch of  the  function $\log(-W(s))$ and,
\begin{align*}
\lim _{ \zeta \to 0 } {\rm arg} (-W(\zeta ))=0,  \qquad \lim _{ \zeta \to \infty } {\rm arg} (-W(\zeta ))=0 
\end{align*}
It follows from Lemma C.2 in \cite{ESV} that the function $\psi $ defined in (\ref{S5PsiDef}) satisfies,
\begin{align*}
&\psi (\zeta )=i\Theta(\zeta )+o(\log |\zeta |),\,\,\,\zeta \to 0\\
&\psi (\zeta )=i\Theta(\zeta )+o(\log |\zeta |),\,\,\,|\zeta| \to \infty\\
&\Theta(\zeta )=-\frac {1} {2\pi }\int _0^\infty \log (|W(s)|) \left( \frac {1} {s-\zeta }-\frac {1} {s+1}\right)ds.
\end{align*}
We deduce that 
$$
\lim _{\mathscr Im (s)\to \infty  }|B(s)|=\lim _{\mathscr Im (s)\to -\infty  }|B(s)|=1
$$
and the result follows.
\end{proof}

\begin{prop}
\label{S3PBP24B}
For all $M>0$ and $R>0$, there exists two positive constants $C _{ 1, M } $ and $C _{ 2, M }$ such that, for all $s\in \CC$, $|\mathscr Re(s)|\le M$, and $|\mathscr Im(s)|>R$,
\begin{equation}
\label{S3PBP24B-1}
C _{ 1, M }\log |\mathscr Im s|\le B(s)\le C _{ 2, M }\log |\mathscr Im s|.
\end{equation}
\end{prop}
\begin{proof}
If $0<\mathscr Re(s)\le 2$ we may apply Proposition (\ref{S3PBP24B}). If for example $\mathscr Re(s)\in (2, 3)$, we use (\ref{S5EBX5}) to write: 
$$
B(s)=-W(s-1)B(s-1)
$$
where now $\mathscr Re(s-1)\in (0, 2)$. We deduce,
$$
C_1|W(s-1)|\le |B(s)|\le C_2|W(s-1)|.
$$
and (\ref{S3PBP24B-1}) follows by Proposition \ref{S3PW17}.
\end{proof}

\begin{rem}
The function $B$ given in (\ref{S5EtheB}) is not the only that satisfies (\ref{S5EBX5}). Indeed  many others are obtained by means of
\begin{align}
B_\ell(s)=e^{2i\pi \ell s}\, B(s), \forall \ell \in \ZZ
\end{align}
and linear combinations of them.  
\end{rem}

It follows  easily  from (\ref{S3PBE0})  in Proposition \ref{S3PB}
\begin{cor}
For all $s\in \CC$ and $Y\in \CC$  such that ${\mathscr Re}(s)\in (0, 3)$ and $s+Y\in \mathcal S  _{ 0, 3 }$
\begin{align}
\label{S3ETheta1}
&\frac {B(s)} {B(s+Y)}=
\exp\left(\int  _{ {\mathscr Re}(\rho ) =\beta  } \log (-W(\rho ))\Theta (\rho-s , Y) d\rho\right),\,\,\beta \in (0, 3)\\
\label{S3ETheta}
&\Theta (\sigma , Y)=\frac {1} {1-e^{-2i\pi \sigma }}-\frac {1} {1-e^{2i\pi (-\sigma +Y)}}.
\end{align}
\end{cor}
The problem (\ref{S4EV1}), (\ref{S4E7}) is  reduced to a simpler one using the  auxiliary function $B(s)$. 
\begin{prop}
\label{PropppS4EBH}
The function defined by the integral
\begin{align}
V(z, s)=\frac {1} {2i\pi }\frac {B(s)} {\sqrt{2\pi }\,z}\int _{\mathscr Re (\sigma )=\beta } \frac {e^{(\sigma-s) \log(-z) }} { B(\sigma )}\frac {d\sigma } {(1-e^{2i\pi (s-\sigma )}) }.
\label{S4E1410}
\end{align}
for $\beta \in (0, 2)$ such that $\beta <\mathscr Re s<\beta +1$, is well defined and analytic for $\mathscr Re(z)>0$ and $ s \in \mathcal S _{0 , 2}$ where it satisfies,
\begin{equation}
zV(z, s)=W(s-1)V(z, s-1)+\frac {1} {\sqrt{2\pi} }.\label{S4E1411}
\end{equation}
\end{prop}
\begin{proof}

Let us define the function $H(z, s)$ as,
\begin{align}
\label{S4EBH}
V(z, s)=e^{-s\log(-z)}B(s)H(z, s).
\end{align}
where  $\log (z)=\log (|z|)+iArg (z)$ and $Arg (z)\in (-2\pi , 0 ]$.

The equation (\ref{S4EV1}) on $V$  yields the following equation for $H$:

\begin{align*}
ze^{-s\log(-z)}B(s)H(z, s)&=e^{-(s-1)\log(-z)}W(s-1)B(s-1)H(z, s-1)+\frac {1} {\sqrt{2\pi }}\\
&=-ze^{-s\log(z)}W(s-1)B(s-1)H(z, s-1)+\frac {1} {\sqrt{2\pi }}\\
B(s)H(z, s)&=-W(s-1)B(s-1)H(z, s-1)+\frac {e^{s\log (-z)}} {\sqrt{2\pi }\, z}\\
B(s)H(z, s)&=B(s)H(z, s-1)+\frac {e^{s\log (-z)}} {\sqrt{2\pi }\, z}
\end{align*}
and then,
\begin{align}
\label{S4E14}
&H(z, s)-H(z, s-1)=\frac {e^{s\log (-z)}} {\sqrt{2\pi }\, z B(s) },\,\,z\in \CC,  {\mathscr Re}(z)>0,\,\,s\in \CC, {\mathscr Re}(s)\in (0, 2)
\end{align}
We may use again the change of variables (\ref{S5EPsi1}) and define,
\begin{align*}
h(z, \zeta )=H(z, s),\,\,\,\widetilde B(\zeta )=B(s)
\end{align*}
and deduce from (\ref{S4E14}) that $h$ has to satisfy
\begin{align}
h(z, r -i0)&=h(z, r +i0)+\frac {e^{2i\pi \beta \alpha (z)}r ^{\alpha (z)}} {\sqrt {2\pi }\,z\widetilde B(r )}\,\,\,\, \forall r>0 \label{S5EhX1}\\ 
\alpha (z)&=\frac {\log (-z)} {2i\pi }.\label{S5EhX2}
\end{align}
It follows that
\begin{align*}
\alpha (z)&=\frac {\log (-z)} {2i\pi }=-i\frac {\log|z|} {2\pi }+\frac {Arg(-z)} {2\pi }
\end{align*}
and the choice of the $\log (z)$ is such that $-1<\mathscr Re(\alpha (z))<0$. By Proposition (\ref{S3PBP24B}) it follows that the  integral 
\begin{align*}
h(z, \zeta )=\frac {1} {2i\pi }\frac {1} {\sqrt{2\pi }}\frac {e^{2i\pi \beta \alpha (z)}} {z}\int _0^\infty \frac {r^{\alpha (z)}} {\widetilde B(r)}\frac {dr} {(r-\zeta) }
\end{align*}
is absolutely convergent and defines a function $h$ analytic on the domain 
$$\left\{(z, s);\,\, z\in \CC,\,\,\mathscr Re (z)>0,\,\,\,s\in\CC\setminus [0, \infty)\right\}$$ 
that satisfies  (\ref{S5EhX1}). Using the original variables we obtain that
\begin{align}
\label{S4E14X}
H(z, s)=\frac {1} {\sqrt{2\pi }}\frac {1} {z}\int _{\mathscr Re (\sigma )=\beta } \frac {e^{ \sigma \log(-z) }} { B(\sigma )}\frac {d\sigma } {(1-e^{2i\pi (s-\sigma )}) }
\end{align}
is well defined, analytic on  $z\in \CC,  {\mathscr Re}(z)>0,\,\,s\in \CC, {\mathscr Re}(s)\in (\beta ,\beta+1)$ where it  satisfies
\begin{align}
\label{S4E1409}
&H(z, s)-H(z, s-1)=\frac {e^{ s\log (-z)}} {\sqrt{2\pi }\, z B(s) }.
\end{align}
Since $\beta \in (0, 2)$ is arbitrary, using a  contour deformation argument  in the integral of the right hand side of (\ref{S4E1409}),  $H$ is  extended as an analytic function  $z\in \CC,  {\mathscr Re}(z)>0$ and  $s\in \CC, {\mathscr Re}(s)\in (0, 2)$.

Using now (\ref{S4EBH}) we recover the function
\begin{align*}
V(z, s)=\frac {1} {2i\pi }\frac {B(s)} {\sqrt{2\pi }\,z}\int _{\mathscr Re (\sigma )=\beta } \frac {e^{(\sigma-s) \log(-z) }} { B(\sigma )}\frac {d\sigma } {(1-e^{2i\pi (s-\sigma )}) }.
\end{align*}
Since $B$ is analytic and non zero on $\mathscr Re(s)\in (0, 2)$ and $\beta \in (0, 2)$ is arbitrary the function $V$ is analytic on  $z\in \CC,  {\mathscr Re}(z)>0$ and  $s\in \CC, {\mathscr Re}(s)\in (0, 2)$ and satisfies the equation (\ref{S4E1411}) for $\mathscr Re s\in (1, 2)$.
\end{proof}

\begin{cor}
\label{S5Cor1U}
The inverse Laplace transform of $V$
\begin{align*}
U(t, s)=\frac {1} {2i\pi }\int  _{ d-i\infty }^{d+i\infty}e^{zt}V(z, s)dz,\,\,\,\,\, \beta -1<d<\beta,
\end{align*}
is well defined for all $t>0$, $\mathscr Re (s)\in (0, 2)$, and satisfies,  
\begin{align}
\label{S5EUX1}
&U(t, s)=\frac {B(s)} {\sqrt{2\pi}}\frac {1} {2i\pi }\int  _{ \mathscr Re(\sigma )=\beta  }\frac {t^{-(\sigma -s)}\Gamma (\sigma-s) } { B(\sigma )}d\sigma,\,\,\,\forall \beta \in (\mathscr Re(s), 2);\\
&\forall t>0,\,U(t, \cdot)\,\hbox{is an analytic function on}\,\,\mathcal S _{ 0, 2 } \label{S5EUX1C}\\
&\forall k\in \NN,\,\,U\in C^k((0, \infty)\times \mathcal S _{ 0, 2 } ) \label{S5EUX1B}\\
&\forall t>0,\,U(t, \cdot)\,\hbox{is analytic on}\,\,\mathcal S _{ 0, 2 } ,\hbox{and  meromorphic in}\,\, \CC \label{S5EUX1D}\\
&\frac {\partial U} {\partial t}(t, s )=W(s-1)U(t, s-1) \forall t>0,\,\,\forall s\in \mathcal S _{ 1, 3 }. \label{S5EUXYp}
\end{align}
\end{cor}
\begin{proof}
For all $\sigma $, $s$ such that $\mathscr Re s<\mathscr Re \sigma $, and $d>0$,

\begin{align*}
\frac {1} {2i\pi }\int  _{ c-i\infty }^{c+i\infty}\frac {e^{zt}} {z}e^{(\sigma -s)\log(-z)}dz=t^{-(\sigma -s)}\Gamma (\sigma  -s)\left(e^{2i\pi (\sigma -s)}-1 \right).
\end{align*}
We use now that Stirling's formula for $\Gamma (z)$ is uniformly valid for ${\rm arg} z\in (-\pi +\varepsilon _0, \pi -\varepsilon _0)$ with $\varepsilon _0>0$, to deduce that, for all $R>0$ and $\beta \in (0, 2)$
\begin{equation}
\label{S5EGamma}
\left|\Gamma (\sigma -s)\right|\le C_R\frac {e^{-\frac {\pi |\sigma |} {2}}} {\sqrt{1+|\sigma |}},\,\,\,\forall s; |s|\le R.
\end{equation}
The integral at the right hand side of (\ref{S5EUX1}) is then absolutely convergent  the identity  (\ref{S5EUX1})  and (\ref{S5EUX1C}) follow for $\beta -1<\mathscr Re s<\beta $. We also deduce from (\ref{S5EGamma}) that for all $k\ge 1$ the integrals 
\begin{equation*}
\int  _{ \mathscr Re(\sigma )=\beta  }\frac {d} {dt} \left(t^{-(\sigma -s)}\right) \frac {\Gamma (\sigma-s) } { B(\sigma )}d\sigma
\end{equation*}
are absolutely convergent and analytic functions of $s$ on the strip $\mathscr Re s \in (0, 2)$. Therefore,
\begin{equation*}
\frac {\partial  ^k} {\partial t^k}U(t, s)=-\frac {B(s)} {\sqrt{2\pi}}\frac {1} {2i\pi }\int  _{ \mathscr Re(\sigma )=\beta  }\frac {d} {dt} \left(t^{-(\sigma -s)}\right) \frac {\Gamma (\sigma-s) } { B(\sigma )}d\sigma.
\end{equation*}
and (\ref{S5EUX1B}) follows.

On the other hand since
\begin{align*}
\frac {1} {2i\pi }\int  _{ d-i\infty }^{d+i\infty}e^{zt}e^{(\sigma -s)\log(-z)}dz=t^{-(\sigma -s)-1}\Gamma (1+\sigma  -s)\left(e^{2i\pi (\sigma -s)}-1 \right)
\end{align*}
the inverse Laplace transform of $zV(z)$ is well defined for all $t>0$ and given by,
\begin{equation*}
\frac {1} {2i\pi }\int  _{ d-i\infty }^{d+i\infty}e^{zt}zV(z, s)dz=-\frac {B(s)} {\sqrt{2\pi}}\int  _{ \mathscr Re(\sigma )=\beta  }\frac {t^{-(\sigma -s)-1}\Gamma (1+\sigma-s) } { B(\sigma )}d\sigma.
\end{equation*}
The expression (\ref{S5EUX1}) indicates that  $U(\cdot, s)\in C((0, \infty))$.  In order to see that  $U(\cdot, s)\in C([0, \infty))$ we first deform the integration contour in (\ref{S5EUX1}) towards lower values of $\beta $ and cross the pole of the function $\Gamma (\sigma -s)$ at $\sigma -s=0$,
\begin{align}
\label{S5EU0} 
U(t, s)=\frac {1} {\sqrt{2\pi }}-\frac {B(s)} {\sqrt{2\pi}}\frac {1} {2i\pi }\int  _{ \mathscr Re(\sigma )=\beta'  }\frac {t^{-(\sigma -s)}\Gamma (\sigma-s) } { B(\sigma )}d\sigma ,\,\,\, \beta '\in (0, \mathscr Re s).
\end{align}
Since  now  $\mathscr Re (\sigma -s)<0$, it follows that $U(\cdot, s)\in C([0, \infty))$ and $U(0, s)=\frac {1} {\sqrt{2\pi }}$. Using 
\begin{align*}
\mathscr L\left(U_t(\cdot, s) \right)(z)=zV(z, s)-U(0, s),
\end{align*}
we deduce
\begin{align*}
\frac {\partial U} {\partial t}(t, s)=\frac {1} {2i\pi }\int  _{ d-i\infty }^{d+i\infty} e^{zt}\left(zV(z, s)-\frac {1} {\sqrt {2\pi }}\right) dz
\end{align*}
We apply now the inverse Laplace transform  to both sides of the equation (\ref{S4E1411}) with $\mathscr Re s\in (1, 2)$, 
 since $U(t)$ is analytic on $\mathcal S _{ 0, 2 }$ and so is $W$ on $\mathcal S _{ -2, 4 }$, (\ref{S5EUXYp}) follows.
\end{proof}

\noindent
The  following decay  property of $U(t)$, makes  possible to invert its Mellin transform.

\begin{prop} 
\label{S3PX1}
For all $s\in \mathcal S$, for $T>0$ and $t\in (0, T)$,
\begin{align}
\left|U(t, s) \right| &\le C_Te^{-2t\log |bs|},\,\,\,\,\,b=\frac {e^{\frac {\gamma _e} {2}}} {2},\label{S3PX1E1}\\
(1+|s|)\left|\frac {\partial  U} {\partial s}(t, s) \right|&+(1+|s|)^{2}\left|\frac {\partial ^2 U} {\partial s^2}(t, s) \right|\le C_Tte^{-2t\log (|bs|)}\label{S3PX1E2}
\end{align}
\end{prop}
The proof of Proposition  \ref{S3PX1} is essentially the same as that of Proposition 8.1 in \cite{EV}, only differing in small details,  and is presented in the Appendix. 

As a Corollary, the inverse Mellin transform of $U(t)$ is well defined.
\begin{cor}
\label{S5CInvM}
For every $t>0$ there exists a unique  distribution $\Lambda(t):=\mathscr M^{-1}(U(t))\in  E' _{ 0, 2 }$, the inverse Mellin transform of $U(t)$ such that:
\begin{align}
\label{S5CInvME1}
&\mathscr M \left(\Lambda (t)\right)(s)=U(t, s),\,\, \forall s\in \mathcal S _{ 0, 2 }\\
\label{S5CInvME1B}
& \Lambda\in C((0, \infty); E' _{ 0, 2 }).
\end{align}
For all $t>0$ it is given by the following expression,

\begin{equation}
\label{S5CInvME2}
\Lambda (t, x)=\left(x\frac {\partial } {\partial x} \right)^2\left( \frac {1} {2\pi i}\int  _{ c-i\infty }^{c+i\infty} U(t, s)s^{-2}x^{-s}ds\right),\,\,c\in (0, 2).
\end{equation}
When $t>1/2$,
\begin{equation}
\label{S5CInvME2B}
\Lambda (t, x)= \frac {1} {2\pi i}\int  _{ c-i\infty }^{c+i\infty} U(t, s)x^{-s}ds,\,\,c\in (0, 2).
\end{equation}
\end{cor}
\begin{proof}
By Corollary \ref{S5Cor1U}, for every $t>0$, the function $U(t)$ is analytic on the strip $\mathscr Re s\in (0, 2)$. By Proposition \ref{S3PX1}
$$
|U(t, s)|\le |bs|^{-2t},\,\,\forall t\in (0, 1).
$$
It follows that, for all $t>0$, the function $s^{-K+2}U(t, s)$  is analytic and bounded on the strip $\mathscr Re s\in (0, 2)$ as $|s|\to \infty$ kor $K=2$. It follows from Theorem 11.10.1 in \cite{ML} that there exists a unique tempered distribution $\Lambda (t)\in E' _{ 0, 2 }$ that satisfies (\ref{S5CInvME1}) and is given by (\ref{S5CInvME2}). 
\
As soon as $t>1/2$, the integral in the right hand side of (\ref{S5CInvME2B}) is absolutely convergent and its Mellin transform is $U(t)$ from where it is equal to $\Lambda(t)$. Property (\ref{S5CInvME1B}) follows from (\ref{S5EUX1B}) and the continuity of the inverse Mellin transform.
\end{proof}

We now obtain the  inverse Mellin transform of both sides of  equation (\ref{S5EUXYp}).
\begin{prop}
\label{S5PInv1} 
\begin{align}
&\Lambda (t) \in C^1(0, \infty; E' _{ 1, 3 }) \label{S5PInv1E00}\\
&\frac {\partial \Lambda} {\partial t}=\left(\frac {\partial \Lambda} {\partial x}\ast H\right)\,\,\hbox{in}\,\,C((0, \infty); E' _{ 1, 3 })\label{S5PInv1EX0}
\end{align}
\end{prop}
\begin{proof} 
By (\ref{S5CInvME1B}),  $\partial _x\Lambda(t)\in C(0, \infty; E' _{ {1, 3} })$ and fo all $s\in  \mathcal S _{ 1, 3 }$,

\begin{equation*}
\mathscr M (\partial _x\Lambda(t))(s)=-(s-1)U(s-1),\,\hbox{and}
\end{equation*}
Since $\mathscr M (H)(s)=-\frac {W(s-1)} {s-1}$, it then follows for all $t>0$,
\begin{equation*}
\mathscr M^{-1}(W(s-1)U(t, s-1))(x)=\left(\frac {\partial \Lambda (t)} {\partial x}\ast H\right)(x)\,\,\hbox{in}\,\,E' _{ 1, 3 } 
\end{equation*}
On the other hand, by (\ref{S5EUXYp}) and Proposition \ref{S3PX1}
\begin{align}
\mathscr M^{-1}\left( \frac {\partial U(t)} {\partial t}\right)(x)&\equiv \mathscr M^{-1}\left(W(s-1)U(t, s-1) \right)=\nonumber \\
&=\left(x\frac {\partial } {\partial x} \right)^2 \left( \frac {1} {2\pi i}\int  _{ c-i\infty }^{c+i\infty}  W(s-1)U(t, s-1)s^{-2}x^{-s}ds\right). \label{S4EUt}
\end{align} 
By  Proposition \ref{S3PX1} again, for all $t>0$ and $x>0$,
\begin{align}
\frac {d} {dt} \left( \frac {1} {2\pi i}\int  _{ c-i\infty }^{c+i\infty} U(t, s)s^{-2}x^{-s}ds\right)=\frac {1} {2\pi i}\int  _{ c-i\infty }^{c+i\infty} W(s-1)U(t, s-1)s^{-2}x^{-s}ds \label{S4EUt}
\end{align} 
and the integral in the right hand side of (\ref{S4EUt}) is absolutely convergent, uniformly for $x$ and $t$  in compacts subsets of $(0, \infty)\times (0, \infty)$. It  is then a continuous function on $(0, \infty)\times (0, \infty)$.  It is then possible to apply  the operator $\left(x \partial _x\right)^2$ to both sides of (\ref{S4EUt})  in the sense of distributions to obtain (\ref{S5PInv1EX0}).
\end{proof}
We prove in the next Proposition,  some first  properties of $\Lambda$.
\begin{prop}
\label{S5CInvM2} 
The function $\Lambda(t)$  defined in Corollary \ref{S5CInvM} satisfies properties (\ref{S5CInvME0300})--(\ref{S5CAlph10}).
\end{prop}
\begin{proof} 
Since $\Lambda (t)\in E' _{ 0, 2 }$, $\mathscr M(((\log x ) \partial ^m_t\Lambda (t))(x))=\partial _s\partial ^m_t U(t, s)$ in $\mathcal S _{ m, 2+m }$. Then,
\begin{equation}
\label{S5CInvMEU}
((\log x ) \partial ^m_t \Lambda (t))(x)=\frac {1} {2\pi i}\int  _{ c'-i\infty }^{c'+i\infty} \partial _s\left(U(t, s-m)\prod _{ \ell=1 }^m W(s-\ell)\right)x^{-s}ds
\end{equation}
with $c'\in (m, 2+m)$, because,  by Proposition \ref{S3PX1}, (\ref{S3PW17E1}) and (\ref{S3PW17E3}), the integral in (\ref{S5CInvMEU}) is absolutely convergent. Since the convergence is  uniform for $x$ and $t$ on compact subsets of $(0, \infty)\times (0, \infty)$, (\ref{S5CInvME03}) follows. A similar argument shows (\ref{S5CInvM2E1}).
On the other hand, when $t>1/2$,  using (\ref{S4L1E1}) if we deform the integration contour in (\ref{S5CInvME2B})  towards lower values of $\mathscr Res$ and cross the pole of $B(s)$ at $s=0$, using ${\rm Res}(B(s), s=0))=-B(1)/W'(0)$
\begin{align*}
\Lambda (t, x)= &  \frac {1} {4 \pi^2\sqrt{2\pi }}\int  _{ {\mathscr Re}(s)=c } x^{-s } \int  _{ \mathscr Re \sigma  =\beta  }\frac {B(s)} {B(\sigma  )}\Gamma (\sigma  -s)t^{-(\sigma -s) }d\sigma  ds\\
= & \frac {B(1)} {2i\pi \sqrt{2\pi }W'(0)}  \int  _{ \mathscr Re \sigma  =\beta  }\frac {\Gamma (\sigma )t^{-\sigma -1 }} {B(\sigma  )} d\sigma+\\
&+ \frac {1} {4 \pi^2\sqrt{2\pi }}\int  _{ {\mathscr Re}(s)=c'' } x^{-s } \int  _{ \mathscr Re \sigma  =\beta  }\frac {B(s)} {B(\sigma  )}\Gamma (\sigma  -s)t^{-(\sigma -s) }d\sigma  ds,\,\,\,c''\in (-1, 0)
\end{align*}
It follows first that $\Lambda \in C([1/2,\infty)\times [0, \infty))$  since both integrals converge uniformly for $x$ and $t$ on compact subsets of $[0, \infty)\times [1/2, \infty)$. 
For $t\in (0, 1/2)$
\begin{equation*}
(\log x )  \Lambda (t, x)=\frac {1} {2\pi i}\int  _{ c-i\infty }^{c+i\infty} \partial _s U(t, s)x^{-s}ds
\end{equation*}
It follows from (\ref{S5EUX1})that  $U(t)$ is meromorphic on the strip $\mathcal S _{ -1, 2 }$ with a simple pole at $s=0$. Then, $\partial _s U(t, s)$ is also meromorphic on  $\mathcal S _{ -1, 2 }$ and has a pole of order 2 at $s=0$. We deduce, for $c''\in (-1, 0)$
\begin{equation*}
(\log x )  \Lambda (t, x)=- \frac {B(1)} {2i\pi \sqrt{2\pi }W'(0)}  \int  _{ \mathscr Re \sigma  =\beta  }\frac {\Gamma (\sigma )t^{-\sigma -1 }} {B(\sigma  )} d\sigma+ \frac {1} {2\pi i}\int  _{ c''-i\infty }^{c''+i\infty} \partial _s U(t, s)x^{-s}ds.
\end{equation*}
We deduce arguing as before that $(\log x )  \Lambda\in C((0, 1/2)\times [0, \infty))$ and (\ref{S5CInvME0300}) follows.
For $t>1$ the identity (\ref{S5CInvME2B}) may be used for $k<2t-1$, and for $c'\in (m+k, 2+m+k)$
\begin{equation}
\label{S5CInvMECK}
\frac {\partial ^{k+m}\Lambda } {\partial x^k\partial t^m}=\frac {(-1)^k} {2\pi i}\int  _{ c'-i\infty }^{c'+i\infty} (s-k)_k \left(U(t, s-m)\prod _{ \ell=1 }^m W(s-\ell)\right)x^{-s-k}ds.
\end{equation}
Property  (\ref{S5CInvME3B}) follows since, by Proposition \ref{S3PX1}, (\ref{S3PW17E1}) and (\ref{S3PW17E3}) 
$$
\left|(s-k)_k U(t, s-m)\prod _{ \ell=1 }^m W(s-\ell)\right |\le C|s|^{k-2t}|\log |s||^{m},\,\, \hbox{for}\,\,|s|>>1,
$$
and therefore, the  integral in (\ref{S5CInvMECK}) converges absolutely in compacts of  $(0, \infty)\times (0, \infty)$.

For  all  $t\in (0, 1/2)$, $r\in(0, 2t)$, and $|s|$ large,
\begin{align}
\label{S3EDr}
\left| \frac {\partial } {\partial s}U(t, s-r)\frac {\Gamma (1-s+r)} {\Gamma (1-s)}x^{1-s} \right|\le |x|^{-\mathscr Re (s+r)} |s|^{-2t-1+r}.
\end{align}
the fractional derivative of order $r$ of $(\log x) \Lambda$, is  then
\begin{equation}
\label{S3EDr2}
\frac {\partial^r (\log x)\Lambda (t)}{\partial  x^r} =
\frac {1} {2\pi i}\int  _{ c'-i\infty }^{c'+i\infty}\frac {\Gamma (1-s+r)} {\Gamma (1-s)} \frac {\partial } {\partial s}U(t, s-r)x^{-s}ds
\end{equation}
$c'\in (r, 2)$,
(cf. \cite{P}, \S 2.10),  where the integral in the right hand side of (\ref{S3EDr2}) converges absolutely for $x$ and $t$ in compact subsets of $(0, \infty)\times (0, \infty)$. For each $t>0$ the function $(\log x)  \Lambda (t)$ has  continuous fractional $x-$derivative of order $r$
on every compact subset of $(0, \infty)$ and by (\ref{S3EDr2}),  for all $t>0$
\begin{equation}
\label{S5CAlph10z}
\forall r'\in (r, 2),\,\exists C _{ r' }>0,\,\,\,\left|\frac {\partial ^r (\log x)\Lambda (t, x)} {\partial x^r} \right|\le C _{ r' }x^{-r'},\,\,\forall x>0.
\end{equation}
By Theorem 3.1 \cite{S},  (\ref{S5CAlph1}) follows for $\alpha =0$ and $r\in (0, 2t)$. 

Since,
\begin{align*}
(\log x)  \Lambda (t, x)=\frac {1} {2i\pi  }\int_{ {\mathscr Re}(s)=c } \frac {\partial U(t, s)} {\partial s}x^{-s}ds,
\end{align*}
by the continuity property (\ref{S5CInvME03}), and an integration by parts,

\begin{equation*}
\lim _{ x\to 1 }(\log x)  \Lambda (t, x)=\frac {1} {2i\pi  }\int_{ {\mathscr Re}(s)=c } \frac {\partial U(t, s)} {\partial s}ds=0.
\end{equation*}
Then,  for $\alpha >0$ such that $\alpha +r<2t$ property  (\ref{S5CAlph1}) is deduced using the result in \cite{M}, p. 14.  Estimate (\ref{S5CAlph10}) follows from the same result in  \cite{M} and  (\ref{S5CAlph10z}).
 \end{proof}

\begin{cor}
\label{S5CSol}
The function $\Lambda$ satisfies
\begin{equation}
\label{S5CSolE1}
\lim _{ t\to 0 }\Lambda (t)=\delta _1,\,\,\,\hbox{in}\,\,\,\,\mathscr D'(0, \infty).
\end{equation}
\end{cor}
\begin{proof}
Consider  any test function $\varphi \in \mathscr D(0, \infty)$ and suppose that ${\rm supp} (\varphi ) \subset (a, b)$ for some $0<a<b<\infty$. Then
\begin{align*}
\langle \Lambda (t), \varphi \rangle -\varphi (1)&=
\int _0^\infty  \mathscr M^{-1}\left(U(t)-\frac {1} {\sqrt{2\pi }}\right)(x)\, \varphi (x) dx\\
&=\frac {1} {2i\pi }\int _0^\infty \int  _{ c-i\infty }^{c+i\infty}\left(U(t, s)-\frac {1} {\sqrt{2\pi }}\right)x^{-s}ds \varphi (x)dx\\
&=\frac {1} {2i\pi }\int  _{ c-i\infty }^{c+i\infty} \int _0^\infty  x^{-s} \varphi (x) dx\left(U(t, s)-\frac {1} {\sqrt{2\pi }}\right)ds\\
&=\frac {1} {2i\pi }\int  _{ c-i\infty }^{c+i\infty} \mathscr M( \varphi) (1-s)\left(U(t, s)-\frac {1} {\sqrt{2\pi }}\right)ds.
\end{align*}
By definition, for $s=c+iv$, $v\in \RR$, $\mathscr Re s\in (\beta ', 2)$,
\begin{align*}
 \mathscr M( \varphi) (1-s)&=\int _0^\infty \varphi (x)x^{-s}dx=\frac {1} {(1-s)(2-s)}\int _0^\infty \varphi ''(x)x^{2-s}dx \le \frac {C} {1+|s|^2}.
\end{align*}
As we have seen above (cf. (\ref{S5EU0})), for $\mathscr Re s\in (\beta ', 2)$,
\begin{align*}
\left|U(t, s)-\frac {1} {\sqrt{2\pi }}\right|=\frac {|B(s)|} {\sqrt{2\pi}}\left|\int  _{ \mathscr Re(\sigma )=\beta'  }\frac {t^{-(\sigma -s)}\Gamma (\sigma-s) } { B(\sigma )}d\sigma\right| ,\,\,\, \beta '\in (0, \mathscr Re s)\\
\le \frac {|B(s)|} {\sqrt{2\pi}}t^{\mathscr Re (s-\beta ')}\left|\int  _{ \mathscr Re(\sigma )=\beta'  }\frac {t^{-(i\mathscr Im (\sigma -s))}\Gamma (\sigma-s) } { B(\sigma )}d\sigma\right|\\
\le C e^{\mathscr Re (s-\beta ')\log t} \log |s| 
\end{align*}
Then, for $\mathscr Re s=c>\beta '$:
\begin{align*}
& \left|\langle \Lambda (t), \varphi \rangle -\varphi (1)\right|=\frac {1} {2i\pi }\left| \int  _{ c-i\infty }^{c+i\infty}\mathscr M( \varphi) (1-s)\left(U(t, s)-\frac {1} {\sqrt{2\pi }}\right)ds\right|\\
& \le C e^{\mathscr  (c-\beta ')\log t} \int  _{ c-i\infty }^{c+i\infty}|\mathscr M( \varphi) (1-s)|  \log |s||ds|\le C e^{\mathscr  (c-\beta ')\log t} \int  _{ \RR}  \frac {\log |v|\, dv} {1+|v|^2}
\underset{t\to 0}{\to} 0.
\end{align*}
\end{proof}

\section{Further Properties of $\Lambda$}
\setcounter{equation}{0}
\setcounter{theo}{0}

The resolution of the initial value problem for equation (\ref{S4E1}) requires yet  several  estimates on the fundamental solution  $\Lambda$. The following notation will be used,
\begin{align}
&\rho (\sigma )={\rm Res}\left(\frac {1} {B(s)}, s=\sigma  \right),\,\,\, r(\sigma) ={\rm Res} (B(s), s=\sigma ) \label{S·res1}\\
&\tilde r(\sigma) ={\rm Res} (s^{-2}B(s), s=\sigma )\label{S·res3}\\
&P(n)={\rm Res}\left( \frac {\Gamma (\omega )} {B(\omega )}, \omega =-n\right),\,\,
Q(n)={\rm Res}\left( \frac {\Gamma (\omega +1)} {B(\omega )}, \omega =-n\right)=-nP(n). \label{S·res2}
\end{align}
Notice that $-n$ is a simple  pole of $ \frac {\Gamma (\omega )} {B(\omega )}$  for $n\in \{0, \cdots 5\}$ and is a double pole for $n\ge 6$.

\subsection{Behavior of $\Lambda$ for $t>1$.}
\label{S4larget}
\begin{prop}
\label{S4L1} For all $t>1$, 
\begin{align}
&\Lambda(t, x)=t^{-3}Q _1(\theta)+Q_2(t, \theta),\,\,\theta =\frac {x} {t} \label{S4L1E1}\\
&Q_1(\theta)=\frac {c_1} {2i\pi \sqrt{2\pi }}\int  _{ {\mathscr Re}(s)=c } \theta^{-s }B(s)\Gamma (3 -s) ds \label{S4L1E2}\\
&Q_2(t, \theta)=-\frac {1} {4\pi^2 \sqrt{2\pi }}\int  _{ {\mathscr Re}(s)=c } \theta^{-s } \int  _{ {\mathscr Re} (\sigma ) =\beta _2 }\frac {B(s)} {B(\sigma  )}\Gamma (\sigma  -s)t^{-\sigma  }d\sigma  ds \label{S4L1E3}\\
&c_1=-\frac {1} { B(1)W(1)W'(2)},\,\,\,\beta _2>3. \label{S4P1E2}
\end{align}
\end{prop}

\begin{proof}
By  (\ref{S5EUX1})  and   (\ref{S5CInvME2B}) at $x=t\theta$,
\begin{align}
\label{S4L1E1MM}
\Lambda(t, x)=\frac {1} {4 \pi^2\sqrt{2\pi }}\int  _{ {\mathscr Re}(s)=c } \theta^{-s } \int  _{ \mathscr Re \sigma  =\beta  }\frac {B(s)} {B(\sigma  )}\Gamma (\sigma  -s)t^{-\sigma  }d\sigma  ds
\end{align}
We deform the  $\sigma $-integration contour to larger values of ${\mathscr Re}(\sigma ) $ and cross the zero of $B(\sigma  )$ at $\sigma =3$.  Since  $Res\!\left(B(\sigma  )^{-1}; \sigma  =3 \right)= (B(1) W(1)W'(2))^{-1}$, we deduce the Lemma.
\end{proof}

\begin{prop}
\label{S4L2} 
For all $\varepsilon >0$ as small as wished,
\begin{align}
&Q_1(\theta)=\frac {2c_1B(1)} {W'(0)}+\mathcal O_\varepsilon \left(|\theta|^{1-\varepsilon }\right)\,\,\,as\,\,\theta \to 0, \label{S4L2E1}\\
&Q_1(\theta)=c_1\theta^{-3}B(3)+\mathcal O_\varepsilon  \left(|\theta|^{-4+\varepsilon } \right)\,\,\,as\,\,\theta \to \infty, \label{S4L2E2}
\end{align}
\end{prop}
\begin{proof}
For $\theta\to 0$ we deform the $s$-integration contour in (\ref{S4L1E2}) towards smaller values of ${\mathscr Re}(s)$ until we cross the first pole of the $B(s)$ located at ${\mathscr Re}(s)=0$. Since  ${\rm Res}(B(s), s=0))=-B(1)/W'(0)$ we deduce
\begin{align*}
Q _1(\theta)=-\frac {c_1\Gamma (3)B(1)} {W'(0)}+\frac {c_1} {2i\pi }\int  _{ {\mathscr Re}(s)=\alpha _2 } \theta^{-s }B(s)\Gamma (3 -s) ds
\end{align*}
where $\alpha _2\in (-1, 0)$ and then,
\begin{align*}
\left|\int  _{ {\mathscr Re}(s)=\alpha _2 } \theta^{-s }B(s)\Gamma (3 -s) ds\right|\le |\theta|^{-\alpha _2}\int  _{ {\mathscr Re}(s)=\alpha _2 }|B(s)|\Gamma (3-s)||ds|.
\end{align*}
Since $\Gamma (3)=2$, (\ref{S4L2E1}) follows. \\
For $\theta\to \infty$ we deform the $s$-integration contour in (\ref{S4L1E2}) towards larger values of ${\mathscr Re}(s)$ until we cross the first pole of $\Gamma (3-s)$ located at ${\mathscr Re}(s)=3$. It follows,
\begin{align*}
Q_1(\theta)=c_1\theta^{-3}B(3)+\frac {c_1} {2i\pi }\int  _{ {\mathscr Re}(s)=\alpha _3 } \theta^{-s }B(s)\Gamma (3 -s) ds
\end{align*}
with $\alpha _3\in (3, 4)$ and then,
\begin{align*}
\left|\int  _{ {\mathscr Re}(s)=\alpha _3 } \theta^{-s }B(s)\Gamma (3 -s) ds\right|\le |\theta|^{-\alpha _3}\int  _{ {\mathscr Re}(s)=\alpha _3 } |B(s)|\Gamma (3 -s) |ds|
\end{align*}
\end{proof}

\begin{prop}
\label{S3L3}
For any $\delta >0$ as small as desired,
\begin{align}
&Q_2(t, \theta)=c_2t^{-4}+b_1(t) +\mathcal O\left(t^{-8}|\theta|^{1-\delta } \right) +\mathcal O\left( |\theta|^{1-\delta } t^{-8-\delta }\right)\,\,\, as\,\,\,\theta  \to 0,\label{S3L3E1}\\
&Q_2(t, \theta)=c_3t^{-4}\theta^{-5}+\mathcal O\left(  |\theta|^{-5-\delta } t^{-4}\right)+\mathcal O\left(  |\theta|^{-5 } t^{-4-\delta }\right)\,\,\, as\,\,\,\theta  \to \infty,\label{S3L3E2}
\end{align}
with
\begin{align*}
b_1(t)=\mathcal O\left(t^{-8 } \right),\,\,t>1;\,\,\,
c_2=-\frac {6\rho(4)} {\sqrt{2\pi }}\frac {B(1)} {W'(0)},\,\,c_3=\frac {B(5)} {\sqrt{2\pi }}\rho(4).
\end{align*}

\end{prop}
\begin{proof} 
We deform the $\sigma  $-integration contour to larger values of ${\mathscr Re} (\sigma ) $, cross the zero of $B(\sigma  )$ at $\sigma  =4$ to obtain
\begin{align*}
&Q_2(t, \theta)=\alpha (\theta)t^{-4}+R_1(t, \theta); \,\,\,\,\alpha (\theta)=\frac{\rho(4 )}{2i\pi} \int  _{ {\mathscr Re}(s) =c }\theta^{-s}B(s)\Gamma (4-s)ds,\\
&R_1(t, \theta)=\frac {1} {4\pi^2 }\int  _{ {\mathscr Re}(s)=c } \theta^{-s } \int  _{ {\mathscr Re}( \sigma ) =4+\delta  }
\frac {B(s)} {B(\sigma  )}\Gamma (\sigma  -s)t^{-\sigma  }d\sigma  ds.
\end{align*}
If $\theta\in (0, 1)$, we use the pole of $B(s)$ at $s=0$ and obtain
\begin{align}
\alpha (\theta)=-\frac {\rho(4)\Gamma (4)B(1)} {\sqrt{2\pi }W'(0)}+\mathcal O\left(|\theta|^{1-\delta } \right),\,\,\,\theta\in (0, 1).\label{S3EBB01}
\end{align}
Then,
\begin{align}
&R_1(t, \theta)=b_1 (t)+\mathcal O\left(|\theta|^{1-\delta }t^{-8 } \right),\,\,\theta \in (0, 1),\,\,t>1.\nonumber \\
&b_1 (t)=-\frac {1} {2i\pi } \frac {B(1)} {W'(0)}\int  _{ {\mathscr Re} (\sigma ) =4+\delta  }
\frac { \Gamma (\sigma  )t^{-\sigma  }} {B(\sigma  )}d\sigma ,\,\,\,|b _1(t)|\le C t^{-8},\,\,t>1.\label{S3EBB1}
\end{align}
and (\ref{S3L3E1}) follows.
Suppose now that $\theta >1$. We use the pole of $\Gamma (4-s))$ at $s=5$ (the points $s=4$ is a zero of $B$)  in the expression of  $\alpha (\theta)$, 
\begin{align*}
\alpha (\theta)=\frac {\theta^{-5}B(5)} {\sqrt{2\pi }}\rho(4)+\mathcal O\left( \theta^{-5-\delta }\right)
\end{align*}
The order of the remainder term comes from the  pole at $s=\sigma _1+2$ of the Gamma function.
On the other hand, using the pole of $B(s)$ at $s=5$ in the expression of $R_1$,
\begin{align*}
R_1(t, \theta)=\mathcal O\left(|\theta|^{-5}t^{-4-\delta } \right),\,t>1, \theta>1
\end{align*}
\end{proof}
\begin{prop}
\label{S3PDL}
For $t>1$,
\begin{align}
&\frac {\partial \Lambda (t, x)} {\partial x} = 6c_1r(-1) t^{-4}+\mathcal O\left(t^{-4}\left|\frac {x} {t}\right|^{\delta } \right)  \,\,\, as\,\,\,\frac {x} {t} \to 0,\label{S3L3DE1}\\
& \frac {\partial \Lambda (t, x)} {\partial x} =c_13B(3) x^{-4}+\mathcal O\left(t^{-4}\left|\frac {x} {t}\right|^{-4-\varepsilon  } \right)  \,\,\, as\,\,\,\frac {x} {t} \to \infty ,\label{S3L3DE2}\\
&\left| \frac {\partial \Lambda (t, x)} {\partial t}\right|\le Ct^{-4},\,\,\forall x\in (0, t/2)\label{S3L3DE2X}\\
&\left| \frac {\partial \Lambda (t, x)} {\partial t}\right|\le  Cx^{-4},\,\,\forall x>2t.\label{S3L3DE2Y}
\end{align}
\end{prop}
\begin{proof}
Since $t>1$, by  (\ref{S5EUX1})  and   (\ref{S5CInvME2B}).
\begin{align*}
\frac {\partial \Lambda (t, x)} {\partial x}= \frac {-1} {4 \pi^2\sqrt{2\pi }}\int  _{ {\mathscr Re}(s)=c } \theta^{-s -1} \int  _{ \mathscr Re \sigma  =\beta  }\frac {sB(s)} {B(\sigma  )}\Gamma (\sigma  -s)t^{-\sigma -1 }d\sigma  ds
\end{align*}
Estimates (\ref{S3L3DE1}, (\ref{S3L3DE2}) follow now from exactly the same contour deformation arguments as in the proofs of Propositions \ref{S4L1}, \ref{S4L2} and \ref{S3L3}. On the other hand, by (\ref{S5CInvME2B}) and (\ref{S5EUXYp}),
\begin{align}
\frac {\partial } {\partial t}\Lambda (t, x)&= \frac {1} {2\pi i}\int  _{ c-i\infty }^{c+i\infty} U(t, s-1)W(s-1)x^{-s}ds \nonumber\\
&= \frac {-x^{-1}} {4\pi^2 \sqrt{2\pi }}\int  _{ c-i\infty }^{c+i\infty}\int  _{ \mathscr Re(\sigma )=\beta  }\frac {B(s) \Gamma (\sigma-s+1) } { B(\sigma )}t^{-\sigma}\left(\frac {x} {t} \right)^{-(s-1)}d\sigma ds.\label{S4EPArtT}
\end{align}
When $\theta <1/2$, deformation of the $\sigma $  integration contours towards  larger values of $\mathscr Re(\sigma )$  and of the   $s$ integration contour towards smaller values  of  $\mathscr Re (s)$ give, due to the zero of $B(\sigma )$ at $\sigma =3$ and the pole of $B(s)$ at $s=0$, the existence of a positive constant $C$ such that
\begin{align*}
\left|\frac {\partial } {\partial t}\Lambda (t, x)\right|\le Cx^{-1}t^{-3 }\theta=\frac {C} {t^4},\,\,\forall t>1, \forall x\in (0, t/2).
\end{align*}
For $\theta >1/2$ we   first deform the $\sigma $  integration contour towards  larger values  of $\mathscr Re(\sigma )$ and then  the   $s$ integration contour is deformed towards larger values  of  $\mathscr Re (s)$. In the  first  step we meet again  the pole  of $B(\sigma )$ again at $\sigma =3$. Then, in the second step   the pole of $\Gamma (4-s)$ at  $s=4$ is met from where,
\begin{align*}
\left|\frac {\partial } {\partial t}\Lambda (t, x)\right|\le Cx^{-1}t^{-3 }\theta^{-3}=\frac {C} {x^4},\,\,\forall t>1, \forall x>2t.
\end{align*}
\end{proof}

\subsection{Behavior of $\Lambda$ for $t\in (0, 1)$. }
\label{S4smallt}
For all $t\in (0, 1)$ we split $[0, \infty)$ as,
\begin{equation*}
[0, \infty)\setminus \{1\}=[0, 1/2]\cup \{x>0;\, 0<|x-1|< 1/2 \}\cup [3/2, \infty).
\end{equation*}
By (\ref{S5CInvME03}), $\Lambda $ is continuous and bounded on $(0, 1)\times  [0, 1/2]$.

\subsubsection{Behavior of $\Lambda$ for $0<t<1$ and $|x-1|>1/2$}

\begin{prop}
\label{S4P9}
For $0<t<1$, and $\varepsilon >0$ as small as desired there exists $C_\varepsilon >0$,
\begin{align}
&|\Lambda (t, x)|\le C_\varepsilon x^{-3+\varepsilon }t^{9-\varepsilon }+C_2x^{-5}t^7,\,\,\,\forall x>3/2\label{S4P9E1}\\
&\left| \frac {\partial \Lambda} {\partial t}(t, x)\right|\le C_\varepsilon x^{-3+\varepsilon }t^{8-\varepsilon }+C_2x^{-5}t^6,\,\,\,\forall x>3/2.\label{S4P9E1B}\\
&\left| \frac {\partial \Lambda} {\partial t}(t, x)\right|\le Cxt^4,\,\,\,\forall x \in (0, t/2).\label{S4P9E1C}
\end{align}
\end{prop}

\begin{proof}
When $t\in (0, 1)$ we may start from (\ref{S5CInvME2}), (\ref{S5EUX1}) and consider then the integral,
\begin{align}
I (t, x)&= \frac {1} {4\pi ^2}\int  _{ c-i\infty }^{c+i\infty} 
\frac {B(s)} {\sqrt{2\pi}}\int  _{ \mathscr Re(\sigma )=\beta  }\frac {t^{-(\sigma -s)}\Gamma (\sigma-s) } { B(\sigma )}d\sigma
s^{-2}x^{-s}ds,\,\,\,\,0<c<\beta <2, \nonumber \\
&=\frac {1} {4\pi^2}\frac {1} {\sqrt{2\pi} }\int  _{ c-i\infty }^{c+i\infty} 
\int  _{ \mathscr Re(\omega  )=\beta  }\frac {t^{-\omega  }B(s)\Gamma (\omega -s) } { B(\omega  )}
s^{-2}\left(\frac {x} {t} \right)^{-s}d\omega  ds. \label{S4P9E1}
\end{align}

Since $x/t>1$ and $0<t<1$, in order to estimate the size of the integral in the right hand side of (\ref{S4P9E1}) it is natural to seek for large values of ${\mathscr Re}(s)$ and smaller values of ${\mathscr Re}(\omega )$. Let us then deform, at   $s$ fixed such that  ${\mathscr Re}(s)=c$,  the $\omega $-integration contour towards lower  values of $\omega $. Since we have taken $\beta >c$, the first singularity that is found is at the pole of $\Gamma (\omega -s)$ where  $\omega =s $. 
\begin{equation*}
\frac {1} {2i\pi }\int  _{ {\mathscr Re}(s)=c }s^{-2}x^{-s}ds=-H(1-x)\, \log(x)
\end{equation*}
we obtain, for  $\beta _1'\in (0, c)$ and $x>1$, or $x<1$,
\begin{align}
I(t, x)=\frac {1} {4\pi^2}\frac {1} {\sqrt{2\pi} } \int  _{ {\mathscr Re}(s)=c }\!\int  _{ {\mathscr Re} (\omega) =\beta' _1 }\frac {s^{-2}B(s)} {B(\omega )}\Gamma (\omega -s)\left(\frac {x} {t} \right)^{-s}t^{-\omega }d\omega ds \label{S5I1}
\end{align}

We let now $\beta _1'$ fixed and move $c$ towards larger values in the integral at the right hand side of (\ref{S5I1}). The function under the integral sign  is singular at  two different families of poles,
\begin{align}
s _{ 1, k }&=\beta _1'+k, \,\,k=1, 2, 3, \cdots\,\,\,(\hbox{poles of}\,\,\Gamma (\omega -s)\,\,\hbox{for}\,\,\mathscr Re s>\beta '_1),\label{S5I2}\\
s _{ 2, n }&=4n+1,\,\,n=1, 2, 3, \cdots \,\,\,(\hbox{poles of}\,\,B(s)).\label{S5I3}
\end{align}

\begin{align}
&\Lambda (t, x)=\left(x\frac {\partial } {\partial x} \right)^2\left( \mu(t)\sum _{ k=1 }^\infty \left(\frac {x} {t}\right)^{-k-\beta '_1}A_k+\sum  _{ n=1 }^\infty \left(\frac {x} {t}\right)^{-4n-1}\nu _n(t)\right),\,\, \frac {x} {t}>1 \label{S31.1Lamb1}\\
&= \mu(t)\sum _{ k=1 }^\infty \left(\frac {x} {t}\right)^{-k-\beta '_1}A_k (k+\beta _1')^2+\sum  _{ n=1 }^\infty \left(\frac {x} {t}\right)^{-4n-1} (4n+1)^2\nu _n(t)\\
&A_k(t)=(-1)^k\frac {(\beta '_1 +k)^{-2}B(\beta '_1 +k)} {\sqrt{2\pi }\, k!};  \,\,\, \mu (t)=\frac {1} {2i\pi }\int  _{ {\mathscr Re} (\omega ) =\beta' _1 }\frac {t^{-\omega }} {B(\omega  )}d\omega  \label{S31.1Lamb2}\\
&\nu_n(t)=\frac {\tilde r_{4n+1}} {\sqrt{2\pi }} \frac {1} {2i\pi }\int  _{ {\mathscr Re} (\omega ) =\beta ' _1 }\frac {\Gamma (\omega -4n-1)} {B(\omega  )}t^{-\omega }d\omega  \label{S31.1Lamb3}
\end{align}

In order to  estimate $\mu (t)$ for $0<t<1$ we deform the integration contour $\mathscr Re \omega =\beta '_1$ towards  lower values of $\mathscr Re \omega $. Since $\beta '_1\in (0, c)$, 
the  singularities are the negative  zeros of $B(\omega )$, $s=-n, n=-6, -7, -8,\cdots$ and
\begin{equation}
\label{S5I6}
\mu(t)=\sum _{ n=6 }^\infty \rho (-n)t^n.
\end{equation}
On the other hand, for each $n\in \NN$, the set of poles of $\Gamma (\omega -4n-1)$ such that ${\mathscr Re}(\omega )<\beta_2'$ is $\left\{-1,  -2, -3, -4, \cdots\right\}$, but $-1$ is a pole of $B(\omega )$ too.  The  zeros of $B(\omega )$ are the negative integers  $\left\{-6, -7, -8, \cdots\right\}$. Therefore, the singularities of $\frac {\Gamma (\omega -4n-1)} {B(\omega  )}$ are the simple poles $\left\{-2, -3, -4, -5\right\}$ and the poles 
$\left\{-6, -7, -8, \cdots \right\}$ of multiplicity two,
\begin{align}
&\nu_n(t)=-\frac {\tilde r_{4n+1}} {\sqrt{2\pi }}\sum _{ \ell=2 }^\infty  \gamma   _{ n, \ell } t^{\ell } \label{S3.1.1nu1}\\
&\gamma   _{ n, \ell }=\frac {(-1)^{\ell+4n+1}} {B(-\ell  )(\ell+4n+1) !},\,\,\,\ell=1, 2, \cdots, 5 \label{S3.1.1nu2}\\
&\gamma   _{ n, \ell }={\rm Res}\left(\frac {\Gamma (\omega -4n-1)} {B(\omega  )}; \omega =-\ell \right),\,\,\ell=6, 7, \cdots \label{S3.1.1nu3}
\end{align}
It follows that,
\begin{align*}
|\Lambda (t, x)|\le C_1 \left(\frac {x} {t} \right)^{-1-\beta '_1}t^6+C_2\left(\frac {x} {t} \right)^{-5}t^2,\,\,\,\frac {x} {t}>1, 0<t<1.
\end{align*}
Since $\beta _1'$ is arbitrary in $(0, c)$ and $c$ is arbitrary in $(0, 2)$, $\beta _1'$ may be taken as close to $2$ as desired. The estimate (\ref{S4P9E1B}) follows from similar arguments. Starting from (\ref{S5CInvME2}) and (\ref{S5EUXYp}), we deduce
\begin{align*}
&\frac {\partial } {\partial t}\Lambda (t, x)=\left(x\frac {\partial } {\partial x} \right)^2\left( \frac {1} {2\pi i}\int  _{ c-i\infty }^{c+i\infty} U(t, s-1)W(s-1)s^{-2}x^{-s}ds\right)\\
&=-\left(x\frac {\partial } {\partial x} \right)^2\left(\frac {x^{-1}} {4\pi^2 \sqrt{2\pi }}\int  _{ c-i\infty }^{c+i\infty}\int  _{ \mathscr Re(\sigma )=\beta  }\frac {B(s) \Gamma (\sigma-s+1) } { B(\sigma )}s^{-2}t^{-\sigma}\left(\frac {x} {t} \right)^{-(s-1)}d\sigma ds\right).
\end{align*}
With the same argument as before we deduce,
\begin{align*}
&\frac {\partial } {\partial t}\Lambda (t, x)= x^{-1}\mu(t)\sum _{ k=1 }^\infty \left(\frac {x} {t}\right)^{-k-\beta '_1+1}A_k (k+\beta _1')^2+x^{-1}\sum  _{ n=1 }^\infty \left(\frac {x} {t}\right)^{-4n} (4n+1)^2\nu _n(t)
\end{align*}
and,
\begin{align*}
\left|\frac {\partial } {\partial t}\Lambda (t, x)\right|\le C_1x^{-1} \left(\frac {x} {t} \right)^{-\beta '_1}t^6+C_2x^{-1}\left(\frac {x} {t} \right)^{-4}t^2,\,\,\,\frac {x} {t}>1, 0<t<1.
\end{align*}
If $x\in (0, t/2)$  the $s$  integration contour is moved towards smaller values of $\mathscr Re(s)$. In that  process, the sequence of poles of $B(s)$, with $\mathscr Re (s)\le 0$ is crossed. These are located at  $s=0, -1$ and points $\sigma _n^*$ defined in Proposition (\ref{S3PW}). We deduce, arguing as before

\begin{align*}
&\partial _t\Lambda (t, x)=\left(x\frac {\partial } {\partial x} \right)^2\left( \frac {1} {t} \tilde \mu _1(t)+\tilde \mu _2(t)\left(\frac {x} {t^2}\right)+\sum  _{ n=0 }^\infty \left(\frac {x} {t}\right)^{-\sigma _n^*} \tilde \nu _n(t)\right),\,\,  \\
&=\tilde \mu _2(t)\left(\frac {x} {t^2}\right)+\sum  _{ n=0 }^\infty (\sigma _n^*)^2\left(\frac {x} {t}\right)^{-\sigma _n^*} \tilde \nu _n(t)\\
&\tilde \nu_n(t)=\frac {\tilde r_{\sigma _n^*}} {\sqrt{2\pi }} \frac {1} {2i\pi }\int  _{ {\mathscr Re} (\omega ) =\beta  }\frac {\Gamma (\omega -\sigma _n^*)} {B(\omega  )}t^{-\omega }d\omega\\
& \tilde \mu _2(t)=\frac {\tilde r_{-1}} {\sqrt{2\pi }\, 2i\pi } \int  _{ \mathscr Re(\sigma )=\beta  }\frac { \Gamma (\sigma+2) } { B(\sigma )}t^{-\sigma} d\sigma
\end{align*}
The functions $\tilde \nu_n$ and $\tilde \mu _2$ are now determined by the sequence of zeros of $B(\sigma )$ such that $\mathscr Re(\sigma )\le 0$. Since the the fist one is at $s=6$ (\ref{S4P9E1B}) follows.
\end{proof}

\subsubsection{Behavior of $\Lambda$ for $t\in (0, 1)$ with and $0< |x-1|\le 1/2$.}
\begin{prop}
\label{S3.3.4P1}
There exists a constant $C >0$ such that
\begin{equation*}
 \Lambda (t, x)\le \frac {Ct} {|x-1|},\,\,\,\forall x;\,\,0<|1-x |<1/2,\,\forall t\in (0, 1).
\end{equation*}

\end{prop}
\begin{proof}
We define the new variables 
\begin{align}
&X=\log x, \,\,\,\tilde \Lambda (t, X)=\Lambda (t, x),\,\,\forall t>0, x>0. \label{S3ELX}
\end{align}
Then, 

\begin{align}
\forall X\in \RR, \,\,\,\tilde \Lambda (t, X)=\frac {1} {2i\pi }\int _{ {\mathscr Re}(s)=c } e^{-s X}U(t, s)ds.  \label{S3ELXZ}
\end{align}
After two integrations by parts:
\begin{align}
\label{S3IBP1}
\tilde \Lambda (t, X)=\frac {1} {X^2}\int_{ {\mathscr Re}(s)=c } \left(e^{-s X}-1\right)\frac {\partial ^2 U} {\partial s^2}(t, s)ds.
\end{align}
When $|s|<1$, it follows that $|sX|<1/2$, $|e^{-sX}-1|=|sX|(1+\mathcal O(|sX|)$ and
$$
|e^{-s X}-1|=|sX|(1+\mathcal O(|sX|)\le C _1 |sX|.
$$
We deduce from (\ref{S3IBP1}) and  Proposition \ref{S3PX1} 
\begin{align*}
\left|\tilde \Lambda (t, X)\right|\le &\frac {t} {|X|} \int\limits_{\substack{ {\mathscr Re}(s)=c  \\ |s|<1 } } \frac {s\, |ds|} {1+|s|^2}+\left|\frac {1} {X^2}\int\limits_{\substack{ {\mathscr Re}(s)=c  \\ |s|>1 } } \left(e^{-s X}-1\right) \frac {\partial ^2 U} {\partial s^2}(t, s)ds\right|.
\end{align*}
But,
\begin{align*}
\int\limits_{\substack{ {\mathscr Re}(s)=c \\ |s|>1 } } \left(e^{-s X}-1\right) \frac {\partial ^2 U} {\partial s^2}(t, s)ds
=\frac {1} {X}\int\limits_{\substack{ {\mathscr Re}(u)=c X \\ |u|>|X| } }\hskip -0.5cm \left(e^{-u}-1\right) \frac {\partial ^2 U} {\partial s^2}\left(t, \frac {u} {X}\right)du
\end{align*}
and by Proposition \ref{S3PX1}
\begin{align*}
&\left|\int\limits_{\substack{ {\mathscr Re}(s)=c  \\ |s|>1} }\frac { \left(e^{-s X}-1\right)} {1+|s|^2}ds\right|
\le \frac {t} {|X|}\int\limits_{\substack{ {\mathscr Re}(u)=c X \\ |u|>|X| } }\frac { \left|e^{-u}-1\right|} {1+|u/X|^2}|du|\\
&=t|X|\int\limits_{\substack{ {\mathscr Re}(u)=c X \\ |u|>|X| } }\frac { \left|e^{-u}-1\right|} {|X|^2+|u|^2}|du|<
t|X|\int\limits_{\substack{ {\mathscr Re}(u)=c X \\ |u|>|X| } }\frac { \left|e^{-u}-1\right|} {|u|^2}|du|.
\end{align*}
If $s=c+iv$, then $e^{-u}=e^{-cX}e^{-iv}$,

\begin{align*}
\left|e^{-u}-1\right|^2=e^{-2cX}\left((\cos^2(vX)-1)+\sin^2 (vX) \right)\le 2e^{c}
\end{align*}

and, if $u=c X+iw$,
\begin{align*}
\int\limits_{\substack{ {\mathscr Re}(u)=c X \\ |u|>|X| } }\frac { \left|e^{-u}-1\right|} {|u|^2}|du|\le 
\sqrt 2e^{c}\hskip -1cm \int\limits_{\substack{ {\mathscr Re}(u)=c X \\ c^2|X|^2+w^2>|X|^2 } }\hskip -0.5cm\frac {dw} {c^2+w^2}\le C\int  _{ \RR }\frac {dw} {c^2+w^2}.
\end{align*}
\end{proof}

\subsection{Behavior of $\Lambda$ as $x\to 1$.}
\label{t0x1}
The following Proposition describes the convergence to the initial data. Its proof, rather long and somewhat technical is given in the Appendix.
\begin{prop}
\label{S4P2}
Uniformly for $X$ in bounded subsets of $\RR$ 
\begin{align}
&\lim _{ t\to 0 }t^{-1} |X|^{1-2t}\tilde \Lambda (t, X)
=1. \label{S3.E46}\\
&\lim _{ t\to 0 }\frac { |X|^{1-2t}} {(1+2t\log|X|)} \frac {\partial \tilde \Lambda } {\partial t}(t, X)=1
=1. \label{S3.E46B}
\end{align}
\end{prop}

\begin{rem} 
\label{S4LP2}
For any $\varphi  \in C_C(\RR)$,
$$
\lim _{ t\to 0 } t \int  _{ \RR }   |X|^{-1+2t} \varphi (X)dX=\varphi (0).
$$
\end{rem}

\begin{cor}
\label{S3Cor17}
\begin{align}
&\lim _{ t\to 0 }\, t^{-1}\left|e^{-1/t} Y\right|^{1-2t}\Lambda \left(t, 1+e^{-1/t}Y\right)=1 \label{S3Cor17E1}
\end{align}
uniformly  on bounded subsets of $\RR$.
\end{cor}
\begin{proof} For $t>0$ sufficiently small, depending on the bounded set  $K$ of $\RR$ where $Y$ varies, $1+e^{-1/t}Y>0$. Then we define  $1+e^{-1/t}Y=e^X$ and by definition $\Lambda(t, 1+e^{-1/t}Y)=\tilde \Lambda \left(t, X\right)$.  By (\ref{S3.E46}), uniformly for $X$ in bounded subsets of $\RR$, 

\begin{align}
&\lim _{ t\to 0 }t^{-1} |X|^{2t-1}\tilde \Lambda (t, X)=1\label{S3Cor17E1}\\
&\lim _{ t\to 0 }t^{-1} |X|^{2t-1}\tilde \Lambda(t, 1+e^{-1/t}Y)=1\label{S3Cor17E2}
\end{align}
But, since 
$$
\lim _{ t\to 0  }e^{-1/t}Y= 0,\, \hbox{uniformly for $Y$ on}\,\, K,
$$
it follows that 
$$
\lim _{ t\to 0  }e^X=1,\, \hbox{uniformly for $Y$ on}\,\, K.
$$
Then 
$$
\lim _{ t\to 0 }\frac {e^{-1/t}Y} {X}=\lim _{ t\to 0 }\frac {e^X-1} {X}=1
$$
from where
\begin{align}
\lim _{ t\to 0 }t^{-1} |X|^{2t-1}\tilde \Lambda(t, 1+e^{-1/t}Y)=&\lim _{ t\to 0 }t^{-1} |e^{-1/t}Y|^{2t-1}\tilde \Lambda(t, 1+e^{-1/t}Y)=1\label{S3Cor17E3}
\end{align}
uniformly for $Y\in K$.
\end{proof}

\begin{cor}
\label{S3Cor18}
For all bounded subset $K\in \RR$, There exists  $\tau >0$ such that for $t\in (0, \tau )$,
\begin{align}
|\Lambda \left(t, x\right)|\le \frac {2t} {\left|x-1 \right|^{1-2t}} ,\,\,\forall x;\,\,(x-1)e^{1/t}\in K.\label{S3Cor18E1}
\end{align}
\end{cor}
\begin{proof}
By Corollary \ref{S3Cor18}, for any bounded $K$ there is $\tau >0$ small enough such that for all $t\in (0, \tau )$,
\begin{align}
|\Lambda \left(t, 1+e^{-1/t}Y\right)|\le \frac {2t} {\left|e^{-1/t} Y\right|^{1-2t}} ,\,\,\forall Y\in K.\label{S3Cor18E2}
\end{align}
In terms of  $x=1+e^{-1/t}Y$, (\ref{S3Cor18E1}) follows.
\end{proof}

\begin{cor}
\label{S4C426}
The function $\Lambda$ satisfies,
\begin{align}
\label{S4C426E1}
&\Lambda \in C((0, \infty), L^1(0, \infty)),
\end{align}
and there exists $C>0$ such that,
\begin{equation}
\label{S4C426E2}
||\Lambda(t)||_1\le \frac {C} {1+t^2},\,\,\forall t>0
\end{equation}
\end{cor}

\begin{proof} 
We prove (\ref{S4C426E2}) first. For $t\in (0, 1)$ we use the estimates in Section \ref{S4smallt}
\begin{align*}
\int _0^\infty |\Lambda (t, x)|dx=\int _0^{1/2} |\Lambda (t, x)|dx+\int  _{ |x-1|<1/2 } |\Lambda (t, x)|dx+\int  _{ 3/2 }^\infty |\Lambda (t, x)|dx.
\end{align*}
\begin{align*}
&\int _0^{1/2} |\Lambda (t, x)|dx\le t\int _0^{1/2}\frac {dx} {|x-1|}\le t. \nonumber \\
&\int  _{ 3/2 }^\infty |\Lambda (t, x)|dx\le C_1t^{7+\beta _1' }\int  _{ 3/2 }^\infty
x^{-1-\beta '_1}dx+C_2t^7 \int  _{ 3/2 }^\infty x^{-6}dx. 
\end{align*}

\begin{align*}
\int  _{ |x-1|<1/2 } |\Lambda (t, x)|dx&=\int  _{ 0<|x-1|<e^{-1/t} } |\Lambda (t, x)|dx+\int  _{ e^{-1/t}<|x-1|<1/2 }  |\Lambda (t, x)|\\
&\le Ct\int  _{ 0<|x-1|<e^{-1/t} }\frac {dx} { |x-1|^{1-t}}+Ct\int  _{ e^{-1/t}<|x-1|<1/2 } \frac {dx} { |x-1|}\\
&=2Ct\int _0^{e^{-1/t}} \frac {dz} {z^{1-t}}+2Ct\int _{e^{-1/t}}^{1/2} \frac {dz} {z}=\frac {2C} {e}-2Ct\log 2+2C.
\end{align*}

For $t>1$, we use the estimates in Section \ref{S4larget}. By Proposition (\ref {S4L1})
\begin{align*}
\int _0^\infty |\Lambda(t, x)|dx&=t^{-3}\int _0^\infty |Q _1(\theta)|dx+\int _0^\infty |Q_2(t, \theta)|dx,\,\,\theta =\frac {x} {t} \\
&=t^{-2}\int _0^\infty |Q _1(\theta)|d\theta+t\int _0^\infty |Q_2(t, \theta)|d\theta. 
\end{align*}
Then (\ref{S4C426E2})  follows since, by Proposition (\ref{S4L2}), $Q_1\in L^1(0, \infty)$ and, by Proposition (\ref{S3L3}),
\begin{equation}
\int _0^\infty |Q_2(t, \theta)|d\theta\le C t^{-4} \label{S4C426E4}
\end{equation}
On the other hand if $t_1>0$ and $|t-t_1|<t_1/4$, for any $\varepsilon >0$ small fixed and $R$ large to be fixed, 

\begin{align*}
&\int _0^\infty |\Lambda (t_1, x)-\Lambda (t, x)|=I_1+I_2+I_3+I_4\\
&I_1=\int _0^{1-\varepsilon } |\Lambda (t_1, x)-\Lambda (t, x)|dx\le \sup_ {\substack{x \in [0, 1-\varepsilon ) }} |\Lambda (t_1, x)-\Lambda (t, x)| \\
&I_2=\int _{1-\varepsilon }^{1+\varepsilon } |\Lambda (t_1, x)-\Lambda (t_2, x)|dx\le 
2\varepsilon  \sup_ {\substack{x \in [1-\varepsilon , 1+\varepsilon )\\ t\in \left(\frac {3t_1} {4} \frac {5t_1} {4}\right) }} |\Lambda (t, x)| \\
&I_3=\int _{1+\varepsilon }^R |\Lambda (t_1, x)-\Lambda (t_2, x)|dx \le  \sup_ {\substack{x \in [1+\varepsilon , R) }} |\Lambda (t_1, x)-\Lambda (t, x)| \\
&I_4=\int _R^\infty |\Lambda (t_1, x)-\Lambda (t_2, x)| dx\le  \int _R^\infty |\Lambda (t_1, x)|dx+\int _R^\infty |\Lambda (t_2, x)|dx
\end{align*}
The terms $I_1$, $I_2$ and $I_3$ tend to zero as $t\to t_1$ by the continuity of $(\log x )\Lambda (t, x)$ for $t>0$ and $x\in \RR^+\setminus \{1\}$.  If $0<t_1<1$, we deduce $I_4\le  CR^{-\beta '_1}$ from  an estimate  similar to (\ref{S4C426E1}) written for $R$ instead of $3/2$. For $t>1$, it follows from (\ref{S4L1E1}) and (\ref{S5EGamma}) that $I_4\le CR^{1-c}$ where $c$ may  by chosen in the interval $(0, 2)$. The choice $c\in (1, 2)$ ensures  that  for all $t>0$, $I_4\to 0$  when $R\to \infty$. This proves (\ref{S4C426E1}). 
\end{proof}

In order to check that $\Lambda$ satisfies (\ref{S3EMK2}) let us show first that $L(\Lambda(t))$  is well defined for all $t>0$ and $x>0$. When $t>1$ this follows from the regularity of the function $\Lambda (t)$.
\begin{prop}
\label{S3L314} 
$L(\Lambda) \in C((1, \infty)\times (0, \infty))$. For all $t>1$, there exists a constant $C>0$ such that 
$$L(\Lambda (t))(x) <\frac {C} {x t^2}\min\left(\frac {1} {t}, \frac {1} {x} \right),\,\,\forall x>0.$$
\end{prop}
\begin{proof}
For $t>2$, $\Lambda(t)\in C^1 (0, \infty)$ and by Propositions \ref{S4L1}--\ref{S3L3} 
$$
|\Lambda (t, x )|\le \min(t^{-3}, x^{-3}).
$$
Therefore, for every $x>0$, and $y\in (0, x/2)$
\begin{align*}
|\Lambda(t, y)-\Lambda(t, x)|K(x, y)\le Cx^{-2}\left(\min(t^{-3}, x^{-3})+\min(t^{-3}, y^{-3})\right)
\end{align*}
Then, if $x\in (x_0-\varepsilon , x_0+\varepsilon )$ for some $x_0>2\varepsilon >0$,
\begin{align*}
|\Lambda(t, y)-\Lambda(t, x)|K(x, y)\1 _{ 0<y<x/2 }\le  \frac{C(x_0-\varepsilon )^{2}\1 _{ 0<y<(x_0+\varepsilon )/2 }}{ \left(\min(t^{-3}, (x_0-\varepsilon )^{-3})+\min(t^{-3}, y^{-3})\right)}
\end{align*}
and since the right hand side belongs to $L^1(0, \infty)$ it follows that
\begin{align*}
&\int _0^{x/2}(\Lambda(t, y)-\Lambda(t, x))K(x, y)dy\in C(0, \infty).
\end{align*}
\begin{align*}
\hskip -1.8cm \hbox{Moreover}\qquad\,\, &\int \limits _{0}^{x/2} |\Lambda(t, y)-\Lambda(t, x)|K(x, y)dy\le  C\min(t^{-3}, x^{-3})x^{-1}+\\
&\hskip -0.3cm\quad +Cx^{-2}\int \limits _{0}^{x/2} \min(t^{-3}, y^{-3})dy
\le  C\min(t^{-3}, x^{-3})x^{-1}+\frac {C} {xt^2}\min(t^{-1}, x^{-1}).
\end{align*}
On the other hand, for $x>0$ and $y\ge 3x/2$,
\begin{align*}
 |\Lambda(t, y)-\Lambda(x)|K(x, y)\le  C\min(t^{-3}, x^{-3})y^{-2}+Cy^{-2} \min(t^{-3}, y^{-3})
\end{align*}
and if $x\in (x_0-\varepsilon , x_0+\varepsilon )$ for some $x_0>2\varepsilon >0$,
\begin{align*}
 |\Lambda(t, y)-\Lambda(x)|K(x, y)\1 _{ y\ge 3x/2 }\le  Cy^{-2}\Big(\min(t^{-3}, (x-x_0)^{-3})+\\
 +C \min(t^{-3}, y^{-3})\Big)\1 _{ y\ge 3(x-x_0)/2 }\in L^1(0, \infty). 
\end{align*}
It follows that 
$$
\int _{3x/2}^\infty(\Lambda(t, y)-\Lambda(t, x))K(x, y)dy\in C(0, \infty)
$$
and,
\begin{align*}
\int \limits _{3x/2}^{\infty }  |\Lambda(t, y)-\Lambda(x)|K(x, y)dy&\le  C\min(t^{-3}, x^{-3})x^{-1}+C\int \limits _{\frac {3x} {2}}^{\infty} \min(t^{-3}, y^{-3})\frac {dy} {y^2}\\
& \le  C\min(t^{-3}, x^{-3})x^{-1}.
\end{align*}
For all $x>0$, $y\in (x/2, 3x/2)$,
\begin{align*}
(\Lambda(t, y)-\Lambda(t, x))K(x, y)\le \sup _{ x/2\le y\le 3x/3  }\left| \frac {\partial \Lambda} {\partial x}(t, y) \right|\frac {1} {y}
\end{align*}
from where, as before we deduce first that $\int \limits _{x}^{3x/2 } (\Lambda(t, y)-\Lambda(t, x))K(x, y) dy\in C(0, \infty)$ and 
\begin{align*}
\int \limits _{x/2}^{3x/2 } (\Lambda(t, y)-\Lambda(t, x))K(x, y) dy&\le C\sup _{ x/2\le y\le 3x/3  }\left| \frac {\partial \Lambda} {\partial x}(t, y) \right|\le C\min(t^{-4}, x^{-4}).
\end{align*}
\end{proof}

For $0<t<2$ we use that, by (\ref{S5CAlph1}),  $\Lambda (t, x)(\log x)^\alpha $ is Holder of order $\rho>0 $ for some $\alpha >0$ and  $\rho $ that depend on $t$. 
\begin{align}
&\int _0^\infty (\Lambda(t, y)-\Lambda(t, x))K(x, y)dy = \Lambda(t, x)(\log x)^\alpha  I_1(x)+I_2(t, x) \label{S3Esplt}\\
&I_1(x)=\int _0^\infty \left(\frac {1} {(\log x)^\alpha } - \frac {1} {(\log y)^\alpha }\right)K(x, y)dy \nonumber\\
&I _2(t, x)=\int  _0^\infty  \left(\frac {\Lambda(t, x)(\log x)^\alpha - \Lambda(t, y)(\log y)^\alpha} {(\log y)^\alpha }\right)K(x, y)dy\nonumber 
\end{align}

\begin{lem}
\label{S3L89I1}
$(\log x ) I_1\in C(0, \infty)$ and for some  constant $C>0$, $|\log x|^\alpha \,| I_1(x)| \le C/x$ 
\end{lem}
\begin{proof} 
The continuity of $I_1$ only requires a uniform estimate on a small neighborhood of every $x>0$ of the function under the integral sign. The bound on  $(\log x)I_1(x)$ follows from point wise estimates of that same function.  The point wise estimates for  $x>0$ are written in detail below.  The uniform estimates on  small  neighborhood of   $x>0$ are deduced as  in the proof of  Lemma (\ref{S3L314}). For $x>0$,  the domain $(0, \infty)$ is split  in two subdomains,

\begin{align}
&I_1=\int   _{ |x-y|\ge x/2 }\left(\frac {1} {(\log x)^\alpha } - \frac {1} {(\log y)^\alpha }\right)K(x, y)dy+\\
&\qquad +\int   _{ |x-y|\ge x/2 }\left(\frac {1} {(\log x)^\alpha } - \frac {1} {(\log y)^\alpha }\right)K(x, y)dy=I _{ 1, 1 }(x)+I _{ 1, 2 }(t, x)\label{S3WWE01}
\end{align}
\begin{align}
&|I _{ 1, 1 }(x)|\le Cx\int\limits  _{ 3x/2 }^\infty\left(\frac {1} {|\log x|^\alpha }+
\frac {1} {|\log y|^\alpha } \right) \frac {dy}{y^3}+\nonumber\\
&\hskip 5cm +\frac {C} {x^2}\int\limits  _{ 0 }^{x/2} \!\! \left(\frac {1} {|\log x|^\alpha }+\frac {1} {|\log y|^\alpha } \right)dy \le \frac {C} {x|\log x|^\alpha } \label{S3WWE0}
\end{align}
The continuity of  $I _{ 1,2 }(t)$ follows as for  $I _{ 1,1 }$.   The mean value Theorem gives,
\begin{align}
\label{S3WWE02}
|I _{ 1,1 }| &=\left|\int\limits  _{ x/2 }^{3x/2}\!\!\frac {(\log x)^\alpha -(\log y)^\alpha   }{(\log x)^\alpha (\log y)^\alpha }K(x, y)dy\right|
\le \frac {C} {x^{1+\alpha } |\log x|^{\alpha}}\int \limits _{ x/2 }^{3x/2}\!\!\frac {dy} {|x-y|^{1-\alpha   } |\log y|^\alpha }
\end{align}
\begin{equation}
\label{S3WWE1}
\hskip -2.5cm \hbox{If}\,\,x>2,\,\hbox{or}\, 3x<2,\qquad \quad \int \limits _{ x/2 }^{3x/2}\!\!\frac {dy} {|x-y|^{1-\alpha   } |\log y|^\alpha }\le \frac {Cx^\alpha } {|\log x|^\alpha }\le \frac {Cx^\alpha } {1+|\log x|^\alpha }.
\end{equation}
If $x\in (2/3, 1)$,  but  a similar argument works for $x\in (1, 2)$, we use the binomial formula,

\begin{align}
&\int \limits _{ x/2 }^{3x/2}\!\!\frac {dy} {|x-y|^{1-\alpha   } |\log y|^\alpha } =\int\limits  _{ x/2 }^{x}\!\!\!... dy+\int\limits  _{ x }^{1}\!\!\!... dy+\int\limits  _{ 1 }^{\frac {3x} {2}}\!\!\!... dy \label{S3WWE2}\\
&\int \limits _{ x/2 }^{x}\!\!\frac {dy} {(x-y)^{1-\alpha   } |\log y|^\alpha }=\frac {1} {x^{1-\alpha }}\sum _{ n=0 }^\infty \binom{1-\alpha }{n}
\int \limits _{ x/2 }^{x}\frac {dy} { |\log y|^\alpha}\le \frac {C} {x^{1-\alpha }} \label{S3WWE3}
\end{align}
\begin{align}
& \int \limits _{ 1 }^{\frac {3x} {2}}\!\!\frac {dy} {(x-y)^{1-\alpha   } |\log y|^\alpha }\le \sum _{ n=0 }^\infty \binom{1-\alpha }{n}
\int \limits _{1 }^{\frac {3x} {2}}\frac {dy} { |\log y|^\alpha}\le  C \label{S3WWE4}\\
&\int \limits _{ x }^1\!\!\frac {dy} {(x-y)^{1-\alpha   } |\log y|^\alpha }\le \sum _{ n=0 }^\infty \binom{1-\alpha }{n}\int \limits _{ x }^1\!\!\frac {dy} {y^{\alpha -1} |\log y|^\alpha }\le C \label{S3WWE5}
\end{align}
Then, by (\ref{S3WWE1})-(\ref{S3WWE5}), for all $x>0$,
\begin{equation}
\int \limits _{ x/2 }^{3x/2}\!\!\frac {dy} {|x-y|^{1-\alpha   } |\log y|^\alpha }\le\frac {Cx^\alpha } {1+|\log x|^\alpha },\,\,\forall x>0.\label{S3WWE6}
\end{equation}
and Lemma follows  from (\ref{S3WWE0}), (\ref{S3WWE02}). 
\end{proof}

\begin{lem}
\label{S3L89I2}
For all $\alpha \in (0, 1)$, $I_2\in C\left(\left(\frac {1-\alpha } {2}, 1\right)\times (0, \infty)\right)$. For all $t\in \left(\frac {1-\alpha } {2}, 1\right)$  there is $C>0$ and $\varepsilon >0$ as small as wanted  such that $I_2(t, x)\le  \frac {C } {x^{1+\varepsilon }(1+|\log x|^\alpha )}$ for all $x>0$. 
\end{lem}
\begin{proof}
 From Proposition \ref{S4P9}, we deduce that if $t\in \left(\frac {1-\alpha } {2}, 1\right)$,  for $\varepsilon >0$ arbitrarily small
\begin{equation}
\label{S3EFG1}
\Lambda (t, x )|\log x|^\alpha \le C \frac {1+x|\log x|^\alpha } {1+x^{4-\varepsilon }}
\end{equation}
Then, if we denote $J(t, x, y)=\left|\frac {\Lambda(t, x)(\log x)^\alpha - \Lambda(t, y)(\log y)^\alpha} {(\log y)^\alpha }\right|K(x, y)$

\begin{align*}
\int \limits _{0}^{x/2}J(t, x, y)dy&\le  C\frac {1+x|\log x|^\alpha } {1+x^{4-\varepsilon }}x^{-2}
\int \limits _{0}^{x/2} \frac {dy} {|\log y|^\alpha }dy+Cx^{-2}\int \limits _{0}^{x/2}  \frac {(1+y|\log y|^\alpha )dy} {(1+y^{4-\varepsilon })|\log y|^\alpha }\\
&\le C\frac {1+x|\log x|^\alpha } {1+x^{4-\varepsilon }}x^{-1}\frac {1} {1+|\log x|^\alpha }+\frac {C} {x(1+x+|\log x|^\alpha )}\\
&\le \frac {C} {x+x^{5-\varepsilon }}+\frac {C} {x(1+x+|\log x|^\alpha )}\le \frac {C} {x+x^2}
\end{align*}
\begin{align*}
&\hskip -0.4cm \int \limits _{\frac {3x} {2}}^{\infty }J(t, x, y)dy\le  C\frac {1+x|\log x|^\alpha } {1+x^{4-\varepsilon }}
\int \limits _{\frac {3x} {2}}^{\infty} \frac {dy} {y^2|\log y|^\alpha }dy+C\int \limits _{\frac {3x} {2}}^{\infty}  \frac {(1+y|\log y|^\alpha )dy} {y^2(1+y^{4-\varepsilon })|\log y|^\alpha }\\
&\quad \le C\frac {1+x|\log x|^\alpha } {1+x^{4-\varepsilon }}\frac {1} {1+x|\log x|^\alpha }+\frac {C} {x(1+|\log x|^\alpha )}\le \frac {C} {x(1+|\log x|^\alpha+x^{4-\varepsilon } )}\\
\end{align*}
By the Holder property of $(\log x)^\alpha \Lambda (t, x)$, \ref{S5CAlph10}) and arguing as in (\ref{S3WWE1})-(\ref{S3WWE5}), 
\begin{align*}
&\int \limits _{x/2}^{3x/2}J(t, x, y) dy\le 
 \frac {C} {x^{1+\rho ' }} \int \limits _{x/2}^{3x/2} \frac {dy} {|x-y|^{1-\rho }|\log y|^\alpha }\le \frac {C } {x^{1+\rho ' -\rho }(1+|\log x|^\alpha )},\,\,\forall x>0
\end{align*}
Using (\ref{S5CAlph1}), the continuity of  $I_2$ follows  with the same argument as for $I_1$ .
\end{proof}
From (\ref{S3Esplt} ),  Lemma  \ref{S3L89I1}   and Lemma \ref{S3L89I2} we obtain,
\begin{cor}
\label{S3CorLL}
For all $\alpha \in (0, 1)$,  $(\log x)^\alpha L(\Lambda)\in C\left(\left(\frac {1-\alpha } {2}, 1\right)\times (0, \infty)\right)$ and for every $t\in \left(\frac {1-\alpha } {2}, 1\right)$ there exists a constant $C>0$ such that
$$
|L(\Lambda (t))(x)|\le \frac {C|\Lambda (t, x)|} {x}+\frac {C} {x^{1+\varepsilon }(1+|\log x|^\alpha )},\,\,\forall x>0.
$$
\end{cor}
\begin{prop} 
\label{S3ELfg}
\begin{align}
(\log x) \frac {\partial \Lambda } {\partial t}=(\log x) L(\Lambda)\,\,\,\hbox{in}\,\, C((0, \infty )\times (0, \infty))\label{S3ELfgE1}
\end{align}
\end{prop}
\begin{proof} If $t>1$, $\Lambda (t)\in C^1(0, \infty)$ and actually, by (\ref{S5PInv1EX0}) and  (\ref{S4EHL007}), satisfies  (\ref{S3EMK2}) for $t>1$ and $x>0$. For $\tau \in (0, 1)$ fixed and $t\in (\tau , 1)$,  $(\log x)^\alpha  \Lambda(t)\in H _{ {\rm loc} }^{\rho'} (0, \infty)$ by (\ref{S5CAlph1}), for $\alpha $ and $\rho' $ such that 
$0<1-\alpha <\rho' <2\tau$. Since $\Lambda(t)\in L^1(0, \infty)$ we also have $x^{-r }\Lambda(t)\in L^1(0, \infty)$ for any $r \in (0, 1)$.  Therefore, there exist  a sequence of regular functions $u_n\in C^1(0, \infty)$  such that, for  $\rho \in(0, \rho' )$ and $r \in (0, 1)$ both fixed
\begin{align}
&\lim _{ n\to \infty }||u_n-(\log x)^\alpha \Lambda|| _{ H^{\rho }(I) }=0,\,\,\,||u_n||_\infty\le C\label{S3EIL1}\\
&\lim _{ n\to \infty }\left|\left|\left(u_n -(\log x)^\alpha  \Lambda \right)x^{-r }\right|\right| _{ L^1(0, \infty) }=0\label{S3EIL2}
\end{align}
 for all $I\subset (0, \infty)$ compact.
By (\ref{S3EIL2}) there exists a function $h\in L^1$ and a subsequence still denoted $u_n$ such that $x^{-r}u_n(x)\le h(x)$ for a. e. $x>0$ and  $v_n=\frac {u_n} {(\log x)^\alpha } $ satisfies 
\begin{equation}
\lim _{ n\to \infty }\left|\left|v_n - \Lambda\right|\right| _{ L^1(0, \infty) }=0\,\,\hbox{and then}\,\,\,
\lim _{ n\to \infty }||H\ast v_n-H\ast \Lambda ||_1=0.
\end{equation}
since $H\in L^1(0, \infty)$. On the other hand, if $w_n=v_n-\Lambda $, the same splitting as  (\ref{S3Esplt}) gives,
\begin{align*}
\int _0^\infty |w_n(y)-w_n(x)|K(x, y)dy \le |w_n(x)(\log x)^\alpha|  I_1+I _{ 2, n },\\
I _{ 2, n }=\int _0^\infty \left|\frac {w_n(x)(\log x)^\alpha - w_n(y)(\log y)^\alpha} {(\log y)^\alpha }\right|K(x, y)dy
\end{align*}
with
\begin{align*}
\left| w_n(x)(\log x)^\alpha - w_n(y)(\log y)^\alpha \right|& \le \left| u_n(x) - u_n(y) \right|\
+\left| \Lambda (x)(\log x)^\alpha - \Lambda (y) (\log y)^\alpha \right|\\
& \le C|x-y|^{\rho }+\left| \Lambda (x)(\log x)^\alpha - \Lambda (y) (\log y)^\alpha \right|.
\end{align*}
By  the Lebesgue's convergence Theorem, for all $x>0$,
\begin{equation*}
\lim _{ n\to \infty }\int _0^\infty \left(v_n(y)-v_n(x)\right)K(x, y)dy=
\int _0^\infty (\Lambda (y)-\Lambda (x))K(x, y)dy.
\end{equation*}
It  follows from next Lemma that $L(v_n) \underset { n\to \infty }{\longrightarrow} L(\Lambda (t)$ in $\mathscr D'(0, \infty)$.
\end{proof}
\begin{lem}
For all interval $I=[a, b]  \subset (0, \infty)$, there exists a constant $C$ such that, $|L(v_n)(x)|\le C|\log x|^{-1}$ for all $x\in I$.
\end{lem}
\begin{proof}
We denote $K=[a/3 , 3b]$ and split $L(v_n)$  as in (\ref{S3Esplt}). Then, for some constant $C$,

\begin{align*}
|v_n(x)||\log x|^\alpha=|u_n(x)| \le C&,\,\,\forall x\ge 0,\quad \hbox{and then},\\
\int _0^{x/2}\frac {|u_n(y)-u_n(x)|} {|\log y|^\alpha }K(x, y)dy&\le \frac {C} {x^2}\int _0^{x/2}\frac {dy} {|\log y|^\alpha }\le \frac {C} {x(1+|\log x|^\alpha )}\\
\int _{3x/2}^\infty \frac {|u_n(y)-u_n(x)|} {|\log y|^\alpha }K(x, y)dy&\le 
C\int _{3x/2}^{\infty} \frac {dy} {y^2|\log y|^\alpha }\le \frac {C} {x(1+|\log x|^\alpha }\\
\int _{ x/2 }^{3x/2} \frac {|u_n(y)-u_n(x)|} {|\log y|^\alpha }K(x, y)dy&\le C ||u_n|| _{ H^\rho (I )}x^{-1}\int _{ x/2 }^{3x/2} \frac {dy} {|x-y|^{1-\rho }|\log y|^\alpha }dy\\
&\le C ||u_n|| _{ H^\rho (K )}\frac {1} {x^{1-\rho }(1+|\log x|^\alpha )}.
\end{align*}
Then, 
\begin{align}
|L(v_n)(x)|&\le  \frac {C|u_n(x)|} {x|\log x|^\alpha }+C\frac {1} {x^{1-\rho }(1+|\log x|^\alpha )}\le C,\,\,\forall x\in I.
\end{align}
\end{proof}
\subsection{Uniqueness of $\Lambda$ in $ E' _{ 0, 2 }$.}
\begin{prop}
\label{S3Euniq} 
For any $T>0$, the function $\Lambda$ is the unique weak solution of (\ref{S4E1}) on $0\le t\le T$ such that, for all $0\le t\le T$,   $\Lambda (t)\in  E' _{ 0, 2 }$ and 
$\mathscr M (\Lambda (t))$ is bounded on $\mathcal S$,     and $\Lambda (t)\rightharpoonup \delta _1$ in $\mathscr D'(0, \infty)$ as $t\to 0$.
\end{prop}

\begin{proof}
Suppose the existence of two solutions $\Lambda_1$ and $\Lambda_2$ satisfying the properties and call $\Lambda=\Lambda_1-\Lambda_2$. Then $\mathscr M(\Lambda(t))$ is analytic on $\mathcal S _{ 0, 2 }$ for $0\le t\le T$ and satisfies (\ref{S5EUXYp}) on $\mathscr Re (s)\in (1, 2)$, $0\le t\le T$. By Proposition (\ref{S3PX1}), $\mathscr M(\Lambda(t))$ is bounded on $\mathcal S$ for $0\le t\le T$. By the condition on the initial data $\mathscr M(\Lambda(t))\to 0$ uniformly for $s$ on compact subsets of $\mathcal S _{ 0, 2 }$. Let $\ell\in C^\infty (0, \infty)$ be  such that $\ell(t)=1$ for $0\le t \le T/2$ and $\ell(t)=0$ if $t\ge T$, and define
$\overline U(t, s)=\mathscr M(\Lambda(t))(s)\ell(t)$ that satisfies
\begin{align}
&\frac {\partial \overline U} {\partial t}(t, s)=W(s-1)\overline U(t, s-1)+r(t, s) \label{S3PU1}\\
&r(t, s)=\mathscr M(\Lambda(t))(s)\ell'(t)\label{S3PU2}
\end{align}
and the function $r$ is bounded on $(0, T)\times \mathcal S _{ 0, 2 }$, $r(t)\equiv 0$ if $0\le t\le T/2$. We may then Laplace transform both sides of (\ref{S3PU1}) and obtain, for some constant $C>0$,
\begin{align}
&z \tilde V(z, s)=-W(s-1)\tilde V(z, s-1)+\tilde r(z, s),\,\, \mathscr Re z>0,\,  \mathscr Re (s)\in (1, 2)\label{S3PU3}\\
&|\tilde r(z, s)|\le Ce^{-\frac {T} {2} \mathscr Re z},\,\,\forall s\in \mathcal S, \mathscr Re z>0.\label{S3PU5}
\end{align}
The function $\tilde V$ may be split as  $\tilde V=\tilde V_p+\tilde V_h$ where $\tilde V_p$ is the particular solution of (\ref{S3PU3}),
\begin{equation*}
\tilde V_p(z, s)=
\frac {1} {2i\pi }\frac {B(s)} {z}\int _{\mathscr Re (\sigma )=\beta } \frac {e^{(\sigma-s) \log(-z) }} { B(\sigma )}\frac {\tilde r(z, \sigma )\,d\sigma } {(1-e^{2i\pi (s-\sigma )}) }\\
\end{equation*}
and $\tilde V_h$ must satisfy
\begin{equation}
\frac {\partial \tilde V_h} {\partial t}(t, s)=-W(s-1)\tilde V_h(t, s-1), \,\,\,\mathscr Re z>0,\,  \mathscr Re (s)\in (1, 2) \label{S3PU4}
\end{equation}
The function $\tilde V_p(z, s)$ is analytic on  $s\in \mathcal S$ for all $\mathscr Re z>0$,  analytic on $\mathscr Re z>0$   and for all $s\in \mathcal S$.
By (\ref{S3PU5}), and our choice of the branch of the $\log$  function in (\ref{S4E1410}),
\begin{align}
\left|\tilde V_p(z, s) \right|&\le Ce^{-\frac {T} {2} \mathscr Re z}
\frac {1} {|z|}\int _{\mathscr Re (\sigma )=\beta } \frac {\left| e^{(\sigma-s) \log(-z) }\right|} { B(\sigma )}\frac {|d\sigma| } {\left|1-e^{2i\pi (s-\sigma )}\right| } \nonumber\\
&\le C _{ z_0 }e^{-\frac {T} {2} \mathscr Re z},\,\,\forall \mathscr Re z\ge z_0>0. \label{S3PU7}
\end{align}
On the other hand, using the function $\tilde V_h$ we define, following the same rationale as in the definition of (\ref{S4EBH}), in the Proof of Proposition  \ref{PropppS4EBH}

\begin{align*}
&\tilde H(z, s)=\frac { \tilde V_h (z, s)e^{s\log (-z)}}{B(s)}\\
&\tilde h(z, \zeta )=\tilde H(z, s), \,\, \zeta =e^{2i\pi (s-\beta )}.
\end{align*}
For every $z$ such that $\mathscr R e(z)>0$,  the function $h(z, \cdot)$ is then analytic on $\CC\setminus \RR^+$ and, by (\ref{S3PU4}),
\begin{equation*}
\tilde h(z, \zeta +i0)=\tilde h(z, \zeta -i0), \,\,\forall \zeta \in \RR^+.
\end{equation*}
It follows that for all $\mathscr R e(z)>0$, $\tilde h(z, \cdot)$ is analytic on  $\CC\setminus \{0\}$. But since, by Proposition \ref{S3PBP24} and (\ref{S3PU7}), we also have
$$
|\tilde h(z, \zeta )|\le C\left|e^{s\log (-z)} \right|=C e^{c\log z}\left|e^{i(s-\beta )Arg (-z)}\right|=C e^{c\log z}\left|\zeta  \right|^{\frac {Arg (-z)} {2\pi }}=C e^{c\log z}\left|\zeta  \right|^{1/2},
$$
by Liouville's Theorem $\tilde h(z)\equiv 0$. Therefore $\tilde H(z)=\tilde V_h(z)=0$ and $\tilde V =\tilde V_p$. By the inverse Laplace  formula 
\begin{equation*}
\overline U (t, s)=\frac {1} {2i\pi }\int  _{ a-i\infty }^{a+i\infty}\tilde V(z, s)e^{z t}dz,\,\,\,
\end{equation*}
and by (\ref{S3PU5}) we have then $\overline U (t, s)=\mathscr M(\Lambda (t,) (s)=0$ for all $s\in \mathcal S$ and $0\le t\le T/2$ from where the result follows.
\end{proof}
\begin{proof}
[\upshape\bfseries{Proof of Theorem  \ref{TheoMain1}}] All the properties of $\Lambda$ have already been proved in Proposition \ref{S5CInvM2}, Proposition \ref{S4P2}, Corollary \ref{S4C426} and Proposition \ref{S3Euniq}. The function $G$  satisfies   (\ref{S4E1})  and   (\ref{S5CSolE1})
by the scaling properties of the equation and the Dirac's delta. The $L^1$ continuity property follows from that of $\Lambda$.
 
\end{proof}

\section{Solution of the Cauchy problem.}
\setcounter{equation}{0}
\setcounter{theo}{0}
For all $y>0$ we define,
 \begin{align}
G(t, x; y)=y^{-1}\Lambda\left(\frac {t} {y}, \frac {x} {y} \right),\,\,\,\forall t>0, x>0, y>0.
\end{align}
By (\ref{S4C426E1}), $G\in C((0, \infty)\times (0, \infty); L^1(0, \infty; dx))$ and for $y>0$ fixed it  inherits properties form $\Lambda$. For example,
$G(\cdot, \cdot, y)$ is a weak solution to  (\ref{S4E1}) and
\begin{equation}
\label{S5CSolE1}
\lim _{ t\to 0 }G (t, \cdot, y)=\delta _y,\,\,\,\hbox{\rm in the weak sense of}\,\,\,\,\mathscr D'(0, \infty).
\end{equation}
The function $G$ also  satisfies the following important property,

\begin{prop}
\label{S3C2P1}
There exists a positive constant $C_G>0$ such that, for all $t>0, x>0$,
\begin{equation}
I(t, x)=\int _0^\infty \left|G\left(t, x; y\right)\right| dy<C_G.
\end{equation}
\end{prop}

The proof of Proposition \ref{S3C2P1} is split in several auxiliary  Lemmas.  Two different cases:
\begin{itemize}
\item If $0<t<x$,
\begin{align}
I(t, x)=\int _0^t \underbrace{(\cdots \cdots)}_{\text{$t/y >1, x/y>1$}}\frac {dy} {y}+
\int _t^x  \underbrace{(\cdots\cdots )}_{\text{$t/y <1, x/y>1$}}\frac {dy} {y}+\int _x^\infty  \underbrace{(\cdots\cdots )}_{\text{$t/y <1, x/y<1$}}\frac {dy} {y}.\label{S3.3E1}
\end{align}
\item For $0<x<t$,
\begin{align}
I(t, x)=\int _0^x \underbrace{(\cdots\cdots )}_{\text{$t/y >1, x/y>1$}}\frac {dy} {y}+
\int _x^t  \underbrace{(\cdots\cdots )}_{\text{$t/y >1, x/y<1$}}\frac {dy} {y}+\int _t^\infty  \underbrace{(\cdots\cdots )}_{\text{$t/y <1, x/y<1$}}\frac {dy} {y}. \label{S3.3E2}
\end{align}
\end{itemize}
\begin{lem}
\label{S4LG1} 
There exists $C>0$ such that,  for all $t>0$ and $x>0$,
\begin{align*}
\int _0^t\left|\Lambda\left(\frac {t} {y}, \frac {x} {y} \right)\right|\frac {dy} {y}\le C
\end{align*}
\end{lem}
\begin{proof}
[\upshape\bfseries{Proof of Lemma \ref{S4LG1}}]
Since $y\in (0, t)$, $t/y>1$ and  by Proposition  \ref{S4L1} and Proposition  \ref{S4L2},
$$
\left|\Lambda\left(\frac {t} {y}, \frac {x} {y}\right)\right|\le C \left(\max\left(\frac {t} {y}, \frac {x} {y}\right)\right)^{-3}.
$$
Then,
\begin{align*}
&\forall x>0,\, \forall t\in (0, x),\,\,\,\int _0^t\left|\Lambda\left(\frac {t} {y}, \frac {x} {y} \right)\right|\frac {dy} {y}\le \int _0^t\left(\frac {x} {y}\right)^{-3}\frac {dy} {y} 
=\frac {t^3} {3x^3}\le 1/3.\\
&\forall t>0,\, \forall x\in (0, t),\,\,\int _0^t \left|\Lambda\left(\frac {t} {y}, \frac {x} {y} \right)\right|\frac {dy} {y}\le  \int _0^t  \left(\frac {t} {y}\right)^{-3}\frac {dy} {y}=\frac {1} {3}
\end{align*}
\end{proof}
It remains now to estimate the two last integrals at the right hand side of  (\ref{S3.3E1}), and the last one at the right hand side of (\ref{S3.3E2}). To this end we will be using a  function $\delta (z)$, defined and continuous on $z\ge 0$ such that,
\begin{equation}
\label{S4Edelta}
\delta\,\,\hbox{is decreasing},\,\,\delta (u)<1\,\,\hbox{for all}\,\,u>0,\,\,\delta (1)=\frac {1} {2},\,\,\delta (u)=\frac {e^{1-u}} {2},\,\forall u\ge \frac {1} {2}.
\end{equation}

\subsection{The domain  $0<t<x$.}
Consider first the domain where $0<t<y<x$ where $0<\frac {t} {y}<1<\frac {x} {y}$. In order to use the estimate on $\Lambda$, this domain is still  subdivided. 
\begin{lem}
Define
\begin{align*}
H_2(z)=z(1+\delta (z))\,\,\hbox{and}\,\,\,
H_1(z)=z(1-\delta (z)),\,\,\,\forall z>0
\end{align*}
These two functions are monotone increasing. Moreover
\begin{align}
&\forall z>0,\,\,H_1(z)<z \label{S3.3E4.00}\\
&\forall z>3/2,\,\,\,H_2^{-1}(z)>1 \label{S3.3E4.0}\\
&\forall z>0, H^{-1}_2(z)<z \label{S3.3E4.1}\\
&\forall x>0,\,\,\forall t\in (0, 2x/3), \frac {2x} {3}<tH_2^{-1}\left(\frac {x} {t} \right)\label{S3.3E4}
\end{align}

\end{lem}
\begin{proof}
Since the function $H_2$ is strictly increasing, its inverse  $H_2^{-1} $ is well defined. The choice  $\delta (1)=1/2$ makes $H_2(1)=3/2$ then $H_2^{-1}(3/2)=1$. By monotonicity it follows that  $H_2^{-1}(z)>H_2^{-1}(3/2)=1$ for all $z>3/2$ and this proves (\ref{S3.3E4.0}). Since $H_2(z)>z$ it follows that $z>H_2^{-1}(z)$ and this shows   (\ref{S3.3E4}).

Since $\delta (1)=1/2$, we have 
$$
\frac {2} {3}\left(1+\delta (1) \right)=1
$$
and the function $\delta (z)$ is  strictly decreasing because so is $\rho (z)$. Therefore $\delta (z)<1/2$ for all $z>1$,
and, for all $t\in (0, 2x/3)$
\begin{align*}
H_2\left(\frac {2x} {3t} \right)=\frac {2x} {3t}\left(1+\delta  \left(\frac {2x} {3t}\right) \right)<\frac {2x} {3t}\left(1+\delta (1) \right)= \frac {x} {t}.
\end{align*}
Since $H_2$ is strictly increasing, so is $H_2^{-1}$, $\frac {2x} {3t}\le H_2^{-1}\left( \frac {x} {t}\right)$
and this proves (\ref{S3.3E4}). 
\end{proof}

\noindent

\begin{lem}
\label{S4ELtx}
For all $t>0$, $x>0$ such that $t<x$,
\begin{align}
&\int _t^{x} \left| \Lambda\left(\frac {t} {y}, \frac {x} {y}\right)\right| \frac {dy} {y}\le C\left(1+t+\Phi _1+\Psi _1+\tilde \Phi _2 \right)  \label{S4ELtx1}\\
&\int _x^\infty \left| \Lambda\left(\frac {t} {y}, \frac {x} {y} \right)\right| \frac {dy} {y}\le C\left(1+\Phi _3+\Psi _3 \right) \label{S4ELtx1B}
\end{align}
where,
\begin{align}
\Phi _1(x, t)&= t\int _{\frac {2x} {3}}^{tH_2^{-1}\left(\frac {x} {t} \right)} \frac {1} {y}\left|\frac {x} {y}-1 \right|^{-1}\frac {dy} {y},\,\,\forall  t\in (0, 2x/3),\label{S4ELtx2}\\
\Psi_1 \left(x, t \right)&= \int _{tH_2^{-1}\left(\frac {x} {t} \right)}^{x}\frac {t} {y}\left| \frac {x} {y}-1\right|^{-1+\frac {2t} {y} } \frac {dy} {y} \,\,\forall  t\in (0, 2x/3) \label{S4ELtx3}\\
\tilde \Phi_2 \left(x, t \right)&=  \int _{t}^{x} \frac {t} {y}\left| \frac {x} {y}-1\right|^{-1+\frac {2t} {y} } \frac {dy} {y} ,\,\,\,\forall t\in (2x/3, x),\label{S4ELtx4}
\end{align}
\begin{align}
\Psi _3\left(x, t\right)&=\int _x^{tH_1^{-1}\left(\frac {x} {t} \right)}  \frac {t} {y}\left| \frac {x} {y}-1\right|^{-1+\frac {2t} {y} } \frac {dy} {y},\,\,\forall t\in (0, x)\label{S4ELtx5}
\\
\Phi_3\left(x, t\right)&=\int _{tH_1^{-1}\left(\frac {x} {t} \right)}^{2x}\frac {t} {y} \left| 1-\frac {x} {y}\right|^{-1}\frac {dy} {y},\forall t\in (0, x)\label{S4ELtx6}
\end{align}
\end{lem}

\begin{proof}
[\upshape\bfseries{Proof of Lemma \ref{S4ELtx}}]
We  show (\ref{S4ELtx1}) first and  start assuming $t\in (0, 2x/3)$. By (\ref{S3.3E4}),
\begin{equation}
\int _t^x\left| \Lambda\left(\frac {t} {y}, \frac {x} {y}; 1 \right)\right| \frac {dy} {y}=\int _t^{\frac {2x} {3}}(\cdots)dy+\int _{\frac {2x} {3}}^{tH_2^{-1}\left(\frac {x} {t} \right)}(\cdots)dy+
\int _{tH_2^{-1}\left(\frac {x} {t} \right)}^{x}(\cdots)dy. \label{S4EEE1}
\end{equation}
In the first integral of the right hand side of  (\ref{S4EEE1}), since $y<2x/3$, by Proposition \ref{S4P9}
\begin{equation*}
\left|\Lambda \left(\frac {t} {y}, \frac {x} {y}\right)\right|\le C_1\left( \frac {x} {t}\right)^{-1-\beta _1'}\left( \frac {t} {y}\right)^{6}+
C_2\left( \frac {x} {t}\right)^{-6}\left( \frac {t} {y}\right)^2,
\end{equation*}
and then,

\begin{align}
\label{S4EX1}
\int _t^{\frac {2x} {3}}\left|\Lambda\left(\frac {t} {y}, \frac {x} {y} \right)\right|\frac {dy} {y}&\le C_1t^6 \int _t^{\frac {2x} {3}}y^{-6}dy+
C_2t^2\int _t^{\frac {2x} {3}}y^{-2}dy \le Ct.
\end{align}

In the second integral of the right hand side of  (\ref{S4EEE1}), simple computations yield,
\begin{align*}
&y\in \left( \frac {2x} {3},  t H_2^{-1}\left( \frac {x} {t}\right)\right)\Longrightarrow \frac {t} {y}H_2(y/t)<\frac {x} {y}<\frac {3} {2}\Longrightarrow  \delta \left( \frac {y} {t}\right)< \frac {x} {y}-1<\frac {1} {2}.
\end{align*}
Since $x>3t/2$ we have $y/t>1$. On the other hand,  $x/t$ may take values  arbitrarily large, and then $H_2^{-1}\left( \frac {x} {t}\right)$ and $y/t$ too. We deduce that $\delta (y/t)\in (0, 1/2)$
and by Proposition \ref{S3.3.4P1},
\begin{equation}
\int _{\frac {2x} {3}}^{tH_2^{-1}\left(\frac {x} {t} \right)}\left|\Lambda \left(\frac {t} {y}, \frac {x} {y} \right)\right|\frac {dy} {y}\le C\Phi _1(x, t).\label{S4EX2}
\end{equation}

In the third integral of the right hand side of (\ref{S4EEE1}), since $t H_2^{-1}\left( \frac {x} {t}\right)<y$, it follows that $t H_2^{-1}\left( \frac {x} {t}\right)<y $, from where
$\frac {x} {t}< H_2\left( \frac {y} {t}\right)=\frac {y} {t}\left(1+\delta  \left( \frac {y} {t}\right)\right) $. Then
$\frac {x} {y}<1+\delta  \left( \frac {y} {t}\right)$
and, since $x/y>1$ also, 
\begin{equation}
0<\frac {x} {y}-1<\delta  \left( \frac {y} {t}\right). \label{S4EX3.2}
\end{equation}
We notice now that since $x/t>3/2$ and  $\frac {3} {2}<\frac {x} {t}=u(1+\delta (u))\le 2u$, we also have $u=H_2^{-1}(x/t)>3/4$. 
Then  $y/t$ varies on the half line $(3/4, \infty)$ and $\delta (y/t)$ varies on $(0, \delta (3/4))$.
We deduce from (\ref{S4EX3.2}), using  Corollary \ref{S3Cor17}, that for some constant $C>0$,
\begin{align}
&\left|\Lambda\left(\frac {t} {y}, \frac {x} {y} \right)\right|\le C\frac {t} {y}\left| \frac {x} {y}-1\right|^{-1+\frac {2t} {y}}.\label{S4EX3}
\end{align}
It follows from (\ref{S4EX1}), (\ref{S4EX2}) and (\ref{S4EX3})  that for $0<t< 2x/3$,
\begin{align}
\label{S4E224}
\int _t^x\left| \Lambda\left(\frac {t} {y}, \frac {x} {y} \right)\right|\frac {dy} {y}\le C\left(t+ \Phi_1 \left(x, t \right)
+ \Psi_1 \left(x, t \right)\right).
\end{align}

Suppose now that  $t\in (2x/3, x)$.  We first deduce that since  $x/t <3/2$ and $H_2^{-1}$ is increasing, $H_2^{-1}(x/t)<H_2^{-1}(3/2)=1$ and then 
$tH_2^{-1}(x/t)<t$. Since $y\in (t, x)$ it follows that $y>tH_2^{-1}(x/t)$ and therefore,
\begin{align*}
H_2(y/t)\equiv \frac {y} {t}(1+\delta (y/t))>\frac {x} {t}\Longrightarrow  1+\delta (y/t)>\frac {x} {y}\Longleftrightarrow \frac {x} {y}-1<\delta (y/t).
\end{align*}
Then, for all $0<t<y<x$, we  have $x/y>1$ and,
$0<\frac {x} {y}-1<\delta (y/t)$. 
By Corollary \ref{S3Cor18},
 and  (\ref{S4ELtx4}) we deduce, when $t\in (2x/3, x)$,
\begin{align}
\label{S4E221}
\int _t^x \left| \Lambda\left(\frac {t} {y}, \frac {x} {y} \right)\right| \frac {dy} {y}\le \tilde C\Phi_2 \left(x, t \right).
\end{align}
and (\ref{S4ELtx1}) follows from (\ref{S4E224}) and   (\ref{S4E221}).

We prove  (\ref{S4ELtx1B}) now. To this end we write,

\begin{align}
\label{S4E21}
\int _x^\infty (\cdots )\frac {dy} {y}=\int _x^{tH_1^{-1}\left(\frac {x} {t} \right)}(\cdots )\frac {dy} {y}+\int _{tH_1^{-1}\left(\frac {x} {t} \right)}^{2x} {(\cdots )}\frac {dy} {y}+\int _{2x}^\infty {(\cdots )}\frac {dy} {y}
\end{align}
In the first term at the right hand side of  (\ref{S4E21}) $x<y< tH_1^{-1}\left(\frac {x} {t} \right)$, then $ 0<1-\frac {x} {y}<\delta  \left( \frac {y} {t}\right) $ from where, by Corollary \ref{S3Cor17} and (\ref{S4ELtx5})
\begin{align}
\int _x^{tH_1^{-1}\left(\frac {x} {t} \right)}\left|\Lambda\left(\frac {t} {y}, \frac {x} {y}; 1 \right)\right|\frac {dy} {y}\le C\Psi _3\left(x, t\right),\, 0<t<x.\label{S4EX6}
\end{align}
In the second integral at the right hand side of  (\ref{S4E21}), $ tH_1^{-1}\left(\frac {x} {t} \right)<y<2x$ and so $\delta  \left( \frac {y} {t}\right)<1-\frac {x} {y}<\frac {1} {2}$
and by  (\ref{S4ELtx6}) and Proposition \ref{S3.3.4P1},
\begin{align}
\int _{tH_1^{-1}\left(\frac {x} {t} \right)}^{2x} \left| \Lambda\left(\frac {t} {y}, \frac {x} {y} \right)\right|\frac {dy} {y}\le C\Phi_3\left(x, t\right)\,\,0<t<x.
\label{S4EX7}
\end{align}
In the last integral at the right hand side of  (\ref{S4E21}), since $y>2x$, by Proposition \ref{S3.3.4P1},
\begin{align}
\int _{2x}^\infty \left|\Lambda\left(\frac {t} {y}, \frac {x} {y} \right)\right|\frac {dy} {y}\le
Ct \int _{2x}^\infty  \frac {1} {|1-x/y|} \frac {dy} {y^2}\le Ct\int _{2x}^\infty \frac {dz} {y^2}\le C. \label{S4EX8}
\end{align}
The estimate   (\ref{S4ELtx1B})  follows now  by (\ref{S4EX6})--(\ref{S4EX8}).
\end{proof}
\subsection{The domain  $0<x<t$.}
\label{S4.2}
We estimate now  the last integral at the right hand side of (\ref{S3.3E2})
\begin{lem}
\label{S4ELtz}
For all $t>0$ and $x\in (0, t)$,
\begin{align}
&\forall t>2x,\,\,\int _t^\infty \left| \Lambda\left(\frac {t} {y}, \frac {x} {y} \right)\right|\frac {dy} {y}\le C \label{S4ELtz1}\\
&\forall t\in (x, 2x),\,\,\int _t^\infty \left| \Lambda\left(\frac {t} {y}, \frac {x} {y} \right)\right|\frac {dy} {y}\le C(1+\Phi  _3+\Psi  _4) \label{S4ELtz2}
\end{align}
where,
\begin{equation}
\Psi _4=\int _t^{tH_1^{-1}\left(\frac {x} {t} \right)}  \frac {t} {y}\left| \frac {x} {y}-1\right|^{-1+\frac {2t} {y} } \frac {dy} {y},\,\,\forall t \in (x, 2x).\label{S4ELtz3}\\
\end{equation}
\end{lem}
\begin{proof}
[\upshape\bfseries{Proof of Lemma \ref{S4ELtz}}]
If $t>2x$ then, $x/y<1/2$ and Proposition \ref{S3.3.4P1} gives (\ref{S4ELtz1}).

For  $t\in (x, 2x)$, $\frac {x} {t}>\frac {1} {2}\equiv H_1(1)$ and $t< tH^{-1}_1\left(\frac {x} {t} \right)$ by the monotonicity of $H_1$. On the other hand, 
$$
H_1\left( \frac {2x} {t}\right)=\frac {2x} {t}\left( 1-\delta  \left( \frac {2x} {t}\right)\right)\ge \frac {2x} {t}\left( 1-\delta (1)\right)=\frac {x} {t}
$$
(where use has been made of  $2x/t\ge 1$), and then, $tH^{-1}_1\left(\frac {x} {t} \right)<2x$.
Therefore,
\begin{align}
\label{S4E23}
\int _t^\infty  (\cdots) dy=\int_ t^{tH^{-1}_1\left(\frac {x} {t} \right)} (\cdots) dy+\int_{tH^{-1}_1\left(\frac {x} {t} \right)}^{2x} (\cdots) dy+\int_ {2x}^{\infty} (\cdots) dy.
\end{align}
In the first term at the right hand side of (\ref{S4E23}) $0<1-\frac {x} {y}<\delta\!\left(\frac {y} {t}\right)$ because $y\in \left(t,tH^{-1}_1\left(\frac {x} {t} \right)\right)$,
\begin{align}
\int_ t^{tH^{-1}_1\left(\frac {x} {t} \right)}  \left|\Lambda\left(\frac {t} {y}, \frac {x} {y} \right)\right|\frac {dy} {y}\le C \Psi _4\left(x, t \right). \label{S4EX9}
\end{align}
by (\ref{S4ELtz3}) and Corollary \ref{S3Cor18}.
In the second integral of the right hand side of (\ref{S4E23}) 
\begin{align*}
y\in \left(H^{-1}_1\left(\frac {x} {t} \right), 2x\right)\Longrightarrow  \delta  \left( \frac {y} {t}\right)<1-\frac {x} {y}<\frac {1} {2}.
\end{align*}
By Proposition  \ref{S3.3.4P1} and (\ref{S4ELtx6})
\begin{align}
\int_{tH^{-1}_1\left(\frac {x} {t} \right)}^{2x} \left| \Lambda\left(\frac {t} {y}, \frac {x} {y} \right)\right|\frac {dy} {y}\le 
C\Phi _3\left(x, t \right).
\label{S4EX10}
\end{align}
In the third integral of the right hand side of (\ref{S4E23})  $y>2x$
then by Proposition \ref{S3.3.4P1},
\begin{align}
\int_{2x} ^\infty \left| \Lambda\left(\frac {t} {y}, \frac {x} {y} \right)\right|\frac {dy} {y}\le C\label{S4EX11}
\end{align}
and  (\ref{S4ELtz2}) follows from (\ref{S4E23})--(\ref{S4EX11}) for $t\in (x, 2x)$.
\end{proof}
\subsection{Estimates  of  the functions $\Phi _\ell$ and $\Psi _\ell$.}
\begin{lem}
\label{S4LFIPS}
There exists a constant $C>0$ such that,
\begin{equation}
\label{S4LFIPSE}
\Phi _1+\Psi _1+\tilde \Phi _2+\Phi _3+\Phi _4+\Psi _4\le C
\end{equation}
\end{lem}

\begin{proof}
[\upshape\bfseries{Proof of Lemma \ref{S4LFIPS}}]
(i) Estimate of $\Phi_1$. 
By definition, for $x>0$ and $t\in (0, 2x/3)$,
\begin{align}
\Phi _1(x, t)&=\frac {Ct} {x}\int _{\frac {2} {3}}^{\frac {t} {x}H_2^{-1}\left(\frac {x} {t} \right)} |1-r|^{-1} dr=\frac {-t} {x}\left( \log \left(1-\frac {t} {x}H_2^{-1}\left(\frac {x} {t} \right) \right)+\log 3 \right).\label{S6EPhi1.1}
\end{align}
Then, for all  $\varepsilon >0$,   $\Phi _1(x, t)$ is  bounded for all $(t, x)$ such that $0<t<x$ and  $\frac {t} {x} H_2^{-1}\left(\frac {x} {t} \right)\in [0, 1-\varepsilon ]$. Assume now that 
$\frac {t} {x} H_2^{-1}\left(\frac {x} {t} \right) \to 1$, and denote $u=H_2^{-1}(x/t)$. Since,
\begin{align}
\label{S4EZX1}
\frac {t} {x} H_2^{-1}\left(\frac {x} {t} \right)=\frac {u} {H(u)}=\frac {1} {1+\delta (u)}
\end{align}
if $\frac {t} {x} H_2^{-1}\left(\frac {x} {t} \right)\to 1$ it follows that $\delta (u)\to 0$, then $u\to \infty$, 
\begin{align*}
\frac {t} {x} H_2^{-1}\left(\frac {x} {t} \right)=\frac {1} {1+\frac {e^{1-u}} {2}}=1-\frac {e^{1-u}} {2}+\mathcal O\left( e^{-2u}\right) ,\,\,\hbox{as}\,\,u\to \infty
\end{align*}
and
\begin{align}
\label{S4EZX2}
\frac {t} {x} H_2^{-1}\left(\frac {x} {t} \right)\underset{u\to \infty}{=}1+\mathcal O\left( e^{-u}\right),\qquad u=H_2^{-1}\left(\frac {x} {t} \right)\underset{u\to \infty}{=}\frac {x} {t} \left(1 +\mathcal O\left( e^{-u}\right)\right).
\end{align}
Using (\ref{S4EZX1}), (\ref{S4EZX2}) and the definition of $\delta $, for $\rho >0$ as small as desired and $u\to \infty$,
\begin{align*}
\frac {t} {x} H_2^{-1}\left(\frac {x} {t} \right)=
\frac {1} {1+e^{-\frac {x} {t}}(1+\mathcal O\left( e^{-(1-\rho)u}\right))}
=\frac {1} {1+e^{-\frac {x} {t}}}\left(1+\mathcal O\left(e^{-(2-\rho )u} \right) \right)
\end{align*}
and it follows that
\begin{align*}
\log \left(1-\frac {t} {x} H_2^{-1}\left(\frac {x} {t} \right)\right)\underset{u\to \infty}{=}-\frac {x} {t}+\mathcal O\left(e^{-\frac {x} {t}}\right).
\end{align*}
We deduce the existence of a constant $C>0$ such that  for all $0<t<2x/3$, 
\begin{equation}
\Phi_1 (x, t)\le C.  \label{S6EPhi1.2}
\end{equation}

\noindent
(ii) Estimate of $\Psi_1$. Since $t\in (0, 2x/3)$ and $y>tH_2^{-1}\left(\frac {x} {t} \right)$ then $x/t<H_2(y/t)<2y/t$. Using that $y<x$,  
also we deduce $0<\left( \frac {x} {y}-1\right)<1$. Since  $1/y>1/x$, 
\begin{align*}
\Psi_1 (x, t)&\le  t \int _{tH_2^{-1}\left(\frac {x} {t} \right)}^{x}
\left( \frac {x} {y}-1\right)^{-1+\frac {2t} {x}} \frac {dy} {y^2}= tx^{-1}  \int _{\frac {t} {x}H_2^{-1}\left(\frac {x} {t} \right)}^1
(1-\rho )^{-1+\frac {2t} {x}} \rho ^{-1-\frac {2t} {x}} d\rho 
\end{align*}
By (\ref{S3.3E4}),  $2H_2^{-1}\left(\frac {x} {t} \right)>\frac {4x} {3t}$, then $\frac {t} {x}H_2^{-1}\left(\frac {x} {t} \right)>\frac {1} {2}$ and,
\begin{equation}
\label{S4EZX7}
\Psi_1 (x, t)\le 
tx^{-1}  \int _{\frac {1} {2}}^1
(1-\rho )^{-1+\frac {2t} {x}} \rho ^{-1-\frac {2t} {x}} d\rho=C.
\end{equation}\\
\noindent
(iii) Estimate of $\tilde \Phi_2$.  When $t\in (2x/3, x)$ and $y\in (t, x)$, $0<\frac {x} {y}<1$ and then, by (\ref{S4ELtx4})
\begin{align}
\tilde \Phi _2(x, t)\le t\int _t^x \left( \frac {x} {y}-1\right)^{-1+\frac {2t} {x}}\frac {dy} {y^2}=\frac {t} {x}\int  _{ \frac {t} {x}}^1(1-r)^{-1+\frac {2t} {x}}r^{-1-\frac {2t} {x}}dr \nonumber\\
\le \frac {t} {x}\int  _{\frac {2} {3}}^1(1-r)^{-1+\frac {2t} {x}}r^{-1-\frac {2t} {x}}dr=2^{-\frac {2t} {x}-1}\le 2^{-4/3}. \label{S4ELtx34}
\end{align}

\noindent
(iv) Estimate of $\Phi_3$.  By definition, for $0<t<x$,
\begin{equation}
\Phi_3(x, t)=\frac {t} {x} \int _{\frac {t} {x}H_1^{-1}\left(\frac {x} {t} \right)}^{2} (r-1)^{-1}\frac {dr} {r}= -\frac {t } {x} \log \left(\frac {t} {x}H_1^{-1}\left(\frac {x} {t} \right) -1\right). \label{S6Phi3.3}
\end{equation}
because, if $v=H_1^{-1}\left(\frac {x} {t} \right)$ then $\frac {x} {t}=H_1(v)=v(1-\delta (v))$, and $
\frac {t} {x}H_1^{-1}\left(\frac {x} {t} \right)=\frac {1} {1-\delta (v)}>1$.\\
The same arguments as in the estimate of the right hand side of (\ref{S6EPhi1.1}), show the existence of a constant  $C>0$ such that  for all $0<t<x$, 
\begin{equation}
\Phi_3(x, t)\le C. \label{S6EPhi3.4}
\end{equation}
(v) Estimate of $\Psi_3$. 
For all $y$  in the domain of integration of $\Psi_3$,  $y<tH_1^{-1}\left(\frac {x} {t} \right)$, and then
$\frac {2t} {y}>\frac {2} {H_1^{-1}\left(\frac {x} {t} \right)}$.
Since $y>x$ also,  we have  $\left(1- \frac {x} {y}\right)\in (0, 1)$ and we deduce from (\ref{S4ELtx5}),
\begin{equation*}
\Psi _3\left(x, t\right)\le  t\int  _{x }^{tH_1^{-1}\left(\frac {x} {t} \right)} \left(1- \frac {x} {y}\right)^{-1+\frac {2} {H_1^{-1}\left(\frac {x} {t} \right)} } \frac {dy} {y^2}
= \frac {t} {x}  \int  _{1}^{\frac {t} {x}H_1^{-1} \left(\frac {x} {t} \right)}  \frac {(r-1)^{-1+\frac {2} {H_1^{-1}\left(\frac {x} {t} \right)} }} {r^{1+\frac {2} {H_1^{-1}\left(\frac {x} {t} \right)} }}dr
\end{equation*}
We use now that, because $\delta (x/t)<1/2$,  $z<H_1(2z)$ and so $\frac {t} {x}H_1^{-1}\left(\frac {x} {t} \right)<2$, to obtain,
\begin{align}
\label{S4EZX9}
\Psi _3\left(x, t\right)\le \frac {t} {x}  \int  _{1}^2\frac {(r-1)^{-1+\frac {2} {H_1^{-1}\left(\frac {x} {t} \right)} }} {r^{1+\frac {2} {H_1^{-1}\left(\frac {x} {t} \right)} }}dr
=\frac {t} {x}H_1^{-1}(x/t)2^ {-1-\frac {2t} {H_1^{-1}(x/t)} }\le C.
\end{align}
(vi) Estimate of $\Psi_4$. By definition, $x<t<y<tH^{-1}_1\left(\frac {x} {t} \right)<2x$, for all $y$ in the domain of integration. Therefore, as for $\Psi _3$,  we  have  $\frac {2t} {y}>\frac {2} {H_1^{-1}\left(\frac {x} {t} \right)}$ and $\left(1- \frac {x} {y}\right)\in (0, 1)$.
Arguing as for $\Psi _3$, we deduce from (\ref{S4ELtz3}), for all $t\in (x, 2x)$,
\begin{align}
\Psi _4(x, t)&\le t \int_ t^{tH^{-1}_1\left(\frac {x} {t} \right)}\left(1- \frac {x} {y}\right)^{-1+\frac {2} {H^{-1}_1\left(\frac {x} {t} \right)}} \frac {dy} {y^2} \le  
t \int _1^2 \frac {(r-1)^{-1+\frac {2} {H_1^{-1}\left(\frac {x} {t} \right)} }} {r^{1+\frac {2} {H_1^{-1}\left(\frac {x} {t} \right)} }} dr\nonumber\\
&= tx^{-1}H_1^{-1}(x/t)2^ {-1-\frac {2t} {H_1^{-1}(x/t)} }\le C.\label{S4EZX10}
\end{align}
Lemma \ref{S4LFIPS} follows from (\ref{S6EPhi1.2})--(\ref{S4EZX10})
\end{proof}
\begin{proof}
[\upshape\bfseries{Proof of Proposition \ref{S3C2P1}}]
Proposition \ref{S3C2P1} follows from Lemmata \ref{S4LG1}--\ref{S4LFIPS}
\end{proof}
It is now possible to define the solution $u$ of the Cauchy problem.
\begin{theo}
\label{S3C2}
(i) For any   $f_0\in L^1 (0, \infty)$, 
\begin{equation}
\int _0^\infty\int _0^\infty  \left|\Lambda \left(\frac {t} {y}, \frac {x} {y}\right) f_0(y)\right| \frac {dy} {y}dx<\infty,\forall t>0.
\end{equation}
The function defined for all $t>0, x>0$ as 
\begin{equation}
\label{S3C2E1}
u(t, x)=\int _0^\infty \Lambda \left(\frac {t} {y}, \frac {x} {y}\right)f_0(y)\frac {dy} {y}
\end{equation}
is such that $u\in L^\infty ((0, \infty); L^1 (0, \infty))\cap C((0, \infty); L^1(0, \infty))$ and there exists $C>0$,
\begin{equation}
\label{S3C2E1T}
\forall t>0,\,\,\,\,||u(t)||_1 \le C||f_0||_1.
\end{equation}
(ii) For every $f_0\in L^\infty (0, \infty)$ the function $u$ given by (\ref{S3C2E1}) is  well defined, it  belongs to $L^\infty ((0, \infty)\times (0, \infty))$ and:
\begin{equation}
\forall t>0,\,\,\,\,||u(t)||_\infty \le C_G||f_0||_\infty.
\end{equation}
\end{theo}
\begin{proof}
[\upshape\bfseries{Proof of Theorem \ref{S3C2}}]
The case  (i) is an easy consequence of Corollary \ref{S4C426}.
\begin{align*}
\int _0^\infty |u(t, x)|dx & \le\int _0^\infty\int _0^\infty \left|f_0(y) \Lambda \left(\frac {t} {y}, \frac {x} {y}\right)\right|\frac {dy} {y}dx\\
&=\int _0^\infty |f_0(y)| \int _0^\infty \left| \Lambda \left(\frac {t} {y}, z \right)\right|dzdy \le C\int _0^\infty \frac {|f_0(y)|dy} {1+(t/y)^2}.
\end{align*}
The case (ii) follows from Proposition \ref{S3C2P1}
\end{proof}
\begin{proof}[\upshape\bfseries{Proof of Theorem \ref{TheoMain2}}]
 Property (\ref{TheoMain2-1})  has been proved in  Theorem \ref{S3C2}.
For all $t>0, t'>0$,

\begin{align*}
&\int _0^\infty |u(t, x)-u(t', x)|dx\le \int _0^\infty |f_0(y)| \int _0^\infty \left|\Lambda\left(\frac {t} {y}, \frac {x} {y} \right)-\Lambda\left(\frac {t'} {y}, \frac {x} {y} \right)\right|dx \frac {dy} {y}.\\
\hskip -0.2cm \hbox{Since:}\quad \qquad &\lim _{ t'\to t }\int _0^\infty \left|\Lambda\left(\frac {t} {y}, \frac {x} {y} \right)-\Lambda\left(\frac {t'} {y}, \frac {x} {y} \right)\right|dx=0,\,\,\forall y>0,\\
&\frac {|f_0(y)|} {y}\int _0^\infty \left|\Lambda\left(\frac {t} {y}, \frac {x} {y} \right)-\Lambda\left(\frac {t'} {y}, \frac {x} {y} \right)\right|dx\le |f_0(y)|\in L^1
\end{align*}
by  dominated convergence Theorem $u\in C(0, \infty; L^1(0, \infty))$.  On the other hand, for all $\varphi \in \mathscr D(0, \infty)$,
\begin{equation*}
\int _0^\infty u(t, x)\varphi (x)dx=\int _0^\infty f_0(y)\int _0^\infty \Lambda\left(\frac {t} {y}, \frac {x} {y}\right)\varphi (x)dx\frac {dy} {y}.
\end{equation*}
By Corollary \ref{S5CSol}, for all $y>0$ fixed,
\begin{align*}
\lim _{ t\to 0 }\int _0^\infty \Lambda\left(\frac {t} {y}, \frac {x} {y}\right)\varphi (x)dx=\lim _{ t\to 0 }\int _0^\infty \Lambda \left(\frac {t} {y}, z\right)\varphi (yz)ydz
=y\varphi (y)
\end{align*}
and since, for some positive constant $C$,
\begin{align*}
\left|\int _0^\infty \Lambda \left(\frac {t} {y}, \frac {x} {y}\right)\varphi (x)dx\right|=\left|\int _0^\infty \Lambda\left(\frac {t} {y}, z\right)\varphi (yz)ydz\right|\le C
\end{align*}
property  (\ref{TheoMain2-3})  follows by the Lebesgue's convergence Theorem. Standard arguments show that $u$ is a weak solution of (\ref{S4E1}). 

If we suppose $f_0\in L^1(0, \infty)\cap L^\infty(0, \infty)$ then,  $u\in L^\infty((0, \infty)\times (0, \infty))$ and estimate  (\ref{TheoMain2-4}) holds true,   as it has been proved in  Theorem \ref{S3C2}.

On the other hand,  for $t>0,x>0$,
\begin{align}
\int _0^\infty (u(t, y)-u(t, x))K(x, y)dy=&\int _0^\infty \int _0^\infty f_0(z)\left( \Lambda \left( \frac {t} {z}, \frac {y} {z}\right)- \Lambda \left( \frac {t} {z}, \frac {x} {z}\right) \right)\frac {dz} {z}K(x, y)dy \nonumber \\
=\int _0^\infty &f_0(z)\left(\int _0^\infty\left( \Lambda \left( \frac {t} {z}, u\right)- \Lambda \left( \frac {t} {z}, \frac {x} {z}\right) \right)K\left(\frac {x} {z}, u\right) du\right)\frac {dz} {z^2} \nonumber\\
=&\int _0^\infty L\left(\Lambda \left(\frac {t} {z} \right) \right)\left(\frac {x} {z} \right)f_0(z) \frac {dz} {z^2}. \label{S5FE21E1}
\end{align}
By Proposition \ref{S3ELfg}, for all $t>0$ and $x>0$, $z\not =x$,

$$
 L\left(\Lambda \left(\frac {t} {z} \right) \right)\left(\frac {x} {z} \right)=\frac {\partial \Lambda} {\partial t} \left( \frac {t} {z}, \frac {x} {z}\right)
$$
and then, for all $t>0, x>0$,
\begin{align}
\label{S5FE21E2}
\int _0^\infty (u(t, y)-u(t, x))K(x, y)dy=\int _0^\infty f_0(z) \frac {\partial \Lambda} {\partial t} \left( \frac {t} {z}, \frac {x} {z}\right)\frac {dz} {z^2}.
\end{align}
By (\ref{S3L3DE2X}), (\ref{S3L3DE2Y}) in Proposition \ref{S3PDL}, for $x>0$, $t>0$ and $t>z$,
\begin{equation}
\label{S5MO1}
\left|  \frac {\partial \Lambda} {\partial t} \left( \frac {t} {z}, \frac {x} {z}\right) \right|\le C\, \frac{z^4}{ \max(t^4, x^4)}.
\end{equation}
When  $x>0$, $t>0$ and $t<z$ we have three different  cases. By (\ref{S4P9E1B}) and (\ref{S4P9E1C}) in Proposition  \ref{S4P9}, for all $\varepsilon >0$ small there exists $C_\varepsilon >0$ such that
\begin{align}
&x>\frac {3z} {2}\Longrightarrow \left|  \frac {\partial \Lambda} {\partial t} \left( \frac {t} {z}, \frac {x} {z}\right) \right|\le 
 C_\varepsilon \left( \frac {x} {z}\right)^{-3+\varepsilon }\left(\frac {t} {z} \right)^6\left(\left( \frac {t} {z}\right)^{2-\varepsilon }+\left( \frac {x} {z}\right)^{-2-\varepsilon } \right)\nonumber\\
&\hskip 4.3cm  \le  C_\varepsilon  x^{-3+\varepsilon }t^6z^{-3-\varepsilon } \label{S5MO2}\\
\label{S5MO3}
&x<\frac {t} {2}<\frac {z} {2}\Longrightarrow \left|  \frac {\partial \Lambda} {\partial t} \left( \frac {t} {z}, \frac {x} {z}\right) \right|\le C\,\frac {xt^4} {z^5}.
\end{align}
By  (\ref{S3.E46B}) in Proposition  \ref{S4P2},
\begin{align}
\left|\frac {x} {z}-1 \right|<\frac {1} {2}\Longrightarrow  \left|  \frac {\partial \Lambda} {\partial t} \left( \frac {t} {z}, \frac {x} {z}\right) \right| &\le C\left| \frac {x} {z}-1\right|^{-1+\frac {2t} {z}}
\left( 1+\frac {2t} {z}\log\left( \frac {x} {z}\right)\right)\nonumber \\
&\le C\left| \frac {x} {z}-1\right|^{-1+\frac {2t} {z}}.\label{S5MO4}
\end{align}
We observe,
\begin{align*}
\int  _{ \frac {2x} {3} }^{2x} \left| \frac {x} {z}-1\right|^{-1+\frac {2t} {z}}\frac {dz} {z^2}\le x^{-1}\int  _{ 2/3 }^2\left| \rho -1\right|^{-1+\frac {t} {x}}\rho ^{-1-\frac {2t} {x}}d\rho \equiv \Theta (t, x).
\end{align*}
Then, for $x>0$, $t_2>t_1>0$ fixed and $t\in (t_1, t_2)$,
\begin{align*}
\left|f_0(z)\frac {\partial \Lambda} {\partial t} \left( \frac {t} {z}, \frac {x} {z}\right)\frac {1} {z^2}\right| &\le 
C|f_0(z)|  \left(\frac {\1 _{ z\le t_2 }} {\max\{t^4, x^4\}}+\1 _{ t_1<z<2x/3 }t_2^{1-\varepsilon }x^{-3+\varepsilon } 
+t_1^{-3}x \1 _{t_1<z }\right)+\\
&\hskip 2.5cm +C||f_0|| _{ L^\infty(2x/3, 2x) }z^{-2}\left| \frac {x} {z}-1\right|^{-1+\frac {2t_1} {z}} \1 _{ 2x/3<z<2x }
\end{align*}
and then,
\begin{equation}
\label{S5FE21E3}
\frac {\partial } {\partial t}\int _0^\infty f_0(z) \Lambda \left(\frac {t} {z}, \frac {x} {z} \right)\frac {dz} {z}=\int _0^\infty f_0(z) \partial _t \Lambda \left(\frac {t} {z}, \frac {x} {z} \right)\frac {dz} {z^2}.
\end{equation}
It follows from (\ref{S5FE21E1})--(\ref{S5FE21E3}) that $u$ satisfies (\ref{S3EMK2})  for all $t>0, x>0$. We also deduce from (\ref{S5FE21E2})--(\ref{S5MO4}) that $L(u)\in L((0, \infty)\times (0, \infty))$ and 
\begin{align*}
\left|L(u(t))(x)\right|\le \frac {C||f_0|| _{ L^1(0, t) }} {\max\{t^2, x^4\}}+
C_\varepsilon ||f_0||_1 \left(t^{1-\varepsilon }x^{-3+\varepsilon }\1 _{ t<2x/3 }+xt^{-3}\1 _{ 2x<t }\right)+\\
+\Theta(t, x)\1 _{ t<2x }||f_0|| _{ L^\infty(2x/3, 2x) }.
\end{align*}
and since $\Theta$ is a bounded function, (\ref{S5FE21-2}) follows. 
\end{proof}

\begin{proof}
[\upshape\bfseries{Proof of Proposition \ref{S4EM1PN}}]
When $t>0$ is fixed and $x\to 0$ we are in the region where $2x<t$ and we write, using the definition  (\ref{SISolu}) of $u$,
\begin{align*}
&u(t, x)=I_1+I_2+I_3\\
&I_1=\int _0^x \Lambda \left(\frac {t} {y}, \frac {x} {y} \right)  f_0(y) \frac {dy} {y},\,\,\,I_2=\int _x^t  \Lambda  \left(\frac {t} {y}, \frac {x} {y} \right)  f_0(y) \frac {dy} {y}\\
&I_3=\int _t^\infty  \Lambda  \left(\frac {t} {y}, \frac {x} {y} \right)  f_0(y) \frac {dy} {y}
\end{align*}
Since  in the two first integrals of the right hand side $t/y>1$, it follows  by  (\ref {S4L1E1}), (\ref{S4L2E1}) and (\ref{S3L3E1})
\begin{align*}
\Lambda \left(\frac {t} {y}, \frac {x} {y} \right)&=\left(\frac {t} {y}\right)^{\!\!\!-3}Q _1\left(\frac {x} {t} \right)+Q_2\left(\frac {t} {y}, \frac {x} {t} \right)\\
Q _1\left(\frac {x} {t} \right)&=\frac {2c_1B(1)} {W'(0)}+\mathcal O _{ \varepsilon  }\left( \frac {x} {t}\right),\,\,\frac {x} {t}\to 0\\
Q_2\left(\frac {t} {y}, \frac {x} {t} \right)&=c_2\left(\frac {t} {y} \right)^{\!\!\!-4}+b_1\left(\frac {t} {y} \right)+\mathcal O\left(\left(\frac {t} {y}\right)^{\!\!\!-4}\left|\frac {x} {t}\right|^{1-\delta } \right),\,\frac {x} {t}\to 0,\,\,\frac {t} {y}>1,
\\
&=c_2\left(\frac {t} {y} \right)^{\!\!\!-4}+b_1\left(\frac {t} {y} \right)+t^{-4}y^4\mathcal O\left(\left|\frac {x} {t}\right|^{1-\delta } \right),\,\frac {x} {t}\to 0,\,\,\frac {t} {y}>1.
\end{align*}
Therefore,
\begin{align*}
I_1+I_2&=\int _0^t  \Lambda  \left(\frac {t} {y}, \frac {x} {y} \right)  f_0(y) \frac {dy} {y}=\frac {2c_1B(1)} {W'(0)}t^{-3}\int _0^tf_0(y)y^2dy\left(1+\mathcal O\left(\frac {x} {t} \right) \right)+\\
&+c_2t^{-4}\int _0^t f_0(y)y^3dy+\int _0^t f_0(y)b_1\left(\frac {t} {y} \right)\frac {dy} {y}+t^{-4}\mathcal O\left(\left|\frac {x} {t}\right|^{1-\delta } \right)\int _0^t f_0(y)y^3 dy
\end{align*}
and,
\begin{align*}
\lim _{ x\to 0 }\int _0^t  \Lambda  \left(\frac {t} {y}, \frac {x} {y} \right)  f_0(y) \frac {dy} {y}=\ell(f_0; t)
\end{align*}
On the other hand,  $2x<t$ and $y>t$ in $I_3$ imply  $x/y<1/2$. By  Proposition \ref{S3.3.4P1},
\begin{align*}
\left| \Lambda\left(\frac {t} {y}, \frac {x} {y} \right)\right|\frac {1} {y}\le \frac { Ct} {y^2\left|\log\left(\frac {x} {y}\right)\right|} 
\end{align*}
Since,
\begin{align*}
&\lim _{ x\to 0 }\frac { t|f_0(y)|} {y^2|\log(x/y)|} =0,\,\,\forall y>0,\\
&\frac { t|f_0(y)|} {y^2|\log(x/y)|} \le  \frac { t|f_0(y)|} {y^2\log(2)} \in L^1(t, \infty)
\end{align*}
it follows by the Lebesgue's convergence Theorem,
$$
\lim _{ x\to 0 }\int _t ^\infty  \Lambda  \left(\frac {t} {y}, \frac {x} {y} \right)  f_0(y) \frac {dy} {y}=0.
$$
We then deduce (\ref{S4L25E1}) where $b_1$ is the function given in (\ref{S3EBB1}) and 
\begin{align}
A_1=-\frac {2} {\sqrt{2\pi }W(1)W'(2)W'(0)},\,\,\,
A_2=\frac {6} {\sqrt{2\pi }W'(0)W(3)W(1)}Res\left(\frac {1} {W(s)}, s=2 \right) \label{S4Ecata}
\end{align}

\end{proof}

\begin{rem}
The values of the constants $A_1$ and $A_2$ in Proposition \ref{S4EM1PN} may be easily evaluated with Mathematica:
\begin{align*}
W(1)\approx -0.369004;\,\,\,\,&W(3)\approx 1.91418;\,\,W'(2)=-W'(0)=\frac {\pi ^2} {12}.\\
A_1\approx & -3.19648,\,\,\,A_2\approx 4.1203.
\end{align*}
On the other hand,
\begin{align*}
&b_1(r)=A_3r^{-\sigma _1-1}+\mathcal O(r)^{-12},\,\,r\to \infty,\,\,\,\,\sigma _1\in (7.04, 7.06),\\
&A_3=-\frac {B(1)B(\sigma _1)} {W'(0)}\Gamma (\sigma _1+1)Res\left(\frac {1} {W(s)}, s=\sigma _1 \right).
\end{align*}
\end{rem}

\begin{rem}
If $(y^2+y^{4+\delta })f_0\in L^1(0, \infty)$,
\begin{align*}
\ell(f_0; t)&\sim A_1t^{-3}\int _0^\infty f_0(y)y^2dy+A_2t^{-4}\int _0^\infty f_0(y)y^3dy +\mathcal O (t)^{-4-\delta },\,\,t\to \infty
\end{align*}
If on the other hand, $f_0(y)\sim y^{-\rho }$ as $y\to 0$ for some $\rho \in (0, 1)$, by (\ref{S4L25E2}),
\begin{align*}
\ell(f_0; t)&=\int _0^1 f_0(t z)\left(A_1z^3+A_2z^4+b_1\!\!\left(\frac {1} {z} \right) \right)\frac {dz} {z}\\
&\sim t^{-\rho  }\int _0^1 \left(A_1z^3+A_2z^4+b_1\!\!\left(\frac {1} {z} \right) \right) z^{-\rho }\frac {dz} {z},\,\,t\to 0.
\end{align*}

\end{rem}

\section{Appendix}
\setcounter{equation}{0}
\setcounter{theo}{0}

\subsection{The Proof of Proposition \ref{S3PX1}}
Based on the expression (\ref{S5EUX1}) of $U(t)$ it closely follows that of Proposition 8.1 in \cite{EV}
\begin{align*}
U(t, s)=-\frac {B(s)} {\sqrt{2\pi}}\int  _{ \mathscr Re(\sigma )=\beta }\frac {t^{-(\sigma -s)}\Gamma (\sigma-s) } { B(\sigma )}d\sigma,\qquad\beta \in (0, 2),\,\, \beta -1<c<\beta
\end{align*}
(similar to (5.1) in \cite{EV}). As in (8.34) of \cite{EV}, this may be written,
\begin{align}
U(t, s)&=-\frac {B(s)} {\sqrt{2\pi}}\int  _{ \mathscr Re(Y )=\beta -\mathscr Re(s ) }\frac {t^{-Y}\Gamma (Y) } { B(s+Y )}dY =\int  _{ \mathscr Re(\sigma )=\beta  }e^{\psi (s, \sigma , t)}A(Y)dY \label{SAE1}
\end{align}
where
\begin{align}
\Psi (s, Y, t)= \int_{ {\mathscr Re}(\rho ) =\beta  } \log\left( -W(\rho )\right)\Theta(\rho -s, Y)d\rho-Y\log t-Y+\left( Y-\frac {1} {2}\right)\log Y,  \label{SAE2}
\end{align}
with $\Theta$ defined in (\ref{S3ETheta}), and
\begin{align}
A(Y)=\frac {\Gamma (Y)} {\sqrt{2\pi }e^{-Y}Y^{Y-1/2}}.  \label{SAE3}
\end{align}
The function $A$ defined in (\ref{SAE3}) is  the same as in (8.5) of \cite{EV}. The function $\Psi $ defined in (\ref{SAE2}) is similar to  (8.4) in \cite{EV}, the only difference lies in the function $W$ instead of $\Phi $. 

The proof of the estimates (\ref{S3PX1E1}), (\ref{S3PX1E2}) of Proposition \ref{S3PX1} follows then the same arguments as in \cite{EV} with only minor differences.
For $s$ in bounded sets, contour deformation and method of residues in the integrals (\ref{SAE1}), (\ref{SAE2}).
For $|s|$ large, these arguments are combined with the stationary phase Theorem applied to $\Psi (s, Y, t)$ as a function of $Y$,  where $s$ and $t$ are fixed. The variable $Y$ is scaled as $Y=2Z \log|s|$, according to the behavior of $W(s)$ as $\mathscr I m(s)\to \infty$, for $\mathscr Re(s)$ in a fixed  bounded interval and the result follows from the following. If we define,
\begin{align}
&\tilde F(s, \zeta )= \int_{ {\mathscr Re}(\rho ) =\beta  } \log\left( -W(\rho )\right)\Theta(\rho -s, \zeta )d\rho \label{SA2EqtildeF}\\
&F(s, Z)= \int_{ {\mathscr Re}(\rho ) =\beta  } \log\left( -W(\rho )\right)\Theta(\rho -s, 2Z \log|s|)d\rho,=\tilde F(s, 2Z \log|s| )\label{SA2EqF}
\end{align}

\begin{lem}
\label{SALemma14.1}
For any constant $C>0$, there exists a constant $L>0$ and $s_0\in \CC$, both depending on $C$, such that, for all $s\in \mathcal T_L\cap B _{ s_0 }(0)^c$ the function $F$ may be extended analytically for $Z\in D(s, C)\cap B _{\frac{|\log|s||} {8} }(0)$ where
\begin{equation*}
D_1(s, C)=\left\{s\in \CC, \mathscr Re(s)<0,\,\,|\mathscr Re(s)|\le C|\mathscr Im(s)+\frac {|\log|s||} {8}| \right\}
\end{equation*}
There also exists a constant $C'>0$, that depends on $C$, such that, for all $Z\in D_1(s, C)\cap B _{ \frac { |\log|s||} {8} }(0)$ and $s\in \mathcal T_L\cap B _{ s_0 }(0)^c$,
\begin{equation}
\label{SA1Lemma14.1E1}
\left|F(s, Z)+Z\log (-W(s))\log|s| \right|\le C'\left( Z^2+\mathcal O\left(\frac {1} {\log |s|} \right)\right).
\end{equation}
\end{lem}

\noindent
Due to the slow decay of the function $U(t, s)$ as $|s|\to \infty$, the following is also needed
\begin{lem} 
\label{SA2Lemma9.3E57}
There  exists a constant $C'>0$ such that,  for all $s\in \mathcal T_L\cap B _{ s_0 }(0)^c$, and $\zeta $ such that   $Z=\zeta / \sqrt{|s|} \in D_1(s, C)\cap B _{ \frac { |\log|s||} {8} }(0)$,
\begin{align}
\label{SA2partialU}
&\left|\frac {\partial \tilde F} {\partial s} \left(s, \zeta  \right) \right|\le \frac {C|\zeta |^2} {|s|^2\log |s|}+Ce^{-a'|s|}
\end{align}
\end{lem}
\begin{proof}
By (\ref{SA2EqF})
\begin{align*}
\frac {\partial \tilde F} {\partial s} (s, \zeta ) &=  \int_{ {\mathscr Re}(r ) =\beta-  {\mathscr Re}(s )} \frac {\partial } {\partial s}\left(\log\left( -W(r+s )\right) \right)\Theta(r, \zeta )dr\\
&= -\int_{ {\mathscr Re}(r ) =\beta-  {\mathscr Re}(s )} \frac {W'(r+s)} {W(r+s)}\Theta(r, \zeta )dr= 
\int_{ {\mathscr Re}(\rho  ) =\beta} \frac {W'(\rho )} {W(\rho )}\Theta(\rho -s, \zeta )d\rho 
\end{align*}
By (\ref{S3PW17E1}) and (\ref{S3PW17E3}),
\begin{align*}
\frac {W'(\rho )} {W(\rho )}&\underset{|\mathscr Im \rho |\to \infty}{=}-\frac {\rho ^{-1}+\mathcal O(|\rho| ^{-2})} {-2\log |\frac {b\rho} {2} |+\frac {2i(u-1)} {v}+\mathcal O\left(|\rho |^{-2} \right)}
\underset{|\mathscr Im \rho |\to \infty}{=}-\frac {1+\mathcal O(|\rho |^{-1})} {\rho\left(2\log |\frac {b\rho} {2} | \right) }
\end{align*}
The proof now follows the lines of Lemma 14.1 in \cite{EV}. Suppose that $\mathscr Im (s)>>1$ and denote $\zeta =Z\sqrt {|s|}$,

\begin{align*}
&\left|\frac {\partial \tilde F} {\partial s} (s, Z\sqrt {|s|}) \right|\le C\int_{ {\mathscr Re}(\rho  ) =\beta}\left|\frac {W'(\rho )} {W(\rho )}\right| \left|\Theta(\rho -s, Z\sqrt {|s|})\right| |d\rho| \\
&\le C\int\limits_{\substack{ {\mathscr Re}(\rho  )=\beta, {\mathscr Im}(\rho  ) >0  \\  |s-\rho |\le \frac {|s|} {4}} } \left|\frac {W'(\rho )} {W(\rho )}\right| \left|\Theta(\rho -s, Z\sqrt {|s|} )\right| |d\rho|+\\
&+C\int\limits_{\substack{ {\mathscr Re}(\rho  )=\beta, {\mathscr Im}(\rho  ) >0  \\  |s-\rho |\ge \frac {|s|} {4}} } \left|\frac {W'(\rho )} {W(\rho )}\right| \left|\Theta(\rho -s, Z\sqrt {|s|} )\right| |d\rho|\\
&+C\int\limits_{\substack{ {\mathscr Re}(\rho  )=\beta, {\mathscr Im}(\rho  ) <0 } } \left|\frac {W'(\rho )} {W(\rho )}\right| \left|\Theta(\rho -s, Z\sqrt {|s|} )\right| |d\rho|=I_1+I_2+I_3
\end{align*}
First,
\begin{align*}
I_1\le C\hskip -0.6cm \int\limits_{\substack{ {\mathscr Re}(\rho  )=\beta,  \mathscr Im \rho >0  \\  |s-\rho |\le \frac {|s|} {4}} }\hskip -0.7cm \frac { \left|\Theta(\rho -s, Z\sqrt {|s|} )\right| |d\rho|} {2\rho \log |\rho | }&\le 
\frac {C} {|s|\log |s|} \int\limits _{\substack{ {\mathscr Re}(\sigma )=\beta-{\mathscr Re}(s   ) \\ \mathscr Im \sigma  >-\mathscr Im s,  |\sigma  |\le \frac {|s|} {4}} }\hskip -0.5cm \left|\Theta(\sigma , Z\sqrt {|s|} )\right|d\sigma \\
&\le \frac {C} {|s|\log |s|}\left(Z^2+e^{-a|s|^{1/2}Z} \right).
\end{align*}
Second,
\begin{align*}
I_2\le C\hskip -0.6cm \int\limits_{\substack{ {\mathscr Re}(\rho  )=\beta,  \mathscr Im \rho >0  \\  |s-\rho |\ge \frac {|s|} {4}} }\hskip -0.7cm \frac { \left|\Theta(\rho -s, Z\sqrt {|s|} )\right| |d\rho|} {2\rho \log |\rho | }&\le 
\int\limits_{\substack{ {\mathscr Re}(\rho  )=\beta,  \mathscr Im \rho >0  \\ \mathscr |Im \rho |\le |s|, |s-\rho |\ge \frac {|s|} {4}} }\hskip -0.7cm \frac { \left|\Theta(\rho -s, Z\sqrt {|s|} )\right| |d\rho|} {2\rho \log |\rho | }+\\
+&\hskip -1cm \int\limits_{\substack{ {\mathscr Re}(\rho  )=\beta,  \mathscr Im \rho >0  \\ \mathscr |Im \rho |\ge |s|, |s-\rho |\ge \frac {|s|} {4}} }\hskip -0.7cm \frac { \left|\Theta(\rho -s, Z\sqrt {|s|} )\right| |d\rho|} {2\rho \log |\rho | }=I _{ 2,1 }+I _{ 2,2 }
\end{align*}
where,

\begin{align*}
&I _{ 2,1 } \le 
\int\limits_{\substack{ {\mathscr Re}(\rho  )=\beta,  \mathscr Im \rho >0  \\ \mathscr |Im \rho |\le |s|, |s-\rho |\ge \frac {|s|} {4}} }\hskip -0.7cm \frac { \left|\Theta(\rho -s, Z\sqrt {|s|} )\right| |d\rho|} {2\rho \log |\rho | }
\le C e^{-a|s|}\hskip -0.6cm  \int\limits_{\substack{ {\mathscr Re}(\rho  )=\beta,  \mathscr Im \rho >0  \\ \mathscr |Im \rho |\le |s|, |s-\rho |\ge \frac {|s|} {4}} }\hskip -0.7cm \frac { |d\rho|} {2\rho \log |\rho | }\le C e^{-a'|s|}\\
&I _{ 2,2 }\le \!\!\!\!\!\! \int\limits_{\substack{ {\mathscr Re}(\rho  )=\beta,  \mathscr Im \rho >0  \\ \mathscr |Im \rho |\ge |s|, |s-\rho |\ge \frac {|s|} {4}} }\hskip -0.7cm \frac { \left|\Theta(\rho -s, Z\sqrt {|s|} )\right| |d\rho|} {2\rho \log |\rho | } \le  
\frac {C} {|s|\log |s|} \!\!\!\!\!\! 
\int\limits_{\substack{ {\mathscr Re}(\rho  )=\beta,  \mathscr Im \rho >0  \\ \mathscr |Im \rho |\ge |s|, |s-\rho |\ge \frac {|s|} {4}} }\hskip -0.7cm e^{-a|s-\rho |} |d\rho|\le  \frac {Ce^{-a'|s|}} {|s|\log |s|}.
\end{align*}
\end{proof}

\subsection{Proof of Proposition \ref{S4P2}.} The proof of Proposition \ref{S4P2} is similar to that of Proposition 9.2 in \cite{EV}. However, some small modification is needed because of the slow decay of $U(t, s)$ as $|s|\to \infty$.
An estimate for $\frac {\partial } {\partial s}(\exp {(\tilde F)})\left(\frac {\sigma } {\rho (t)}, \zeta \right) $ similar to (\ref{SA1Lemma14.1E1}) is our first step.
\begin{lem}
\label{SA2Lemma9.3} 
For all $\varepsilon _0>0$ there exists a positive constant $C$ such that, for all  $M>\varepsilon _0$, for all  $\sigma $ such that $\mathscr Re (\sigma /\rho (t) )$ lies in compact subsets of $(0,2)$ and $\varepsilon _0\le |\sigma |\le M $,  and for all $\zeta $ such that $0<|\mathscr Re \zeta |<1$ and
\begin{equation}
\label{SA2Lemma9.3E3} 
|\mathscr Im (\zeta )|=o\left(t^{-1}\right),\,\,t\to 0.
\end{equation}
the following estimate holds,
\begin{equation}
\label{SA2Lemma9.3E2} 
\left| \frac {\partial \tilde F} {\partial s} \left(\frac {\sigma } {\rho (t)}, \zeta \right) e^{\tilde F\left(\frac {\sigma } {\rho (t)}, \zeta \right)}t^{-\zeta }-  \frac {\zeta \rho (t) e^{-\zeta \log\left( 2\log\left|\frac {b\sigma } {\rho (t)}\right|\right) }} {2 \sigma \log |\frac {b\sigma } {2\rho (t)} |  }t^{-\zeta } \right|
\le h_M(t)
\end{equation}
\begin{equation}
\label{SA2Lemma9.3E1} 
h_M(t)=C\left(\rho (t)^2o(t^{-1})+e^{-a' \varepsilon _0 /\rho (t)} \right)e^{\mathcal O\left(t\log M \right)},\,\,\hbox{as}\,t\to 0.
\end{equation}
Moreover, there is $\delta _0>0$, that depends on $\varepsilon _0$ and $M$, such that for all$\zeta $ such that $0<|\mathscr Re (\zeta )|<1$, $|\mathscr Im  (\zeta )|\le \delta _0/t^2$, for $\mathscr Re (\sigma /\rho (t) )$ in compact subsets of $(0, 2)$ and $\varepsilon _0\le |\sigma |\le M $,
\begin{align}
\left| \frac {\partial \tilde F} {\partial s} \left(\frac {\sigma } {\rho (t)}, \zeta  \right)e^{\tilde F\left( \frac {\sigma } {\rho (t)}, \zeta \right)} t^{-\zeta }\right|&\le  
C(1+|\zeta| )\left(t\rho (t)^2t^{-4}+Ce^{-a' \varepsilon _0 /\rho (t)}\right)e^{\mathcal O\left(t\log M \right)}
\label{SA2Lemma9.3E4} 
\end{align}
where the constant $C$ may depend on $\delta _0$ and $\varepsilon _0$ but not on $M$.
\end{lem}
\begin{proof}
We write,
\begin{align}
&\left| \frac {\partial \tilde F} {\partial s} \left(\frac {\sigma } {\rho (t)}, \zeta \right) e^{\tilde F\left(\frac {\sigma } {\rho (t)}, \zeta \right)}t^{-\zeta }-  \frac {\zeta \rho (t)e^{-\zeta \log\left( 2\log\left|\frac {b\sigma } {\rho (t)}\right|\right) }} {2 \sigma \log |\frac {b\sigma } {2\rho (t)} |  }t^{-\zeta } \right| \le A_1+A_2 \label{SA2Lemma9.3E20} \\
&A_1\equiv A_1 \left(\frac {\sigma } {\rho (t)}, \zeta \right) =\left| e^{\tilde F\left(\frac {\sigma } {\rho (t)}, \zeta \right)}-  e^{-\zeta \log\left( 2\log\left|\frac {b\sigma } {\rho (t)}\right|\right) } \right| 
\left| \frac {\zeta \rho (t) t^{-\zeta }} {2 \sigma \log |\frac {b\sigma } {2\rho (t)} |  }\right| \label{SA2Lemma9.3E22} \\
&A_2 \equiv A_2 \left(\frac {\sigma } {\rho (t)}, \zeta \right) =\left| \frac {\partial \tilde F} {\partial s} \left(\frac {\sigma } {\rho (t)}, \zeta \right) -  \frac {\zeta \rho (t)} {2 \sigma \log |\frac {b\sigma } {2\rho (t)} |  } \right|
\left|e^{\tilde F\left(\frac {\sigma } {\rho (t)}\zeta \right)}t^{-\zeta }\right| \label{SA2Lemma9.3E23} 
\end{align}
Let us estimate first $A_1$. To this end,
\begin{align}
&\left|e^{F\left( \frac {\sigma } {\rho (t)}, \frac {\zeta } {\log\left| \sigma/\rho (t) \right|}\right)}-e^{-\zeta \log\left( 2\log\left|\frac {b\sigma } {\rho (t)} \right|\right) } \right|
\le C \left|e^{-\zeta \log\left( 2\log\left|\frac {b\sigma } {\rho (t)} \right|\right) }\right| \times \nonumber \\
&\hskip 2cm \times \left| F\left( \frac {\sigma } {\rho (t)}, \frac {\zeta } {\log\left| \sigma/\rho (t) \right|}\right)+\zeta \log\left( 2\log\left| \frac {b\sigma } {\rho (t)}\right|\right) \right|.\label{SA2Lemma9.3E50}
\end{align}
We first notice, since $\mathscr R e(\sigma) /\rho (t)|$ lies in a compact set, $|u|=|\mathscr R e(\sigma)|\le C\rho (t)\le \varepsilon _0/2$ for $t$ small enough and then $|v|=|\mathscr (\sigma )|\ge \varepsilon _0/2$.  We deduce,
\begin{align}
\mathscr Re \left(\zeta  \log\left( 2\log \left|\frac {b\sigma   } {\rho (t)}\right| \right)\right)=(\mathscr Re \zeta) \log\left| 2\log \left|\frac {b\sigma   } {\rho (t)}\right| \right| \nonumber\\
=(\beta _1-\alpha _1)\log \left(\frac {2} {t}+\log |b\sigma |\right),\,\,\,t\to 0 \nonumber\\
=(\beta _1-\alpha _1)\left( \log\frac {2} {t}+\mathcal O\left(t\log M \right)\right)\,\,t\to 0 \label{SA2PCE1}
\end{align}
Since $ \left|t^{-\zeta }\right|=e^{-(\beta _1-\alpha _1)\log t}$,
we have,
\begin{align}
& \left| t^{-\zeta }e^{-\zeta \log\left( 2\log\left|\frac {b\sigma } {\rho (t)} \right|\right) }\right|\le  Ce^{\mathcal O\left(t\log M \right)}. \label{SA2Lemma9.3E51} 
\end{align}
(Notice that, if $|v|$ is in a bounded set, the term $\log|b\sigma |$ is included in $\mathcal O_1(1)$ and if $|v|>>1$ for large $M$ then $\log|b\sigma |>>1$ too and $\log \left(\log|b\sigma |+\frac {2} {t}+\mathcal O_1(1) \right)>\log \left(\frac {2} {t}+\mathcal O_1(1) \right)$.)

On the other hand, by (\ref{SA1Lemma14.1E1}), if $t$ is small enough,
\begin{align*}
\left| F  \left( \frac {\sigma } {\rho (t)}, \frac {\zeta } {\log\left| \sigma/\rho (t) \right|}\right)+\zeta \log\left( 2\log\left| \frac {b\sigma } {\rho (t)}\right|\right) \right|
&\le C\left(\frac {1} {\log|\frac {\rho (t)} {|\sigma |}|}+\left(\frac {|\zeta |} {\log\left| \sigma/\rho (t) \right|}  \right)^2\right)\\
&\le  C\left(\frac {1} {(\log\varepsilon _0+1/t)}+\frac {|\zeta |^2} {(\log\varepsilon _0+1/t)^2} \right).
\end{align*}
We deduce,
\begin{align}
\left|  e^{\tilde F\left(\frac {\sigma } {\rho (t)}, \zeta \right)}t^{-\zeta }-e^{-\zeta \log\left( 2\log\left|\frac {b\sigma } {\rho (t)}  \right|\right) }t^{-\zeta } \right|\le 
C(t+t^2|\zeta |^2)e^{\mathcal O\left(t\log M \right)}.  \label{SA2Lemma9.3E54}
\end{align}
It follows from (\ref{SA2Lemma9.3E22}), (\ref{SA2Lemma9.3E50}) and (\ref{SA2Lemma9.3E54})
\begin{align}
A_1 \left(\frac {\sigma } {\rho (t)}, \zeta \right) \le \left| \frac {\zeta \rho (t)} {2 \sigma \log |\frac {b\sigma } {2\rho (t)} |  }\right| (t+t^2|\zeta |^2)e^{\mathcal O\left(t\log M \right)}\\
\le \left| \frac {t \zeta \rho (t)} {2 \varepsilon _0  }\right| (t+t^2|\zeta |^2)e^{\mathcal O\left(t\log M \right)}
\end{align}
In order to estimate $A_2$ we  first use (\ref{SA2Lemma9.3E54}) and (\ref{SA2Lemma9.3E51} ) to get
\begin{align}
\left|e^{\tilde F\left(\frac {\sigma } {\rho (t)}\zeta \right)}t^{-\zeta }\right| &\le C(t+t^2|\zeta |^2)e^{\mathcal O\left(t\log M \right)}+\left|e^{-\zeta \log\left( 2\log\left|\frac {b\sigma } {\rho (t)} \right|\right) }t^{-\zeta }\right| \nonumber\\
&\le C(1+t+t^2|\zeta| ^2 )e^{\mathcal O\left(t\log M \right)}.\label{SA2JPM27}
\end{align}

Since, from  (\ref{SA2partialU}),
\begin{align}
\label{SA2JPM28}
\left| \frac {\partial \tilde F} {\partial s} \left(\frac {\sigma } {\rho (t)}, \zeta \right)  \right|\le 
\frac {C\rho (t)^2|\zeta |^2} {|\sigma |^2\log |\sigma /\rho (t)|}+Ce^{-a'|\sigma /\rho (t)|}
\end{align}
it follows,
\begin{align*}
A_2 \left(\frac {\sigma } {\rho (t)}, \zeta \right) \le C(1+|\zeta |^2t^2)\left( \frac {C\rho (t)^2|\zeta |^2} {|\sigma |^2\log |\sigma /\rho (t)|}+Ce^{-a'|\sigma /\rho (t)|} \right) e^{\mathcal O\left(t\log M \right)}\\
\le C(1+|\zeta |^2t^2)\left( t\rho (t)^2|\zeta |^2+e^{-a'\varepsilon _0/\rho (t)|} \right) e^{\mathcal O\left(t\log M \right)}
\end{align*}
If we suppose that $|\zeta |=o(t^{-1})$, we deduce,  (\ref{SA2Lemma9.3E2}) with 
\begin{equation}
h_M(t)\underset{t\to 0}{=}C\left(\rho (t)^2o(t^{-1})+e^{-a' \varepsilon _0 /\rho (t)} \right)e^{\mathcal O\left(t\log M \right)}.
\end{equation}
If we only assume $|\zeta |\le \delta _0 t^{-2}$, then, by (\ref{SA2JPM27}),
\begin{align*}
\left|  \frac {\partial \tilde F} {\partial s} \left(\frac {\sigma } {\rho (t)}, \zeta \right) e^{\tilde F\left(\frac {\sigma } {\rho (t)}\zeta \right)}t^{-\zeta }\right| &\le C(1+|\zeta| )\times \\
&\times \left( \frac {C\rho (t)^2t^{-4}} {|\sigma |^2\log |\sigma /\rho (t)|}+Ce^{-a'|\sigma /\rho (t)|}\right)e^{\mathcal O\left(t\log M \right)}\\
&\le C(1+|\zeta| )\left(t\rho (t)^2t^{-4}+Ce^{-a' \varepsilon _0 /\rho (t)}\right)e^{\mathcal O\left(t\log M \right)}
\end{align*}
which proves (\ref{SA2Lemma9.3E4}).
\end{proof}

\begin{lem}
\label{SA2Lemma9.5}
For all positive constant $\varepsilon _0>0$
\begin{align}
&\lim _{ t\to 0 }\rho (t)^{-1}\Bigg |\frac {\partial } {\partial s}U\left(t, \frac {\sigma } {\rho (t)} \right) -H\left(t, \frac {\sigma } {\rho (t)}\right) \Bigg|=0.\label{SA2Lemma9.5E1}\\
&H\left(t, \frac {\sigma } {\rho (t)}\right)=-\frac {1} {2i\pi }\frac {t} {\sqrt{2\pi }}  \int  \limits_{ {\mathscr Re}(\zeta )  =(\beta  _1-\alpha  _1) }
\!\!\! \frac {\zeta \rho (t)e^{\left(- \zeta  \log\left( 2\log \left|\frac {b\sigma   } {\rho (t)}\right| \right)\right)}} {2\sigma \log |\frac {b\sigma } {2\rho (t)} | } t^{-\zeta }\Gamma (\zeta )d\zeta \label{SA2Lemma9.5E1B}\
\end{align}
\begin{align}
&\lim _{ t\to 0 }\rho (t)^{-1}\Bigg |\frac {\partial } {\partial s}U_t\left(t, \frac {\sigma } {\rho (t)} \right) -H_1\left(t, \frac {\sigma } {\rho (t)}\right) \Bigg|=0.\label{SA2Lemma9.5E1Dt}\\
&H_1\left(t, \frac {\sigma } {\rho (t)}\right)=-\frac {1} {2i\pi }\frac {t} {\sqrt{2\pi }}  \int  \limits_{ {\mathscr Re}(\zeta )  =(\beta  _1-\alpha  _1) }
\!\!\! \frac {\zeta \rho (t)e^{\left(- \zeta  \log\left( 2\log \left|\frac {b\sigma   } {\rho (t)}\right| \right)\right)}} {2\sigma \log |\frac {b\sigma } {2\rho (t)} | } t^{-\zeta +1}\Gamma (\zeta +1 )d\zeta \label{SA2Lemma9.5E1BDt}\
\end{align}
uniformly  for $\mathscr R e(\sigma )/\rho (t)$ in  compact subsets of $(0, 2)$ and $|\sigma |\in (\varepsilon _0, M(t))$ for $M(t)> \varepsilon _0$ such that  $\log M(t)\in (0, t^{-\theta})$ for some 
$\theta \in (1, 2)$.
\end{lem}
\begin{proof}
From (\ref{S5EUX1}), 
\begin{equation}
\label{S5EUXA1}
U(t, s)=-\frac {1} {\sqrt{2\pi}}\int  _{ \mathscr Re(\zeta  )=\beta - \mathscr Re(s)  \ } e^{\tilde F(s, \zeta )} t^{-\zeta }\Gamma (\zeta )d\zeta ,\,\,\,\forall \beta \in (0, 2);\,\, \beta -1<c<\beta,
\end{equation}
It follows,
\begin{align*}
\frac {\partial U} {\partial s}(t, \sigma /\rho (t))=\frac {1} {\sqrt{2\pi }} \int  _{ {\mathscr Re}(\zeta )  =\beta _1-\alpha _1} \frac {\partial } {\partial s}\left
(e^{\tilde F(\sigma /\rho (t), \zeta )} \right)\Gamma (\zeta ) t^{-\zeta }d\zeta\\
=\frac {1} {\sqrt{2\pi }} \int  _{ {\mathscr Re}(\zeta )  =\beta _1-\alpha _1}  \frac {\partial \tilde F} {\partial s} (\sigma /\rho (t), \zeta ) e^{\tilde F(\sigma /\rho (t), \zeta )}\Gamma (\zeta ) t^{-\zeta }d\zeta.
\end{align*}
and, we may then write,

\begin{align}
\frac {\partial } {\partial s}U\left(t, \frac {\sigma } {\rho (t)} \right)&=\frac {1} {\sqrt{2\pi }} \int  _{ {\mathscr Re}(\zeta )  =\beta _1-\alpha _1}
\frac {\partial \tilde F} {\partial s} (\sigma /\rho (t), \zeta ) e^{\tilde F(\sigma /\rho (t), \zeta )} \Gamma (\zeta ) t^{-\zeta }d\zeta\nonumber\\
&=\frac {1} {\sqrt{2\pi }} \int\limits  _ {\substack{ {\mathscr Re}(\zeta )=\beta _1-\alpha _1 \\ \mathscr Im \zeta \le  \frac {1} {t^2|\log t|}} }(\cdots)d\zeta +
\frac {1} {\sqrt{2\pi }} \int\limits  _ {\substack{ {\mathscr Re}(\zeta )=\beta _1-\alpha _1 \\ \frac {1} {t^2|\log t|}\le \mathscr Im \zeta \le  \frac {\delta _0} {t^2}} }(\cdots)d\zeta+ \nonumber\\
&+\frac {1} {\sqrt{2\pi }} \int\limits  _ {\substack{ {\mathscr Re}(\zeta )=\beta _1-\alpha _1 \\  \mathscr Im \zeta \ge  \frac {\delta _0} {t^2}} }(\cdots)d\zeta =J_1+J_2+J_3.\label{SA2Lemma9.5E2}
\end{align}
We now write,
\begin{align}
J_1&=\frac {1} {\sqrt{2\pi }} \int\limits  _ {\substack{ {\mathscr Re}(\zeta )=\beta _1-\alpha _1 \\ \mathscr Im \zeta \le \frac {1} {t^2 |\log t|} } }
\left| \frac {\partial \tilde F} {\partial s} (\sigma /\rho (t), \zeta ) e^{\tilde F(\sigma /\rho (t), \zeta )}-  \frac {\zeta \rho (t)e^{-\zeta \rho (t) \log\left( 2\log\left|\frac {b\sigma } {\rho (t)}\right|\right) }} {2 \sigma\log |\frac {b\sigma /\rho (t)} {2} |  } \right| \Gamma (\zeta ) t^{-\zeta }d\zeta+\nonumber\\
&+\frac {1} {\sqrt{2\pi }}  \int  \limits_{ {\mathscr Re}(\zeta )  =(\alpha _1-\beta _1) }
\!\!\! \frac {\zeta \rho (t) e^{-\zeta \log\left( 2\log\left|\frac {b\sigma } {\rho (t)}\right|\right) }} {2 \sigma \log |\frac {b\sigma } {2\rho (t)} |  }t^{-\zeta }\Gamma (\zeta )d\zeta-\nonumber\\
&-\frac {1} {\sqrt{2\pi }}  \int  \limits_{\substack{ {\mathscr Re}(\zeta )=\beta _1-\alpha _1 \\ \mathscr Im \zeta \ge  \frac {1} {t^2|\log t|}} }
\!\!\! \frac {\zeta \rho (t) e^{-\zeta \log\left( 2\log\left|\frac {b\sigma } {\rho (t)}\right|\right) }} {2 \sigma \log |\frac {b\sigma } {2\rho (t)} |  } t^{-\zeta }\Gamma (\zeta )d\zeta=J _{ 1, 1 }+J _{ 1, 2 }+J _{ 1, 3 } \label{SA2Lemma9.5E3}
\end{align}
In the third integral in the right hand side of (\ref{SA2Lemma9.5E3}) we use (\ref{SA2PCE1}) and
\begin{equation*}
\left|t^{-\zeta }\right|=e^{-(\beta _1-\alpha _1)\log t},\,\,\,\,|\Gamma (\zeta )|\le Ce^{-\frac {\pi |\zeta |} {2}}
\end{equation*}
and obtain
\begin{align*}
&\left| \zeta  e^{\left(- \zeta  \log\left( 2\log \left|\frac {b\sigma   } {\rho (t)}\right| \right)\right)} t^{-\zeta }\Gamma (\zeta )  \right| \le Ce^{\mathcal O\left(t\log M \right)}|\zeta| e^{-\frac {\pi |\zeta |} {2}}\\
&\rho (t)^{-1}|J _{ 1, 3 }|\le Ce^{\mathcal O\left(t\log M \right)}\int  _{ |\zeta |\ge \frac {1} {t^2|\log t|} }|\zeta| e^{-\frac {\pi |\zeta |} {2}}|d\zeta \le Ce^{\mathcal O\left(t\log M \right)}e^{-\frac {\pi } {4t^2|\log t|}}
\end{align*}
from where it follows that $\rho (t)^{-1}J _{ 1, 3 }\to 0$ as $t\to 0$ uniformly  for $\mathscr R e(\sigma )/\rho (t)$ in  compact subsets of  $(0, 2)$, $|\sigma |\in (\varepsilon _0, M)$, $\log M\in (0, t^{-\theta})$.

The first integral in the right hand side of (\ref{SA2Lemma9.5E3}) is estimated using Lemma (\ref{SA2Lemma9.3}). By (\ref{SA2Lemma9.3E1}),
\begin{align*}
|J _{ 1,1 }|\le h_M(t)\hskip -0.5cm  \int\limits  _ {\substack{ {\mathscr Re}(\zeta )=\beta _1-\alpha _1 \\ \mathscr Im \zeta \le \frac {1} {t^2 |\log t|} } } \hskip -0.3cm\left|\Gamma (\zeta ) t^{-\zeta }\right| |d\zeta|
\end{align*}

The term $J_2$ is bounded using (\ref{SA2Lemma9.3E4}),
\begin{align*}
|J _{ 2 }|\le C \left(t\rho (t)^2t^{-4}+Ce^{-a' \varepsilon _0 /\rho (t)}\right)e^{\mathcal O\left(t\log M \right)}\hskip -0cm \int _ {\substack{ {\mathscr Re}(\zeta )=\beta _1-\alpha _1 \\ \frac {1} {t^2|\log t|}\le \mathscr Im \zeta \le  \frac {\delta _0} {t^2}} }
(1+|\zeta| )e^{\frac {-\pi |\zeta |} {2}}|d\zeta|
\end{align*}
and therefore,
\begin{align*}
\lim _{ t\to 0} \rho (t)^{-1}|J _{ 2 }|=0
\end{align*}
uniformly fo $|\sigma |\in (\varepsilon _0, M)$, $\log M\in (0, t^{-\theta})$.

In order to bound  $J_3$ we use the properties of the function  $B(s)$  in Proposition \ref{S3PBP} and Proposition \ref{S3PBP24B}.
It follows from  Proposition \ref{S3PBP} that for $\mathscr Re (s)\in (0, 2)$, $|B(s)|>0$. Then, for all constant $R>0$ there exists $C_R>0$ such that
$$
|B(s)|\ge C_R\,\,\,\forall s,\,|s|\le R.
$$
On the other hand, by Proposition (\ref{S3PBP24B}),
$$
|B(s)|\ge C|\log |s||,\,\,\,\forall s,\, |s|\ge R
$$
It  follows that, for $\mathscr R e(s)$ on any compact subset of $(0, 2)$, the function $|B(s)|$ is uniformly bounded from below by a positive constant.
Then, for $|\zeta |\ge \delta _0/t^2$, and $t$ small,
\begin{align*}
&\left|\frac {\partial \tilde F} {\partial s} (\sigma /\rho (t), \zeta ) e^{\tilde F(\sigma /\rho (t), \zeta )} \Gamma (\zeta ) t^{-\zeta }\right|=
\left|\frac {\partial \tilde F} {\partial s} (\sigma /\rho (t), \zeta )\right|\left|\frac {B\left(\frac {\sigma  } {\rho (t)}\right) } {B\left(\frac {\sigma } {\rho (t)} +\zeta \right )} \Gamma (\zeta ) 
t^{-\zeta }\right| \\
&\le 
C(1+|\zeta| )\left(\frac {\rho (t)^2t^{-4}} {|\sigma |^2\log |\sigma /\rho (t)|}+e^{-a'|\sigma /\rho (t)|}\right)|\log|\sigma /\rho (t)|| e^{-\frac {|\pi ||\zeta |} {2}}e^{-(\beta _1-\alpha _1)\log t}\\
&\le C(1+|\zeta| )\left(\frac {\rho (t)^2t^{-4}} {\varepsilon _0^2(t^{-1}+\log \varepsilon _0) }+e^{-a'|\sigma /\rho (t)|}\right)(\log M+t^{-1})e^{-\frac {|\pi ||\zeta |} {2}}e^{(\beta _1-\alpha _1)|\log t|}\\
&\le C(1+|\zeta| )\left(\rho (t)^2t^{-4} +e^{-a' \varepsilon _0 /\rho (t)}\right)(\log M+t^{-1})e^{-\frac {|\pi |} {4 t^2}}e^{(\beta _1-\alpha _1)|\log t|} e^{-\frac {|\pi ||\zeta |} {4}}.
\end{align*}
and
\begin{align*}
|J_3|&\le C\left(\rho (t)^2t^{-4} +e^{-a' \varepsilon _0 /\rho (t)}\right)(\log M+t^{-1})e^{-\frac {|\pi |} {4 t^2}} \int\limits _ {\substack{ {\mathscr Re}(\zeta )=\beta _1-\alpha _1 \\  \mathscr Im \zeta \ge  \frac {\delta _0} {t^2}} }(1+|\zeta| )e^{-\frac {|\pi ||\zeta |} {4}}d\zeta\\\
&\le C\left(\rho (t)^2t^{-4} +e^{-a' \varepsilon _0 /\rho (t)}\right)(\log M+t^{-1})e^{-\frac {|\pi |} {4 t^2}}
\end{align*}
Therefore,
\begin{align*}
\lim _{ t\to 0 }\rho (t)^{-1}|J_3|=0.
\end{align*}
uniformly fo $|\sigma |\in (\varepsilon _0, M)$, $\log M\in (0, t^{-\theta})$.

It may be proceeded in a similar way with $U_t$ since,
\begin{align}
\label{S5EUXA1Dt}
U_t(t, s)=\frac {1} {\sqrt{2\pi}}\int\limits  _{ \mathscr Re(\zeta  )=\beta - \mathscr Re(s)  \ } e^{\tilde F(s, \zeta )} t^{-\zeta -1}\Gamma (\zeta+1 )d\zeta ,\,\,\,\forall \beta \in (0, 2);\,\, \beta -1<c<\beta
\end{align}
and then,
\begin{align*}
\frac {\partial U_t} {\partial s}(t, \sigma /\rho (t))=\frac {1} {\sqrt{2\pi }} \int  _{ {\mathscr Re}(\zeta )  =\beta _1-\alpha _1}  \frac {\partial \tilde F} {\partial s} (\sigma /\rho (t), \zeta ) e^{\tilde F(\sigma /\rho (t), \zeta )}\Gamma (\zeta+1 ) t^{-\zeta -1 }d\zeta.
\end{align*}
from where (\ref{SA2Lemma9.5E1Dt})  follows with the same arguments that gave (\ref{SA2Lemma9.5E1}).
\end{proof}

\begin{lem}
\label{SA2Lemma9.6}
\begin{align*}
&H\left(t, \frac {\sigma } {\rho (t)} \right)
=-\frac {t \rho (t) } {2\sigma  }\exp\left(-2t\log \left|\frac {b\sigma } {\rho (t)} \right|  \right).\\
&H_1\left(t, \frac {\sigma } {\rho (t)} \right)= \frac {\partial H} {\partial t}\left(t, \frac {\sigma } {\rho (t)} \right)
\end{align*}
\end{lem}
\begin{proof}
The integral in (\ref{SA2Lemma9.5E1B}) can  be computed adding the residues of the integrand at the poles $\zeta =-n$ of the Gamma function,
\begin{align*}
H\left(t, \frac {\sigma } {\rho (t)} \right)
=&\frac {\rho (t) } {2\sigma  \log |\frac {b\sigma /\rho(t)} {2} |} \sum _{ n=0 }^\infty \frac {(-1)^nt^n} {n!} n \exp \left( n  \log\left( 2\log \left|\frac { b\sigma   } {\rho (t)}\right|  \right)\right)\\
= &-\frac {t \rho (t) } {2\sigma  }\exp\left(-2t\log \left|\frac {b \sigma  } {\rho (t)} \right|  \right).
\end{align*}
On the other hand,
\begin{align*}
H_1\left(t, \frac {\sigma } {\rho (t)} \right)
=&\frac {\rho (t) } {2\sigma  \log |\frac {b\sigma /\rho(t)} {2} |} \sum _{ n=0 }^\infty \frac {(-1)^nt^n} {n!} (n+1) \exp \left( (n+1)  \log\left( 2\log \left|\frac { b\sigma   } {\rho (t)}\right|  \right)\right)\\
=&\frac {\exp\left(-2t\log \left|\frac {b \sigma  } {\rho (t)} \right|  \right)\rho (t)} {2\sigma }- 2\log \left|\frac { b\sigma   } {\rho (t)}\right|\frac {t\,\exp\left(-2t\log \left|\frac {b \sigma } {\rho (t)} \right|  \right)\rho (t)} {2\sigma }\\
=& \frac {\partial H} {\partial t}\left(t, \frac {\sigma } {\rho (t)} \right).
\end{align*}
\end{proof}
\begin{prop}
\begin{equation*}
\mathscr M^{-1}( H(t))(X)=-\frac {2t} {\pi }\Gamma (-2t)\sin (\pi t) |X|^{2t}{\rm sign}(X ).
\end{equation*}
\end{prop}
\begin{proof}
If we call $X=\rho (t)Y$,
\begin{align*}
\mathscr M^{-1}( H(t))(X)&=\frac {1} {2i\pi }\int  _{ {\mathscr Re}(s) =\alpha _1  }H(t, s)
e^{-s \rho (t) Y } ds\\
&=\frac {1} {2i\pi \rho (t)}\int  _{ {\mathscr Re}(\sigma ) =\alpha _1 \rho (t)  }H\left(t, \frac {\sigma } {\rho (t)}\right) e^{-\sigma  Y } d\sigma  \\
&=\frac {t} {4i\pi}\int  _{ {\mathscr Re}(\sigma) =\alpha _1 \rho (t) }\sigma ^{-1}\exp\left(-2t\log \left|\frac {b\sigma  } {\rho (t)} \right| \right)  
e^{-\sigma Y } d\sigma
\end{align*}
we  deform the integration contour to ${\mathscr Re}(\sigma) =0$, and change variables $bv \to v $,
\begin{align*}
\frac {t} {4i\pi }\int _{ {\mathscr Re}(\sigma) =0 }\!\!\!\!\!\!\!\!\sigma ^{-1}e^{\left(-2t\log \left|\frac {bv } {\rho (t)} \right| \right) }
e^{-iv Y } d\sigma=\frac {t} {4\pi }\int  _{\RR}  v ^{-1}e^{\left(-2t\log \left|\frac {v } {\rho (t)} \right| \right) }
e^{-iv \frac {Y} {b} } dv
\end{align*}
Then, after  the change of variables $v=\rho (t)w$, $dv=\rho (t)dw$,
\begin{align*}
\frac { t} {4\pi }\int  _{ \RR} v^{-1}\exp\left(-2t\log \left|\frac {v } {\rho (t)} \right| \right)  e^{-iv \frac {Y} {b} }dv&=
\frac {t} {4\pi  }\int  _{ \RR} v^{-1}\exp\left(-2t\log \left|w \right| \right)  e^{-iw \frac {\rho (t)Y} {b} }dw\\
&=-\frac {2t} {\pi }\Gamma (-2t)\sin (\pi t) |X|^{2t}{\rm sign}(X )
\end{align*}
\end{proof}

\begin{proof}
[\upshape\bfseries{Proof of Proposition \ref{S4P2}}]
We use (\ref{S3ELXZ})  to write the left hand side of (\ref{S3Cor17E1}) as,
\begin{align*}
 t^{-1}|X |^{1-2t}\tilde\Lambda (t, X)&= t^{-1}|X |^{1-2t}X^{-1}\left(X\tilde\Lambda (t, X)\right)\\
&=\frac {1} {2i\pi }t^{-1}|X |^{1-2t}X^{-1} \int  _{ \mathscr Rr(s)=\alpha _1 }\frac {\partial U} {\partial s}(t, s)e^{-sX}ds.
\end{align*}
For $X=\rho (t)Y$,

\begin{align}
\int  _{ \mathscr Rr(s)=\alpha _1 }\frac {\partial U} {\partial s}(t, s)e^{-sX}ds&=\frac {1} {2i\pi \rho (t) }\int \limits _{{\mathscr Re}(\sigma)  =\alpha _1 \rho (t) }\frac {\partial U} {\partial s}\left(t, \frac {\sigma } {\rho (t)}\right)e^{-\sigma Y }d\sigma \\
\int \limits _{{\mathscr Re}(\sigma)  =\alpha _1 \rho (t) }\frac {\partial U} {\partial s}\left(t, \frac {\sigma } {\rho (t)}\right)e^{-\sigma Y }d\sigma&=I_1+I_2+I_3\label{SA2PropB1A}\\
I_k=\frac {1} {2i\pi  }\int\limits  _ {\substack{ {\mathscr Re}(\sigma  )=\alpha _1 \rho (t)\\ \sigma \in D_k } } &\frac {\partial U} {\partial s}\left(t, \frac {\sigma } {\rho (t)}\right)  e^{-\sigma Y }d\sigma\\
D_1=B _{ \varepsilon _0 }(0),&\,D_2= B_{M(t)}(0)\setminus B _{ \varepsilon _0 }(0),\, D_3=B_{M(t)}(0)^c
\end{align}
where $\log M(t)=t^{-3/2}$. On $D_1$ and $D_3$ we use (\ref{S3PX1E2}) of Proposition \ref{S3PX1},
\begin{align*}
\left|\frac {\partial  U} {\partial s}\left(t, \frac {\sigma } {\rho (t)}\right) \right|& \le C_Tte^{-2t\log (|b\sigma /\rho (t)|)}\left(1+\left|\frac {\sigma } {\rho (t)}\right|\right)^{-1}\nonumber \\
&\le Cte^{-2t \log |bv|}e^{2t\log (\rho (t))}\left(1+\left|\frac {\sigma } {\rho (t)}\right|\right)^{-1}\le Ct\rho (t)|\sigma |^{-2t-1}.
\end{align*}
from where,
\begin{align}
&|I_1|\le Ct\rho (t) \varepsilon _0 \label{SA2PropB1E1}\\
&|I_3|\le C\rho (t) M(t)^{-2t}. \label{SA2PropB1E2}
\end{align}

On $D_2$
\begin{align*}
&I_2=I _{ 2, 1 }+I _{ 2,2 }\\
&I _{ 2, 1 }=\frac {1} {2i\pi  }\int\limits  _ {\substack{ {\mathscr Re}(\sigma  )=\alpha _1 \rho (t)\\ \sigma \in D_2 } } \left(\frac {\partial U} {\partial s}\left(t, \frac {\sigma } {\rho (t)}\right)-H\left(t, \frac {\sigma } {\rho (t)}\right) \right)e^{-\sigma Y }d\sigma\\
&I _{ 2,2 }=\frac {1} {2i\pi }\int\limits  _ {\substack{ {\mathscr Re}(\sigma  )=\alpha _1 \rho (t)\\ \sigma \in D_2 } } H\left(t, \frac {\sigma } {\rho (t)}\right)e^{-\sigma Y }d\sigma
\end{align*}
The first integral is estimated as
\begin{align*}
|I _{ 2, 1 }|\le \frac {1} {2i\pi  }\int\limits  _ {\substack{ {\mathscr Re}(\sigma  )=\alpha _1 \rho (t)\\ \sigma \in D_2 } } \left|\frac {\partial U} {\partial s}\left(t, \frac {\sigma } {\rho (t)}\right)-H\left(t, \frac {\sigma } {\rho (t)}\right) \right| |d\sigma|.
\end{align*}
and by Lemma \ref{SA2Lemma9.5}
\begin{equation}
\label{SA2PropB1E3}
\lim _{ t\to 0 }\rho(t) ^{-1}|I _{ 2, 1 }|=0.
\end{equation}
We write the second as

\begin{align}
\left|I _{ 2, 2 }- \frac {1} {2i\pi }\int\limits  _ {\substack{ {\mathscr Re}(\sigma  )=\alpha _1 \rho (t) } } H\left(t, \frac {\sigma } {\rho (t)}\right)e^{-\sigma Y }d\sigma\right|&\le
C\left| \int\limits  _  {\substack{ {\mathscr Re}(\sigma  )=\alpha _1 \rho (t)\\ \sigma \in D_1 } } H\left(t, \frac {\sigma } {\rho (t)}\right)e^{-\sigma Y }d\sigma\right|+\nonumber\\
+&C\left| \int\limits  _  {\substack{ {\mathscr Re}(\sigma  )=\alpha _1 \rho (t)\\ \sigma \in D_3 } } H\left(t, \frac {\sigma } {\rho (t)}\right)e^{-\sigma Y }d\sigma\right|
\label{SA2PropB1E4}
\end{align}
and the explicit expression of $H(t)$ gives,

\begin{align}
&\left| \int\limits  _  {\substack{ {\mathscr Re}(\sigma  )=\alpha _1 \rho (t)\\ \sigma \in D_1 } } H\left(t, \frac {\sigma } {\rho (t)}\right)e^{-\sigma Y }d\sigma\right| \le Ct\rho (t)\varepsilon_0\label{SA2PropB1E5}\\
&\left| \int\limits  _  {\substack{ {\mathscr Re}(\sigma  )=\alpha _1 \rho (t)\\ \sigma \in D_3 } } H\left(t, \frac {\sigma } {\rho (t)}\right)e^{-\sigma Y }d\sigma\right|\le 
Ct\rho (t)M(t)^{-2t} \label{SA2PropB1E6}
\end{align}
by similar calculations as those giving  (\ref{SA2PropB1E1}) and (\ref{SA2PropB1E2}).

It follows from (\ref{SA2PropB1A}) and (\ref{SA2PropB1E1})--(\ref{SA2PropB1E6}) that for all $\varepsilon _0>$ there exists $\tau $ small enough such that, for all $t\in (0, \tau )$ and all $Y\ge 0$, 
\begin{equation*}
t^{-1}\rho (t)^{-1}\left(|I_1|+|I_3|+|I _{ 2, 1 }|+\left|I _{ 2,2 }-\rho (t)(\mathscr M^{-1}(H(t))(\rho (t)Y)\right|\right)\le C\left(\varepsilon _0+t^{-1}M(t)^{-2t}\right)
\end{equation*}
and then, uniformly on $Y\in \RR$,
\begin{align}
\lim _{ t\to 0 }t^{-1}\rho (t)^{-1}\left(|I_1|+|I_3|+|I _{ 2, 1 }|+\left|I _{ 2,2 }-\rho (t)(\mathscr M^{-1}(H(t))(\rho (t)Y)\right| \right)=0,\label{SA2PropB1E7}
\end{align}

Therefore, since for $X=\rho (t)Y$
\begin{align*}
&\int  _{ \mathscr Re(s)=\alpha _1 }\frac {\partial U} {\partial s}(t, s)e^{-sX}ds=\rho (t)^{-1}(I_1+I_2+I_3)\\
&=\rho (t)^{-1}(I_1+I_3+I _{ 2,1 }+\left(I _{ 2,2 }-\rho (t)(\mathscr M^{-1}(H(t))(X)\right))+(\mathscr M^{-1}(H(t))(X)
\end{align*}

\begin{align*}
&t^{-1}X^{-1}|X|^{1-2t}\int  _{ \mathscr Re(s)=\alpha _1 }\frac {\partial U} {\partial s}(t, s)e^{-sX}ds=t^{-1}X^{-1}|X|^{1-2t}\rho (t)^{-1}\left(I_1+I_3+I _{ 2,1 }+\right.\\
&\left.+\left(I _{ 2,2 }-\rho (t)\mathscr M^{-1}(H(t)\right)(X)\right)+t^{-1}X^{-1}|X|^{1-2t}\mathscr M^{-1}(H(t))(X)
\end{align*}
and by (\ref{SA2PropB1E7})  we deduce, 
\begin{align*}
\lim _{ t\to 0 }t^{-1}X^{-1}|X|^{1-2t}&\int  _{ \mathscr Re(s)=\alpha _1 }\frac {\partial U} {\partial s}(t, s)e^{-s\rho (t)Y}ds=\\
=&\lim _{ t\to 0 }t^{-1}X^{-1}|X|^{1-2t}\mathscr M^{-1}(H(t))(X)=1.
\end{align*}
uniformly for $X$ in bounded subsets of $\RR$.
\end{proof}
\subsection{Linearization.}
When $R(p, p_1, p_2)\!-\!R(p_1, p, p_2)\!-\!R(p_2, p_1, p)$ is written in terms of the new function $\Omega $ defined in (\ref{S2Elinzn2}) and only 
linear terms with respect to $\Omega $ are kept, the resulting equation is

\begin{align}
n_0(1+n_0)\frac {\partial \Omega (t)} {\partial t}=&\, L _{ I_3 }(\Omega (t)) \label{S1ERRR0}\\
L _{ I_3 }(\Omega (t))=&\int _0^\infty \left( \mathscr U(k, k')\Omega (t, k')-\mathscr V(k, k')\Omega (t, k)\right)k'^2dk', \label{S1ERRR}\\
\frac{1}{8 n_\mathrm{c} a^2 m^{-2}}\mathscr{U}(k, k') &= 
\Big[ \frac{m \theta(k-k')}{k k'}  \nonumber \\
 & \times 
 n_0(\omega(k))[1+n_0(\omega(k'))][1+n_0(\omega(k)-\omega(k'))] + (k \leftrightarrow k')  \Big] \nonumber \\ 
 &-\frac{m}{k k'} \, n_0(\omega(k)+\omega(k'))[1+n_0(\omega(k))][1+n_0(\omega(k'))] , \label{S1ERRRU}\\
\frac{1}{8 n_\mathrm{c} a^2 m^{-2}}\mathscr{V}(k, k') &= 
 \frac{m \theta(k-k')}{k k'}
 n_0(\omega(k))[1+n_0(\omega(k'))][1+n_0(\omega(k)-\omega(k'))] \nonumber\\
 &+\frac{m \theta(k'-k)}{k k'}n_0(\omega(k'))[1+n_0(\omega(k))][1+n_0(\omega(k')-\omega(k))]  \label{S1ERRRV}
 \end{align}

The functions  $\mathscr U(k, k')$ and $\mathscr V(k, k')$ have a non integrable singularity along the diagonal $k=k'$.  However, these singularities cancel each other when the two terms are combined as in (\ref{S1ERRR}) as far as it is assumed that, for all $t>0$, $\Omega (t)\in C^\alpha (0, \infty)$ for some $\alpha >0$.  But the integrand $\left( \mathscr U(k, k')\Omega (t, k')-\mathscr V(k, k')\Omega (t, k)\right)$ can not be split as for the linearized  Boltzmann equations for classical particles (\cite{CIP}). However, an explicit calculation shows that, for all $k>0$,
\begin{equation}
\label{S1ERRE}
L _{ I_3 }(\omega )(k)=\int _0^\infty \left( \mathscr U(k, k') k'^2-\mathscr V(k, k')k^2\right)k'^2dk'=0
\end{equation}
from where we deduce, for all $k>0$,
\begin{align*}
\int _0^\infty &\left( \mathscr U(k, k') \frac {k'^2} {k^2}\Omega (t, k)-\mathscr V(k, k')\Omega (t, k)\right)k'^2dk'
=\frac {\Omega (t, k) } {k}L _{ I_3 }(\omega )(k)=0.
\end{align*}
We may then write,
\begin{align*}
L _{ I_3 }(\Omega (t))&=\int _0^\infty \left( \mathscr U(k, k')\Omega (t, k')-\mathscr V(k, k')\Omega (t, k)\right)k'^2dk'\\
&=\int _0^\infty  \mathscr U(k, k')\left( \frac {\Omega (t, k')} {k'^2}-  \frac {\Omega (t, k)} {k^2}\right)k'^2k'^2dk'.
\end{align*}
In terms of the new variables (\ref{S3EMK2}), for $u$ a regular function, the simplified equation (\ref{S3EMK2})  may be written as follows,

\begin{align}
\int _0^\infty (u(y)-u(x))K(x, y)dy)&=\int _0^\infty \int_x^y\frac {\partial  u} {\partial z}(z)dzK(x, y)dy\nonumber\\
=-\int _0^x \frac {\partial  u} {\partial z}(z)& \int _0^zK(x, y)dydz+\int _x^\infty \frac {\partial  u} {\partial z}(z) \int _z^\infty K(x, y)dydz \label{S4EHL007}
\end{align}
and this gives equation (\ref{S4E1}) with
\begin{align}
H\left(\frac {x} {z} \right)&=\1 _{ z>x }\int _z^\infty K(x, y)dy-\1 _{ 0<z<x }\int _0^zK(x, y)dy\label{S4EHL0}
\end{align}
where an explicit integration gives (\ref{S4EHL}).

\section*{Acknowledgements} The research of the author is supported by grants MTM2017-82996-C2-1-R of MINECO and IT1247-19 of the Basque Government. The hospitality of IAM of the University of Bonn, and its support through SFB 1060 are gratefully acknowledged. The author  thanks Pr. M. Valle at the UPV/EHU   for enlightening discussions and  comments.


\begin{thebibliography}{AAA}

\bibitem{AN} L. Arkeryd, A. Nouri.
\newblock Bose Condensates in Interaction with Excitations: A Kinetic Model
\newblock {\it Commun. Math. Phys.}
\newblock {\bf 310},  765--788, (2012)


\bibitem{BZ}  A. M. Balk, V. E. Zakharov. 
\newblock Stability of Weal Turbulence Kolmogorov Spectra, 
\newblock {\it Amer. Math. Soc. Transl.} 
\newblock  {\bf 182},  31--81, (1998)

\bibitem{Bi} M. J. Bijlsma, E. Zaremba, and H. T. C. Stoof. 
\newblock Condensate growth in trapped Bose gases, 
\newblock {\it Phys. Rev. A,} 
\newblock {\bf 62}, 063609 (2000)

\bibitem {CIP}
C. Cercignani, R.  Illner and M. Pulvirenti.
\newblock {\it The Mathematical Theory of Dilute Gases},  Applied Mathematical Sciences Vol. 106. 
\newblock Springer-Verlag New York, (1994)

\bibitem{CR}
Z. Cheng and S. Redner
\newblock Scaling Theory of Fragmentation
\newblock {\em Phys. Rev. Lett.}, {\bf 60}, 2450--2453,  (1988)

\bibitem{CE}
E. Cort\'es and M. Escobedo.
\newblock On a system of equations for the normal fluid--condensate interaction in a Bose gas.
\newblock {\em J. Funct. Anal.}, 
\newblock {\bf 278}  1083015, (2020)


\bibitem{D} S. Dyachenko, A. C. Newell, A. Pushkarev and V. E. Zakharov.
\newblock Optical turbulence: weak turbulence, condensates and collapsing filaments in the nonlinear Schr\"odinger equation, 
\newblock {\em Phys. D}
\newblock{\bf  57},  96-160, (1992)

\bibitem{Eckern}
U.~Eckern.
\newblock Relaxation processes in a condensed {B}ose gas.
 {\it  J. Low Temp. Phys.}, {\bf 54}, 333--359 (1984)

\bibitem{ESV}
M. Escobedo, S. Mischler, J.J.L. Vel\'azquez.On the Fundamental Solution of a Linearized Uehling–Uhlenbeck Equation,
{\it Arch. Rat. Mech. Anal.}
\newblock {\bf 186}, 309--349, (2007)


\bibitem{EV}
M. Escobedo, J.J.L. Vel\'azquez. On the Fundamental Solution of a Linearized Homogeneous Coagulation Equation,
 {\it Comm. in Maths. Phys.}, {\bf 297}, 759--816, (2010)

\bibitem{EV2}
M. Escobedo,  J.J.L. Vel\'azquez
\newblock Classical non-mass-preserving solutions of coagulation equations
\newblock {\em Ann. I. H. Poincar\'e--AN},  
\newblock {\bf 29},  589--635,  (2012)

\bibitem{EV3}
M. Escobedo, J.J.L. Vel\'azquez
On the theory of Weak Turbulence for the Nonlinear Schrödinger Equation,
\newblock {\em Memoirs of the AMS},  
\newblock {\bf 238}, 1124 (2015).


\bibitem{EPV}
M. Escobedo, F. Pezzotti, M. Valle,
\newblock Analytical approach to relaxation dynamics of condensed Bose gases.
\newblock {\em Annals of Physics}, 
\newblock {\bf 326},  808--827, (2011)

\bibitem {J}
C. Josserand, Y. Pomeau and S. Rica,
\newblock Self-similar Singularities in the Kinetics of Condensation
\newblock{\em J. Low. Temp. Phys.}, 
\newblock {\bf 145}, 231--265, (2006)

\bibitem{AV}
A.~H.~M. Kierkels and J.~J.~L. Vel{\'a}zquez.
\newblock On the transfer of energy towards infinity in the theory of weak
  turbulence for the nonlinear Schr{\"o}dinger equation.
\newblock {\em J. Stat. Phys.} 2015.


\bibitem{Kirkpatrick}
T.~R. Kirkpatrick and J.~R. Dorfman,
\newblock Transport in a dilute but condensed nonideal {B}ose gas: Kinetic equations.
{\it  J. Low. Temp. Phys.}
\newblock {\bf 58} (3), 301--331, (1985)


\bibitem{ML} O. P. Misra and J. L. Lavoine,
\newblock {\em Transform Analysis of Generalized Functions}.
\newblock North-Holland Mathematics Studies. Elsevier Science, 1986.

\bibitem{M}  N. I. Muskhelishvil,
\newblock {\em Singular Integral Equations: Boundary problems of functions theory and their applications to mathematical physics}.
\newblock  Springer Netherlands, 1977

\bibitem{P} I. Podlubny
\newblock {\em Fractional Differential Equations}.
\newblock Academic Press, 1999.


\bibitem{Sa} S. Samko, A. A. Kilbas, O. I. Marichev,
\newblock {\em Fractional Integrals and  Derivatives}.
\newblock Gordon and Breach Science Publishers, 1993.

\bibitem{ST} A. Soffer, M.-B. Tran
\newblock On the dynamics of finite temperature trapped Bose gases
\newblock {\it Adv.  Maths.}
\newblock {\bf 325}, 533-607, (2018)


\bibitem{S}   H. Spohn, 
\newblock Kinetics of Bose Einstein Condensation, 
\newblock {\it Phys. D}
\newblock {\bf  239}, 627–634,  (2010) 


\bibitem{Za} E. Zaremba, T. Nikuni, A. Griffin, 
\newblock Dynamics of trapped Bose gases at finite temperatures, 
\newblock {\it J. Low Temp. Phys}
\newblock  {\bf 116}, 277--345, (1999)


\end{thebibliography}
 \end{document}